\documentclass[a4paper,oneside,10pt]{article}%
\usepackage{amsmath}
\usepackage{amsfonts}
\usepackage{amssymb}
\usepackage{graphicx}
\usepackage{color}
\usepackage[square,numbers,sort&compress]{natbib}%
\usepackage{bbm}
\usepackage{amsfonts} 
\usepackage{amssymb}
\usepackage{array}
\usepackage{bm}
\usepackage{amsmath}
\usepackage{mathtools}
\usepackage{indentfirst} 
\usepackage{hyperref}
\usepackage[capitalise]{cleveref}
\hypersetup{hypertex=true,
	colorlinks=true,
	linkcolor=blue,
	anchorcolor=blue,
	citecolor=blue}
\providecommand{\U}[1]{\protect \rule{.1in}{.1in}}

\newtheorem{theorem}{Theorem}[section]

\newtheorem{definition}[theorem]{Definition}

\newtheorem{lemma}[theorem]{Lemma}

\newtheorem{proposition}[theorem]{Proposition}
\newtheorem{remark}[theorem]{Remark}

\newenvironment{proof}[1][Proof]{\noindent \textbf{#1.} }{\  $\Box$}
\newenvironment{proof of Theorem 4.7}[1][Proof of Theorem 4.7]{\noindent \textbf{#1.} }{\  $\Box$}
\numberwithin{equation}{section}

\begin{document}
	
	\title{An averaging principle for nonlinear parabolic PDEs via reflected FBSDEs driven by $G$-Brownian motion }
	\author{Mengyao Hou \thanks{Shandong University, Jinan, Shandong 250100, PR China. houmengyao05@163.com.
		    }
	}
	\maketitle
	
	\textbf{Abstract}. In this paper, we are concerned with the averaging problem for a class of forward-backward stochastic differential equations with reflection driven by $G$-Brownian motion (reflected $G$-FBSDEs), which corresponds to the singular perturbation problem of a kind of fully nonlinear partial differential equations (PDEs) with a lower obstacle. The reflection keeps the solution above a given stochastic process. By the use of the nonlinear stochastic techniques and viscosity solution methods, we prove that the limit distribution of solution is the unique viscosity solution of an obstacle problem for a fully
	nonlinear parabolic PDEs.
	
	
	
	\addcontentsline{toc}{section}{\hspace*{1.8em}Abstract}
	
	\section{Introduction}
	
	\noindent The averaging principle, initiated by Khasminskii in theseminal work \cite{K}, is a very efficient and important tool in the study of stochastic differential equations (SDEs) for modelling problems arising in many practical research areas. In fact, averaging principle is indeed an effective method for studying dynamical systems with highly oscillating coefficients. Under certain suitable conditions, the highly oscillating components can be ``averaged out'' to establish an averaged system. Given that the averaged system governs the evolution of the original system over long time scales, we can analyze complex stochastic equations with related averaging stochastic equations, which paves a convenient and easy way to conclude some significant properties. Since then, the averaging principle for diffusion processes have been studied with great interest, it also provides a powerful tool for the research of singular perturbation problems for a kind of fully nonlinear PDEs. In particular, Pardoux and Veretennikov \cite{PV} adopted the weak convergence mehods to relax the assumptions  built on the coefficients. For more research on this aspect, we refer the reader to \cite{BEP,CF,CL,FW,HW,WR} and the references therein.\par
Motivated by the need to address financial problems under volatility uncertainty and the associated fully nonlinear PDEs, Peng \cite{Peng1,Peng2,Peng3} introduced a novel nonlinear expectation theory known as the $G$-expectation theory. In this framework, a new type of Brownian motion and the corresponding stochastic calculus of It\^o's type were constructed. Based on the $G$-expectation theory, Hu et al. \cite{HJP1} formulated a new kind of backward stochastic differential equations with a nonincreasing $G$-martingale ($G$-BSDEs), and established the well-posedness of $G$-BSDEs. And also, the comparison principle, the Feynman-Kac formula and Girsanov transformation were also be concluded in \cite{HJP2}.\par

Especially, \cite{HJW} studied the averaging problem for the following forward-backward stochastic differential equations driven by $G$-Brownian motion ($G$-FBSDEs) with rapidly oscillating coefficients:
	\begin{equation}\label{(1.1)}
		  \left\{
		\begin{aligned}
			X^{\varepsilon,t,x}_{s}=&\,x+\displaystyle\int_t^{s} b\Big(\frac{r}{\varepsilon},X^{\varepsilon,t,x}_{r}\Big)dr+\sum\limits_{i,j=1}^d\int_t^{s} h_{ij}\Big(\frac{r}{\varepsilon},X^{\varepsilon,t,x}_{r}\Big)d\langle B^{i},B^{j}\rangle_{r}+\int_t^{s} \sigma\Big(\frac{r}{\varepsilon},X^{\varepsilon,t,x}_{r}\Big)dB_{r},\\
			Y^{\varepsilon,t,x}_{s}=&\,\varphi(X^{\varepsilon,t,x}_{T})+\int_s^{T} f\Big(\frac{r}{\varepsilon},X^{\varepsilon,t,x}_{r},Y^{\varepsilon,t,x}_{r},Z^{\varepsilon,t,x}_{r}\Big)dr-\int_s^{T}Z^{\varepsilon,t,x}_{r} dB_{r}-\big(K^{\varepsilon,t,x}_{T}-K^{\varepsilon,t,x}_{s}\big)\\
			&+\sum_{i,j=1}^d\int_s^{T} g_{ij}\Big(\frac{r}{\varepsilon},X^{\varepsilon,t,x}_{r},Y^{\varepsilon,t,x}_{r},Z^{\varepsilon,t,x}_{r}\Big)d\langle B^{i},B^{j}\rangle_{r},
		\end{aligned} \right.
	\end{equation}
where $B=(B^1,\ldots,B^d\mathbf{)}^\top$ is a $d$-dimensional $G$-Brownian motion defined in the $G$-expectation space $(\Omega_{T},$ $L^{1}_{G}(\Omega_{T}),\hat{\mathbb{E}})$. They prove that the limit distribution of the solution $Y^{\varepsilon,t,x}$ is the unique viscosity solution to a fully nonlinear PDE. More detailed results can see \cite{HJW}.\par

Owing to the importance in BSDE theory and in applications, many scholars tried to explore more relevant conclusions and relax the conditions on the coefficients. El Karoui et al. \cite{EKPP} introduced the problem of reflected backward stochastic differential equation (reflected BSDE), which means that the solution to this kind of equation is required to be above a certain given continuous boundary process, called the obstacle. In \cite{CK} and \cite{HLJP}, They proved the existence and uniqueness of the equation when there are two reflecting obstacles. In this setting, the solution is constrained to remain between two specified continuous processes, known as the lower and upper obstacles. \par

In recent years, Li et al. \cite{LPS} introduced reflected $G$-BSDEs with a lower obstacle. Existence was establi-\\ shed through approximation via penalization, while uniqueness was derived from a prior estimates. They also built the relation between reflected $G$-BSDEs and the corresponding obstacle problems for fully nonlinear parabolic PDEs by nonlinear Feynman-Kac formula under the $G$-expectation framework. For further insights, readers may refer to \cite{LP,LS,L}. Among them, \cite{LP} constructed reflected $G$-BSDE with an upper obstacle. The study of reflected $G$-BSDEs with two obstacles is undertaken by \cite{LS}. A new kind of approximate Skorohod condition is proposed to derive the uniqueness and existence of the solutions when the upper obstacle is a generalized $G$-It\^{o} process. \cite{L} considered the reflected backward stochastic differential equations driven by $G$-Brownian motion with time-varying Lipschitz coefficients.  \par

In this paper, we mainly study the following dynamical systems with the lower obstacle: for each $(t,x) \in [0,T]\times \mathbb{R}^n$, $\varepsilon \in (0,1)$,
		\begin{equation}\label{(1.2)}
		\left\{
		\begin{aligned} 
			X^{\varepsilon,t,x}_{s}=&\,x+\displaystyle\int_t^{s} b\Big(\frac{r}{\varepsilon},X^{\varepsilon,t,x}_{r}\Big)dr+\sum\limits_{i,j=1}^d\int_t^{s} h_{ij}\Big(\frac{r}{\varepsilon},X^{\varepsilon,t,x}_{r}\Big)d\langle B^{i},B^{j}\rangle_{r}+\int_t^{s} \sigma\Big(\frac{r}{\varepsilon},X^{\varepsilon,t,x}_{r}\Big)dB_{r},\\
			Y^{\varepsilon,t,x}_{s}=&\,\varphi(X^{\varepsilon,t,x}_{T})+\int_s^{T} f\Big(\frac{r}{\varepsilon},X^{\varepsilon,t,x}_{r},Y^{\varepsilon,t,x}_{r},Z^{\varepsilon,t,x}_{r}\Big)dr-\int_s^{T}Z^{\varepsilon,t,x}_{r} dB_{r}+\big(A^{\varepsilon,t,x}_{T}-A^{\varepsilon,t,x}_{s}\big)\\
			&+\sum_{i,j=1}^d\int_s^{T} g_{ij}\Big(\frac{r}{\varepsilon},X^{\varepsilon,t,x}_{r},Y^{\varepsilon,t,x}_{r},Z^{\varepsilon,t,x}_{r}\Big)d\langle B^{i},B^{j}\rangle_{r},
		\end{aligned}  \right.
	\end{equation}
where $B=(B^1,\ldots,B^d\mathbf{)}^\top$ is a $d$-dimensional $G$-Brownian motion defined in the $G$-expectation space $(\Omega_{T},L^{1}_{G}(\Omega_{T}),\hat{\mathbb{E}})$. For the coefficients of the equation \eqref{(1.2)}, $b,h_{ij}=h_{ji}:\mathbb{R}^{+}\times\mathbb{R}^n \rightarrow \mathbb{R}^n$, $\sigma:\mathbb{R}^{+} \times\mathbb{R}^n \rightarrow \mathbb{R}^{n \times d }$, $f,g_{ij}=g_{ji}:\mathbb{R}^{+} \times \mathbb{R}^n \times \mathbb{R} \times \mathbb{R}^{1 \times d} \rightarrow \mathbb{R}$, $\varphi:\mathbb{R}^n \rightarrow \mathbb{R}$, $S: \mathbb{R}^{+} \times \mathbb{R}^n \rightarrow \mathbb{R}$ are continuous deterministic functions. We denote by \eqref{(1.2)} the reflected $G$-FBSDE with a lower obstacle $L^{\varepsilon,t,x}_{s}=$ $S(s,X^{\varepsilon,t,x}_{s})$. Our goal is to study the asymptotic property of $Y^{\varepsilon,t,x}$ as $\varepsilon \rightarrow 0$.\par

The difficulties of dealing with this kind of equation mainly originate from the following two aspects. On the one hand, for the classical cases, we can adopt the martingale approach to deal with the limit process of the solution sequence. But the nonlinear martingale method is still imperfect, the classical martingale approach breaks down in the nonlinear $G$-expectation framework. On the other hand, owing to the restriction of the lower obstacle, we have no way to verify that the terminal value is indeed not less than the value of the obstacle process. To this end, we draw support from the ideas on the asymptotic behavior of $Y^{\varepsilon,t,x}$ in \cite{HJW} and Feynman-Kac formula in \cite{LPS} and construct two auxiliary $G$-BSDEs to overcome these problems.\par

In this work, we firstly introduce some priori estimates with respect to the equation \eqref{(1.2)} under the regularity assumptions. With the help of the Feynman-Kac formula for the reflected $G$-FBSDE (see \cite{LPS}), it allows us to conclude that the unique viscosity solution $(u^{\varepsilon})_{\varepsilon\geq 0}$ of the PDE is uniformly bounded and equi-continuous. According to the Arzel\`a-Ascoli theorem, we obtain the convergent subsequence of the viscosity solutions. Note that, compared with the result in \cite{HJW}, we consider in this paper a different averaging problem\\ for a class of reflected $G$-FBSDEs. Due to the restriction of the given lower obstacle, the method of conducting asymptotic analysis by using the stochastic backward semigroup approach (see Peng \cite{Peng4}) can not be directly applied in certain cases. To overcome it, we adopt the approximation method via penalization to construct two auxiliary stochastic backward semigroups based on the general $G$-BSDE instead of reflected $G$-BSDE, which can transform the reflected $G$-BSDE in \eqref{(1.2)} into a general $G$-BSDE. In fact, it eliminates the influence of the lower obstacle on terminal values to a certain extent. Thus, we can still apply original method used in \cite{HJW} for the auxiliary $G$-BSDE to fulfill the asymptotic analysis.\par

Throughout this paper, the letter $C$ will denote a positive constant, with or without subscript, its value may change in different occasions. We will write the dependence of the constant on parameters explicitly if it is essential.\par

The rest of the paper is organized as follows. In Section \ref{section 2}, we present some basic notations and results with respect to the $G$-expectation theory. In Section \ref{section 3}, we introduce the averaging problem and give some priori estimates for the solution of reflected $G$-FBSDE. In Section \ref{section 4}, we state our main results and establish an approximation theorem as an averaging principle.
	
	\section{Preliminaries}\label{section 2}
	
	\noindent We recall some basic results and notions of $G$-expectation, which are needed in the sequel. More relevant details can be found in \cite{Peng1,Peng2,Peng3}.\par 
In this paper, let $\Omega= C^{d}_{0}(\mathbb{R}^{+})$ be the space of all $\mathbb{R}^{d}$-valued continuous paths $(\omega_t)_{t\geq 0}$ with $\omega_0=0$, equipped with the distance
\begin{align}
		\rho\big(\omega^{(1)},\omega^{(2)}\big)\coloneqq \sum_{i=1}^{\infty}2^{-i}\Big[\Big(\max_{t\in[0,i]}\big\vert \omega^{(1)}_t-\omega^{(2)}_t \big\vert\Big) \wedge 1\Big], \quad \mathrm{for}\  \omega^{(1)},\omega^{(2)}\in \Omega.\notag
\end{align}
Let $B_t(\omega) = \omega_t$ be the canonical process. For each $t \in [0,\infty)$, define
$\Omega_{t}=\{\omega_{\cdot \vee t}: \omega \in \Omega\}$. Set $Lip(\Omega)=\bigcup_{t \geq 0}Lip(\Omega_t)$ and
\begin{equation}
	Lip(\Omega_{t})=\{\varphi(B_{t_1 \wedge t},\ldots,B_{t_n \wedge t}): n\geq 1,\,t_1,\ldots,t_n\in\left[0,\infty\right),\,\varphi \in C_{l,Lip}(\mathbb{R}^{d\times n})\},\notag
\end{equation}
where $C_{l,Lip}(\mathbb{R}^{d\times n})$ denotes the space of local Lipschitz continuous functions on $\mathbb{R}^{d\times n}$.\par
Let $G(\cdot):\mathbb{S}(d) \rightarrow \mathbb{R}$ be a given monotone and sublinear function. By Theorem 1.2.1 in Peng \cite{Peng3}, there exists a bounded, convex and closed subset $\Sigma \subset \mathbb{S}_{+}(d)$, such that
\begin{align}
	G(A)=\frac{1}{2}\sup_{B \in \Sigma}(A,B),\quad A\in \mathbb{S}(d).\notag
\end{align}\par
For this, Peng \cite{Peng3} construct a sequence of
$d$-dimensional random vectors $(\xi_i)^{\infty}_{i
=1}$ on a sublinear expectation space $(\widetilde{\Omega},\widetilde{\mathcal{H}}, \widetilde{\mathbb{E}})$ such that $\xi_i$ is $G$-normally distributed and $\xi_{i+1}$ is independent from $(\xi_1,\ldots, \xi_i)$ for
each $i = 1, 2,\cdots $. \par
Now the definition of sublinear expectation $\hat{\mathbb{E}}$ defined on $Lip(\Omega)$ will be presented via the following proce- dure:
for each $X \in Lip(\Omega)$ with
\begin{align}
X=\varphi(B_{t_1}-B_{t_0},B_{t_2}-B_{t_1},\ldots,B_{t_n}-B_{t_{n-1}})\notag
\end{align}
for some $\varphi \in C_{l,Lip}(\mathbb{R}^{d \times n})$ and $0=t_0 \,\textless\,t_1\,\textless\,\cdots\,\textless\, t_n \,\textless\, \infty $, we set
\begin{align}
	\hat{\mathbb{E}}[\varphi(B_{t_1}-B_{t_0},B_{t_2}-B_{t_1},\ldots,B_{t_n}-B_{t_{n-1}})]
	\coloneqq \widetilde{\mathbb{E}}[\varphi(\sqrt{t_1-t_0}\xi_1,\ldots,\sqrt{t_{n}-t_{n-1}}\xi_{n})].\notag
\end{align}
The related conditional expectation of $X= \varphi(B_{t_1},B_{t_2}-B_{t_1},\ldots,B_{t_n}-B_{t_{n-1}})$ under $\Omega_{t_j}$ is defined by
\begin{equation}
	\hat{\mathbb{E}}_{t_j}[X]
	=
	\hat{\mathbb{E}}_{t_j}[\varphi(B_{t_1},B_{t_2}-B_{t_1},\ldots,B_{t_n}-B_{t_{n-1}})]
	\coloneqq \psi(B_{t_1},B_{t_2}-B_{t_1},\ldots,B_{t_j}-B_{t_{j-1}}),\notag
\end{equation}
where
\begin{align}
	 \psi(x_1,\ldots,x_j)=\widetilde{\mathbb{E}}[\varphi(x_1,\ldots,x_j,\sqrt{t_{j+1}-t_{j}}\xi_{j+1},\ldots,\sqrt{t_{n}-t_{n-1}}\xi_{n})].\notag
\end{align}
The corresponding canonical process $(B_{t})_{t \geq 0}$ on the sublinear expectation space $(\Omega,Lip(\Omega),\hat{\mathbb{E}})$ is called a $G$-Brownian motion.\par
It is easy to check that $(\hat{\mathbb{E}}_{t})_{t \geq 0}$ satisfies the following properties.
\begin{proposition}\label{proposition 2.1}
	For each $X,Y \in Lip(\Omega)$ and $ t \geq 0\mathrm{,}$ we have\\
	${}\hspace{0.62em}$$\bm{(\mathbf{i})}\ $Monotonicity: If $X \geq Y\mathrm{,}$ then $\hat{\mathbb{E}}_{t}[X] \geq \hat{\mathbb{E}}_{t}[Y]$$\mathrm{;}$\\
	${}\hspace{0.31em}$$\bm{(\mathbf{ii})}\ $Constant preserving: $\hat{\mathbb{E}}_{t}[X] =X\mathrm{,}$ for $X \in Lip(\Omega_{t})$$\mathrm{;}$\\
	$\bm{(\mathbf{iii})}\ $Sub-additivity: $\hat{\mathbb{E}}_{t}[X+Y] \leq \hat{\mathbb{E}}_{t}[X]+ \hat{\mathbb{E}}_{t}[Y]$$\mathrm{;}$\\
	$\bm{(\mathbf{iv})}\ $Positive homogeneity: $\hat{\mathbb{E}}_{t}[XY] \leq X^+\hat{\mathbb{E}}_{t}[Y]+ X^-\hat{\mathbb{E}}_{t}[-Y]\mathrm{,}$ for $X \in Lip(\Omega_{t})$$\mathrm{;}$\\
	${}\hspace{0.33em}$$\bm{(\mathbf{v})}\ $Consistency: $\hat{\mathbb{E}}_{s}[\hat{\mathbb{E}}_{t}[X]]=\hat{\mathbb{E}}_{s \wedge t}[X]\mathrm{,}$ in particular$\mathrm{,}$ $\hat{\mathbb{E}}[\hat{\mathbb{E}}_{t}[X]]=\hat{\mathbb{E}}[X]$.
\end{proposition}

For $p \geq 1$, denote by
	\begin{equation}
		L^{p}_{G}(\Omega)\coloneqq \big\{\mathrm{the}\ \mathrm{completion}\ \mathrm{of}\ \mathrm{the}\ \mathrm{space}\ Lip(\Omega)\ \mathrm{under}\ \mathrm{the}\ \mathrm{norm}\ \Vert X \Vert_{L^p_G}\coloneqq(\hat{\mathbb{E}}[\vert X\vert^{p}])^{\frac{1}{p}}\big\}.\notag
	\end{equation}
Similarly, we can define $L^{p}_{G}(\Omega_{T})$ for any $T \in \left[0,\infty \right)$, and let $S^{0}_{G}(0,T)=\{h(t,B_{t_1 \wedge t},\ldots,B_{t_n \wedge t}): t_1,\ldots,t_n\in [0,T],h\in C_{b,Lip}(\mathbb{R}\times \mathbb{R}^{d \times n})\}$. For each given $p \geq 1 $ and $\eta \in S^{0}_{G}(0,T)$, define $\Vert \eta \Vert_{S^{p}_{G}}=(\hat{\mathbb{E}}[\sup_{t \in [0,T]}\vert \eta_{t}\vert^{p}])^{\frac{1}{p}}$ and denote by $S^{p}_{G}(0,T)$ the completion of $S^{0}_{G}(0,T)$ under $\Vert \cdot \Vert_{S^{p}_{G}}$. In this paper, we suppose that $G$ is non- degenerate, i.e., there exists two constants $0 \,\textless\, \underline{\sigma}^{2} \leq \bar{\sigma}^{2} \,\textless\, \infty$, such that
\begin{equation}
	\frac{1}{2} \underline{\sigma}^{2}tr[A-B]
	\leq
    G(A)-G(B) 
    \leq 
    \frac{1}{2} \bar{\sigma}^{2}tr[A-B]\ \mathrm{for}\ A\geq B.\notag
\end{equation}

\begin{definition}\label{definition 2.2}
\rm{A} process $\{M_{t}\}_{t \geq 0}$ is called a $G$-martingale if for each $t \geq 0$, $M_{t} \in L^1_{G}(\Omega_{t})$ and 
		\begin{equation}
		\hat{\mathbb{E}}_{s}[M_t]=M_{s},\quad\forall 0 \leq s \leq t.\notag
		\end{equation}
	\end{definition}

\begin{definition}
\rm{L}et $M^{0}_{G}(0,T)$ be the collection of processes in the following form: for a given partition $\{t_0,t_1,\ldots,t_N\}$ of $[0,T]$,
\begin{align}
	\eta_{t}(\omega)=\sum\limits_{n=0}^{N-1} \xi_{n}(\omega)\mathbb{I}_{\left[t_n,t_{n+1}\right)}(t),\notag
\end{align}
where $\xi_{n} \in Lip(\Omega_{t_n})$, $n=0,1,\ldots,N-1$. For each given $p \geq 1$ and $\eta \in M^{0}_{G}(0,T)$, define $\Vert \eta \Vert_{M^{p}_{G}}=(\hat{\mathbb{E}}[\int\nolimits_{0}^{T}\vert \eta_s \vert^{p} ds])^{\frac{1}{p}}$
and denote by $M^{p}_{G}(0,T)$ the completion of $M^{0}_{G}(0,T)$ under $\Vert \cdot \Vert_{M^{p}_{G}}$.
\end{definition}\par

For each $1 \leq i,j \leq d$, we denote by $(\langle B\rangle_{t})_{ij}\coloneqq \langle B^{i},B^{j} \rangle_{t}$ the mutual variation process. Then for two random processes $\xi \in M^{1}_{G}(0,T) ,\eta \in M^{2}_{G}(0,T)$, the $G$-It\^o integrals $\{\int\nolimits_{0}^{t} \xi_{s}d\langle B^{i},B^{j} \rangle_{s}\}_{0 \leq t \leq T}$ and $\{\int\nolimits_{0}^{t} \eta_{s}dB^{i}_{s}\}_{0 \leq t \leq T}$ are well defined (see Peng \cite{Peng2,Peng3}). In particular, we present some properties as follows, which will be used in this paper.

\begin{proposition}\label{proposition 2.4} \rm{(\cite{Peng2,Peng3})} $Assume$ $that$ $\xi \in M^{p}_{G}(0,T;\mathbb{R}^{d})$ $and$ $\eta \in M^{p}_{G}(0,T;\mathbb{S}(d))$. $Then$, $for$ $each$ $p \geq 2$, $we$ $have$\vspace{1.5pt}\\
	${}\hspace{0.62em}$$\bm{(\mathbf{i})}\  \hat{\mathbb{E}}\Big[\int\nolimits_{0}^{T} \xi_{t}dB_{t}\Big]=0$$\mathrm{;}$\vspace{1.5pt}\\
	${}\hspace{0.31em}$$\bm{(\mathbf{ii})}\  \hat{\mathbb{E}}\bigg[\sup\limits_{t \in [0,T]}\big\vert \int\nolimits_{0}^{t} \xi_{s}dB_{s}\big\vert^{p}\bigg] \leq C(p)\hat{\mathbb{E}}\bigg[\Big(\int\nolimits_{0}^{T} \vert \xi_{t}\vert^{2}dt\Big)^{\frac{p}{2}}\bigg]$$\mathrm{;}$\vspace{1.5pt}\\
	$\bm{(\mathbf{iii})}\  \hat{\mathbb{E}}\bigg[\Big\vert \int\nolimits_{0}^{T} \eta_{t}d\langle B\rangle_{t}\Big\vert^{p}\bigg] \leq C(p,T)\hat{\mathbb{E}}\bigg[\int\nolimits_{0}^{T} \vert \eta_{t}\vert^{p}dt\bigg]$.
\end{proposition}

	\section{Prior estimates of $G$-BSDEs}\label{section 3}
	
	\noindent In this section, we present some assumptions and establish an averaging problem under the $G$-expectation framework. We proceed to derive a series of relevant estimates regarding the solution processes perturbed by a small parameter $\varepsilon \in (0,1)$ throughout the rest of this work, which essentially plays an important role in deducing the asymptotic behavior of the reflected $G$-FBSDE with a perturbation parameter $\varepsilon$. \par
For each fixed $T\geq 0$ and $(t,\xi)\in [0,T]\times\bigcap_{p\textgreater2}L^{p}_{G}(\Omega_{t};\mathbb{R}^{n})$, we consider the following dynamical systems on finite time horizon $[t,T]$ with a small parameter $\varepsilon\in (0,1)$:\\
	\begin{equation}\label{(3.1)}
	\left\{
	\begin{aligned} 
		\vspace{-0.25pt}
		X^{\varepsilon,t,\xi}_{s}=&\,\xi+\displaystyle\int_t^{s} b\Big(\frac{r}{\varepsilon},X^{\varepsilon,t,\xi}_{r}\Big)dr+\sum\limits_{i,j=1}^d\int_t^{s} h_{ij}\Big(\frac{r}{\varepsilon},X^{\varepsilon,t,\xi}_{r}\Big)d\langle B^{i},B^{j}\rangle_{r}+\int_t^{s} \sigma\Big(\frac{r}{\varepsilon},X^{\varepsilon,t,\xi}_{r}\Big)dB_{r},\\
		Y^{\varepsilon,t,\xi}_{s}=&\,\varphi(X^{\varepsilon,t,\xi}_{T})+\int_s^{T} f\Big(\frac{r}{\varepsilon},X^{\varepsilon,t,\xi}_{r},Y^{\varepsilon,t,\xi}_{r},Z^{\varepsilon,t,\xi}_{r}\Big)dr-\int_s^{T}Z^{\varepsilon,t,\xi}_{r} dB_{r}+\big(A^{\varepsilon,t,\xi}_{T}-A^{\varepsilon,t,\xi}_{s}\big)\\
		&+\sum_{i,j=1}^d\int_s^{T} g_{ij}\Big(\frac{r}{\varepsilon},X^{\varepsilon,t,\xi}_{r},Y^{\varepsilon,t,\xi}_{r},Z^{\varepsilon,t,\xi}_{r}\Big)d\langle B^{i},B^{j}\rangle_{r},
	\end{aligned}  \right. \vspace{5pt}
\end{equation}
where $B=(B^1,\ldots,B^d\mathbf{)}^\top$ is a $d$-dimensional $G$-Brownian motion defined in the $G$-expectation space $(\Omega_{T},$ $L^{1}_{G}(\Omega_{T}),\hat{\mathbb{E}})$, $\langle B^{i},B^{j}\rangle$ is the corresponding mutual variation process, $b,h_{ij}=h_{ji}:\mathbb{R}^{+}\times\mathbb{R}^n \rightarrow \mathbb{R}^n$, $\sigma:\mathbb{R}^{+} \times\mathbb{R}^n $ $\rightarrow \mathbb{R}^{n \times d }$, $f,g_{ij}=g_{ji}:\mathbb{R}^{+} \times \mathbb{R}^n \times \mathbb{R} \times \mathbb{R}^{1 \times d} \rightarrow \mathbb{R}$, $\varphi:\mathbb{R}^n \rightarrow \mathbb{R}$, $S: \mathbb{R}^{+} \times \mathbb{R}^n \rightarrow \mathbb{R}$ are continuous deterministic functions. $A^{\varepsilon,t,\xi}$ is continuous nondecreasing process with $A^{\varepsilon,t,\xi} \in S^{\alpha}_{G}(t,T)$. We denote by $\{L^{\varepsilon,t,\xi}_{s}\}_{s \in [t,T]}$ the lower obstacle of the $G$-BSDE given above, satisfying $L^{\varepsilon,t,\xi}_{s}=$ $S(s,X^{\varepsilon,t,\xi}_{s})$.\par
We need the following assumptions:\\
\textbf{($\mathbf{H1}$)} There exists a constant $L \,\textgreater\, 0$ such that, for any $s \in \mathbb{R}^{+}$, $x$, $x' \in \mathbb{R}^n$,
\begin{equation}
	\vert \vartheta(s,x)-\vartheta(s,x')\vert \leq L \vert x-x'\vert \ \mathrm{and}\ \vert \vartheta(s,0)\vert \leq L,\ \mathrm{for}\ \vartheta=b,\hspace{0.05em}h_{ij},\hspace{0.05em}\sigma. \notag
\end{equation}
\textbf{($\mathbf{H2}$)} There exist a constant $L \,\textgreater\, 0$ and a positive integer $m$ such that, for any $s\in \mathbb{R}^{+}$, $x$, $x' \in \mathbb{R}^n$, $y,y'\in \mathbb{R}$,\par

\setlength{\parindent}{2.85em}
 and $z$, $z' \in \mathbb{R}^{1 \times d}$,
\begin{equation}
	\vert \vartheta(s,x,y,z)-\vartheta(s,x',y',z')\vert \leq L((1+\vert x\vert^{m}+\vert x'\vert^{m})\vert x-x'\vert\,+\,\vert y-y'\vert\,+\,\vert z-z'\vert) \notag
\end{equation}\par

\setlength{\parindent}{2.85em}
and
\begin{align}
\vert \vartheta(s,0,0,0) \leq L,\ \mathrm{for}\ \vartheta=f,\hspace{0.05em} g_{ij}\ \mathrm{and}\ \varphi. \notag
\end{align}
\textbf{($\mathbf{H3}$)} $S$ is Lipschitz continuous function w.r.t. $x$ and bounded from above, $S(T,x) \leq \varphi(x)$ for any $x \in \mathbb{R}^{n}$.\\
\textbf{($\mathbf{H4}$)} $S\in C^{1,2}_{Lip}([0,T]\times \mathbb{R}^{n} )$ and $S(T,x) \leq \varphi(x)$ for any $x \in \mathbb{R}^{n}$, where $C^{1,2}_{Lip }([0,T]\times \mathbb{R}^{n} )$ is denoted by the\par 

\setlength{\parindent}{2.75em}
space of all functions in $C^{1,2}([0,T]\times \mathbb{R}^{n} )$ whose partial derivatives of order less than or equal to 2 and\par

\setlength{\parindent}{2.75em}
itself are Lipschitz continuous functions with respect to $x$.\par

\setlength{\parindent}{1.5em}
Under assumption (H1), $G$-SDE in \eqref{(3.1)} exists the unique solution $X^{\varepsilon,t,\xi}_{\cdot} \in M^{2}_{G}(t,T;\mathbb{R}^{n})$. Furthermore, the following estimates and more relevant results can be found in Chap. V of Peng \cite{Peng3}.

\begin{lemma}\label{lemma 3.1}
	Suppose that $\mathrm{(H1)}$ holds and $\xi\mathrm{,}$ ${\xi}'\in L^{p}_{G}(\Omega_{t};\mathbb{R}^{n})$ for any $p \geq 2$. Then$\mathrm{,}$ for each $0 \leq t \leq s \leq T\mathrm{,}$ we have$\vspace{3.5pt}$
	\setlength{\parindent}{0.66em}

	$\bm{(\mathbf{i})}$ $\hat{\mathbb{E}}_{t}\bigg[\sup\limits_{r \in [t,T]} \big\vert X^{\varepsilon,t,\xi}_{r}\big\vert^p\bigg] \leq C(L,p,T)(1+\vert \xi \vert^{p})\rm{;}$\par
	\vspace{2pt}
	\setlength{\parindent}{0.33em}
	$\bm{(\mathbf{ii})}$ $\hat{\mathbb{E}}_{t}\bigg[\sup\limits_{r \in [t,s]} \big\vert X^{\varepsilon,t,\xi}_{r}-\xi \big\vert^p\bigg] \leq C(L,p,T)(1+\vert \xi \vert^{p})\vert t-s \vert^{\frac{p}{2}} \rm{;}$\vspace{2pt}\\
	$\bm{(\mathbf{iii})}$ $\hat{\mathbb{E}}_{t}\bigg[\sup\limits_{r \in [t,T]} \big\vert X^{\varepsilon,t,\xi}_{r}-X^{\varepsilon,t,\xi'}_{r} \big\vert^p\bigg] \leq C(L,p,T)\vert \xi-\xi' \vert^{p}$.
\end{lemma}\par
For simplicity, let $c$ be a constant that constrains the function from above, which is mentioned in assumption (H3). We denote by $\mathcal{S}^{2}_{G}(0,T)$ the collection of processes $(Y,Z,A)$ such that $Y \in S^{2}_{G}(0,T),\,Z$ $\in H^{2}_{G}(0,T)$, $A$ is a continuous nondecreasing process with $A_{0}=0$ and $A \in S^{2}_{G}(0,T)$. It follows from Lemma \ref{lemma 5.1} in appendix that, for any $(t,\xi)\in [0,T]\times\bigcap_{p\textgreater2}L^{p}_{G}(\Omega_{t};\mathbb{R}^{n})$, there exists a unique triple $(Y^{\varepsilon,t,\xi},Z^{\varepsilon,t,\xi},A^{\varepsilon,t,\xi}) \in \mathcal{S}^{2}_{G}(t,T)$ under assumptions (H1)-(H3) or (H1)-(H2),\,(H4), which solves the reflected $G$-BSDE given in \eqref{(3.1)}, and we get$\vspace{3pt}$\\
${}\hspace{0.22em}$ (i) $Y^{\varepsilon,t,\xi}_{s}=\varphi(X^{\varepsilon,t,\xi}_{T})+\int_s^{T} f\big(\frac{r}{\varepsilon},X^{\varepsilon,t,\xi}_{r},Y^{\varepsilon,t,\xi}_{r},Z^{\varepsilon,t,\xi}_{r}\big)dr-\int_s^{T}Z^{\varepsilon,t,\xi}_{r} dB_{r}+(A^{\varepsilon,t,\xi}_{T}-A^{\varepsilon,t,\xi}_{s})\vspace{3pt}\\
{} \hspace{5.47em}+\sum\limits_{i,j=1}^d\int_s^{T} g_{ij}\big(\frac{r}{\varepsilon},X^{\varepsilon,t,\xi}_{r},Y^{\varepsilon,t,\xi}_{r},Z^{\varepsilon,t,\xi}_{r}\big)d\langle B^{i},B^{j}\rangle_{r},\ t \leq s \leq T$;\vspace{3pt}\\
${}\ $(ii) $Y^{\varepsilon,t,\xi}_{s} \geq S(s,X^{\varepsilon,t,\xi}_{s}),\ t \leq s \leq T$;\vspace{3pt}\\
(iii) $\{A^{\varepsilon,t,\xi}_{s}\}$ is nondecreasing and continuous, and $\{-\int_t^{s} \big(Y^{\varepsilon,t,\xi}_{r}-S(r,X^{\varepsilon,t,\xi}_{r})\big)dA^{\varepsilon,t,\xi}_{r},t \leq s \leq T \}$ is a nonincr-\\
${}\hspace{2em}$easing $G$-martingale.
\begin{lemma}\label{lemma 3.2}
	Suppose that $\mathrm{(H1)}$-$\mathrm{(H3)}$ hold and $\xi,\,\xi'\in L^{p}_{G}(\Omega_{t};\mathbb{R}^{n})$ for any $p \geq 2$. Then$\mathrm{,}$ for each $0 \leq t \leq s \leq T\mathrm{,}$ we have\par
	
	\setlength{\parindent}{0.66em}
	$\bm{(\mathbf{i})}$ $\big\vert Y^{\varepsilon,t,\xi}_{s}\big\vert \leq C(T,L,c,\underline{\sigma})\big(1+\big\vert X^{\varepsilon,t,\xi}_{s} \big\vert^{m+1}\big)\rm{;}$\par
	
	\setlength{\parindent}{0.33em}
	$\bm{(\mathbf{ii})}$ $\big\vert Y^{\varepsilon,t,\xi}_{s}-Y^{\varepsilon,t,\xi'}_{s}\big\vert^{2} \leq C(L,T,c,\underline{\sigma})\big(1+\big\vert X^{\varepsilon,t,\xi}_{s} \big\vert^{2m}+\big\vert X^{\varepsilon,t,\xi'}_{s} \big\vert^{2m}\big)\big\vert X^{\varepsilon,t,\xi}_{s}-X^{\varepsilon,t,\xi'}_{s} \big\vert^{2}\\
	{} \hspace{10.75em}+C(L,T,c,\underline{\sigma})\big(1+\big\vert X^{\varepsilon,t,\xi}_{s} \big\vert^{m+1}+\big\vert X^{\varepsilon,t,\xi'}_{s} \big\vert^{m+1}\big)\big\vert X^{\varepsilon,t,\xi}_{s}-X^{\varepsilon,t,\xi'}_{s} \big\vert\rm{;}$\\
	$\bm{(\mathbf{iii})}$ $\hat{\mathbb{E}}_{t}\Big[\int\nolimits_t^{T}\vert Z^{\varepsilon,t,\xi}_{r}\big\vert^2dr\Big] \leq C(L,T,c,\underline{\sigma})(1+\vert \xi \vert^{2m+2})$.
\end{lemma}
\begin{proof}
	Without loss of generality, we assume that $g_{ij}=0,\,i,\,j=1,\dots,d$ in the following proof.\\
	${}\hspace{1.5em}$(i) Applying Lemma \ref{lemma 5.2} in appendix, we obtain
	\begin{equation}\label{(3.2)}
		\big\vert Y^{\varepsilon,t,\xi}_{s}\big\vert^{2} \leq C(T,L,c,\underline{\sigma}) \hat{\mathbb{E}}_{s}\bigg[1+\big\vert \varphi(X^{\varepsilon,t,\xi}_{T}) \big\vert^{2}+\int\nolimits_{s}^{T}\Big\vert f\Big(\frac{r}{\varepsilon},X^{\varepsilon,t,\xi}_{r},0,0\Big) \Big\vert^{2} dr\bigg].
	\end{equation}
With the help of assumption (H2), we have
    \begin{align}\label{(3.3)}
    	\big\vert \varphi(X^{\varepsilon,t,\xi}_{T}) \big\vert 
    	\leq 
    	L\big(1+\big\vert X^{\varepsilon,t,\xi}_{T} \big\vert^{m}\big)\big\vert X^{\varepsilon,t,\xi}_{T} \big\vert+\big\vert \varphi(0) \big\vert 
    	\leq 
    	C(L)\big(1+\big\vert X^{\varepsilon,t,\xi}_{T} \big\vert^{m+1}\big),
     \end{align}
 and 
    \begin{align}
    	\Big\vert f\Big(\frac{r}{\varepsilon},X^{\varepsilon,t,\xi}_{r},0,0\Big) \Big\vert 
 	    \leq C(L)\big(1+\big\vert X^{\varepsilon,t,\xi}_{r} \big\vert^{m+1}\big).\label{(3.4)}
    \end{align}
Substituting \eqref{(3.3)} and \eqref{(3.4)} into \eqref{(3.2)}, we get
    \begin{equation}\label{(3.5)}
    	\begin{aligned}
    	\big\vert Y^{\varepsilon,t,\xi}_{s}\big\vert^{2} 
    	\leq
    	 &\,C(T,L,c,\underline{\sigma}) \hat{\mathbb{E}}_{s}\bigg[ \big(1+\big\vert X^{\varepsilon,t,\xi}_{T} \big\vert^{2m+2}\big) +\int\nolimits_{s}^{T}\big(1+\big\vert X^{\varepsilon,t,\xi}_{r} \big\vert^{2m+2}\big) dr\bigg]\\
    	\leq &\,C(T,L,c,\underline{\sigma})\bigg(1+\hat{\mathbb{E}}_{s}\bigg[\sup\limits_{r \in [s,T]}\big\vert X^{\varepsilon,t,\xi}_{r} \big\vert^{2m+2}\bigg]\bigg).
        \end{aligned} 
    \end{equation}
According to the uniqueness of solution to the $G$-SDE in \eqref{(3.1)}, it is not hard to  obtain $X^{\varepsilon,t,\xi}_{r}=X^{\varepsilon,s,X^{\varepsilon,t,\xi}_{s}}_{r}$ for any $r \in [s,T]$. In the view of this relation and Lemma \ref{lemma 3.2}\,(i), we derive that
     \begin{equation}\label{(3.6)}
	     \hat{\mathbb{E}}_{s}\bigg[\sup\limits_{r \in [s,T]}\big\vert X^{\varepsilon,t,\xi}_{r} \big\vert^{2m+2}\bigg]
	     =
	     \hat{\mathbb{E}}_{s}\bigg[\sup\limits_{r \in [s,T]}\big\vert X^{\varepsilon,s,X^{\varepsilon,t,\xi}_{s}}_{r} \big\vert^{2m+2}\bigg]
	     \leq 
	     C(L,T)(1+\big\vert X^{\varepsilon,t,\xi}_{s} \big\vert^{2m+2}).
     \end{equation}
Thus, from \eqref{(3.5)}-\eqref{(3.6)}, we get
     \begin{equation}
     	\big\vert Y^{\varepsilon,t,\xi}_{s}\big\vert 
     	\leq 
     	C(T,L,c,\underline{\sigma})\big(1+\big\vert X^{\varepsilon,t,\xi}_{s} \big\vert^{m+1}\big).\notag
     \end{equation}
Thus, we obtain the desired result.\\
 ${}\hspace{1.5em}$(ii)
By Lemma \ref{lemma 5.5} in appendix, we have
     \begin{equation}\label{(3.7)}
     	\begin{aligned}
     	\big\vert Y^{\varepsilon,t,\xi}_{s}-Y^{\varepsilon,t,\xi'}_{s}\big\vert^{2}
     	\leq
     	&\,C\hat{\mathbb{E}}_{s}\Big[\big\vert \varphi(X^{\varepsilon,t,\xi}_{T})-\varphi(X^{\varepsilon,t,\xi'}_{T})\big\vert^{2}\Big]\\
     	&+C\hat{\mathbb{E}}_{s}\bigg[\int\nolimits_{s}^{T} \Big\vert  f\Big(\frac{r}{\varepsilon},X^{\varepsilon,t,\xi}_{r},Y^{\varepsilon,t,\xi}_{r},Z^{\varepsilon,t,\xi}_{r}\Big)- f\Big(\frac{r}{\varepsilon},X^{\varepsilon,t,\xi'}_{r},Y^{\varepsilon,t,\xi}_{r},Z^{\varepsilon,t,\xi}_{r}\Big)\Big\vert^{2}dr\bigg] \\
     	&+C\bigg(\hat{\mathbb{E}}_{s}\bigg[\sup\limits_{r \in [s,T]}\big\vert S(r,X^{\varepsilon,t,\xi}_{r})-S(r,X^{\varepsilon,t,\xi'}_{r})\big\vert^{2}\bigg]\bigg)^{\frac{1}{2}}(\Psi_{s,T})^{\frac{1}{2}}\\
     	\coloneqq &\,I_{1}+I_{2}+I_{3}(I_{4}+I_{5})^{\frac{1}{2}},
     	\end{aligned}
     \end{equation}
 and
     \begin{align}
 		\Psi_{s,T}
 		=&
 		\,\hat{\mathbb{E}}_{s}\bigg[\sup\limits_{r \in [s,T]}\hat{\mathbb{E}}_{r}\bigg[1+\big\vert \varphi(X^{\varepsilon,t,\xi}_{T})\big\vert^{2} +\int\nolimits_{s}^{T} \Big\vert  f\Big(\frac{u}{\varepsilon},X^{\varepsilon,t,\xi}_{u},0,0\Big)\Big\vert^{2}du\bigg]\bigg]\notag\\
 		&+\hat{\mathbb{E}}_{s}\bigg[\sup\limits_{r \in [s,T]}\hat{\mathbb{E}}_{r}\bigg[1+\big\vert \varphi(X^{\varepsilon,t,\xi'}_{T})\big\vert^{2} +\int\nolimits_{s}^{T} \Big\vert  f\Big(\frac{u}{\varepsilon},X^{\varepsilon,t,\xi'}_{u},0,0\Big)\Big\vert^{2}du\bigg]\bigg]\notag\\\label{(3.8)}
 		\coloneqq&\,I_{4}+I_{5},
 	    \end{align}
where the positive constant $C$ depends on $T,L,c$ and $\underline{\sigma}$.\\
As for $I_1$, from assumption (H2) and Cauchy-Schwarz inequality, we get
     \begin{equation}
     	\begin{aligned}
     	I_{1}&=
     	C\hat{\mathbb{E}}_{s}\Big[\big\vert \varphi(X^{\varepsilon,t,\xi}_{T})-\varphi(X^{\varepsilon,t,\xi'}_{T})\big\vert^{2}\Big]\\\notag
     	&\leq
     	C\hat{\mathbb{E}}_{s}\bigg[\sup\limits_{r \in [s,T]}\Big(1+\big\vert X^{\varepsilon,t,\xi}_{r} \big\vert^{2m}+\big\vert X^{\varepsilon,t,\xi'}_{r} \big\vert^{2m}\Big) 
     	\sup\limits_{r \in [s,T]}\big\vert X^{\varepsilon,t,\xi}_{r}-X^{\varepsilon,t,\xi'}_{r} \big\vert^{2} \bigg]\\
     	&\leq
     	C\bigg(\hat{\mathbb{E}}_{s}\bigg[\sup\limits_{r \in [s,T]}\Big(1+\big\vert X^{\varepsilon,t,\xi}_{r} \big\vert^{4m}+\big\vert X^{\varepsilon,t,\xi'}_{r} \big\vert^{4m}\Big)\bigg]\bigg)^{\frac{1}{2}}
     	\bigg(\hat{\mathbb{E}}_{s}\bigg[\sup\limits_{r \in [s,T]}\big\vert X^{\varepsilon,t,\xi}_{r}-X^{\varepsilon,t,\xi'}_{r} \big\vert^{4}\bigg]\bigg)^{\frac{1}{2}}.
     	\end{aligned}
     \end{equation}
Based on \eqref{(3.6)} and Lemma \ref{lemma 3.1}\,(iii), we conduct a similar analysis as \eqref{(3.5)}-\eqref{(3.6)} and deduce that
     \begin{equation}
     	\begin{aligned}
     	I_{1}&\leq
     	C\Big(1+\big\vert X^{\varepsilon,t,\xi}_{s} \big\vert^{2m}+\big\vert X^{\varepsilon,t,\xi'}_{s} \big\vert^{2m}\Big) 
     	\big\vert X^{\varepsilon,t,\xi}_{s}-X^{\varepsilon,t,\xi'}_{s} \big\vert^{2}. \notag
 		\end{aligned}
     \end{equation}
As for $I_{2}$, by assumption (H2) and a similar analysis as $I_{1}$, we get
     \begin{equation}
     	\begin{aligned}
		I_{2}
		&=
		C\hat{\mathbb{E}}_{s}\bigg[\int\nolimits_{s}^{T} \Big\vert  f\Big(\frac{r}{\varepsilon},X^{\varepsilon,t,\xi}_{r},Y^{\varepsilon,t,\xi}_{r},Z^{\varepsilon,t,\xi}_{r}\Big)- f\Big(\frac{r}{\varepsilon},X^{\varepsilon,t,\xi'}_{r},Y^{\varepsilon,t,\xi}_{r},Z^{\varepsilon,t,\xi}_{r}\Big)\Big\vert^{2}dr\bigg]\\
		&\leq
	    C\hat{\mathbb{E}}_{s}\bigg[\sup\limits_{r \in [s,T]}\Big(1+\big\vert X^{\varepsilon,t,\xi}_{r} \big\vert^{2m}+\big\vert X^{\varepsilon,t,\xi'}_{r} \big\vert^{2m}\Big) 
	    \sup\limits_{r \in [s,T]}\big\vert X^{\varepsilon,t,\xi}_{r}-X^{\varepsilon,t,\xi'}_{r} \big\vert^{2} \bigg]\\\notag
	    &\leq
	    C\Big(1+\big\vert X^{\varepsilon,t,\xi}_{s} \big\vert^{2m}+\big\vert X^{\varepsilon,t,\xi'}_{s} \big\vert^{2m}\Big) 
	    \big\vert X^{\varepsilon,t,\xi}_{s}-X^{\varepsilon,t,\xi'}_{s} \big\vert^{2}.
	    \end{aligned}
    \end{equation}
As for $I_{3}$, applying assumption (H3), Lemma \ref{lemma 3.1}\,(iii) and $X^{\varepsilon,t,\xi}_{r}=X^{\varepsilon,s,X^{\varepsilon,t,\xi}_{s}}_{r}$ for any $r \in [s,T]$, we have
\begin{align}
		I_{3}
		&=
		C\bigg(\hat{\mathbb{E}}_{s}\bigg[\sup\limits_{r \in [s,T]}\big\vert S(r,X^{\varepsilon,t,\xi}_{r})-S(r,X^{\varepsilon,t,\xi'}_{r})\big\vert^{2}\bigg]\bigg)^{\frac{1}{2}}\notag\\
		&\leq
		C\bigg(\hat{\mathbb{E}}_{s}\bigg[
		\sup\limits_{r \in [s,T]}\big\vert X^{\varepsilon,s,X^{\varepsilon,t,\xi}_{s}}_{r}-X^{\varepsilon,s,X^{\varepsilon,t,\xi'}_{s}}_{r} \big\vert^{2} \bigg]\bigg)^{\frac{1}{2}}\notag\\\notag
		&\leq
		C\big\vert X^{\varepsilon,t,\xi}_{s}-X^{\varepsilon,t,\xi'}_{s} \big\vert.
	\end{align}
By a similar analysis as \eqref{(3.6)}, we use \eqref{(3.3)}-\eqref{(3.4)} and Lemma \ref{(3.1)} to derive that 
\begin{equation}
   	\begin{aligned}
		I_{4}
		=&
		\,\hat{\mathbb{E}}_{s}\bigg[\sup\limits_{r \in [s,T]}\hat{\mathbb{E}}_{r}\bigg[1+\big\vert \varphi(X^{\varepsilon,t,\xi}_{T})\big\vert^{2} +\int\nolimits_{s}^{T} \Big\vert  f\Big(\frac{u}{\varepsilon},X^{\varepsilon,t,\xi}_{u},0,0\Big)\Big\vert^{2}du\bigg]\bigg]\\ \notag
		\leq
		&\,\hat{\mathbb{E}}_{s}\bigg[\sup\limits_{r \in [s,T]}   \bigg(\hat{\mathbb{E}}_{r}\bigg[C\Big(1+\big\vert X^{\varepsilon,t,\xi}_{T} \big\vert^{2m+2}\Big)\bigg]
		+C\hat{\mathbb{E}}_{r}\bigg[\sup\limits_{u \in [r,T]} \Big(1+\big\vert X^{\varepsilon,t,\xi}_{u} \big\vert^{2m+2}\Big)\bigg]\bigg)\bigg]\\
		&+ \hat{\mathbb{E}}_{s}\bigg[\sup\limits_{r \in [s,T]} \int\nolimits_{s}^{r}C\Big(1+\big\vert X^{\varepsilon,t,\xi}_{u} \big\vert^{2m+2}\Big)du \bigg]\\
	   \leq
		&\,\hat{\mathbb{E}}_{s}\bigg[\sup\limits_{r \in [s,T]}C\Big(1+\big\vert X^{\varepsilon,t,\xi}_{r} \big\vert^{2m+2}\Big)\bigg]
		+C\hat{\mathbb{E}}_{s}\bigg[\sup\limits_{r \in [s,T]}\Big(1+\big\vert X^{\varepsilon,t,\xi}_{r} \big\vert^{2m+2}\Big)\bigg]\\
		\leq
		&\,C\Big(1+\big\vert X^{\varepsilon,t,\xi}_{s} \big\vert^{2m+2}\Big).
	\end{aligned}
\end{equation}
Likewise, we deduce that
   \begin{align}
		I_{5}
		&\leq
		C\Big(1+\big\vert X^{\varepsilon,t,\xi'}_{s} \big\vert^{2m+2}\Big).\notag
	\end{align}
Thus, we substitute the above results into \eqref{(3.7)} to yield the desired result.\\
 ${}\hspace{1.5em}$(iii) 
In terms of Lemma \ref{lemma 5.4} in appendix, it is clear that
     \begin{equation}
     	\begin{aligned}
     	\hat{\mathbb{E}}_{t}\bigg[\int\nolimits_t^{T}\big\vert Z^{\varepsilon,t,\xi}_{r}\big\vert^2dr\bigg] 
     	\leq 
     	&\,C\bigg(\hat{\mathbb{E}}_{t}\bigg[\sup\limits_{r \in [t,T]}\big\vert Y^{\varepsilon,t,\xi}_{r} \big\vert^{2}\bigg]\bigg)^{\frac{1}{2}}
     	\bigg(\hat{\mathbb{E}}_{t}\bigg[\bigg(\int\nolimits_{t}^{T}\Big\vert  f\Big(\frac{r}{\varepsilon},X^{\varepsilon,t,\xi}_{r},0,0\Big) \Big\vert dr\bigg)^{2}\bigg]\bigg)^{\frac{1}{2}}\\ \notag
     	&+C\hat{\mathbb{E}}_{t}\bigg[\sup\limits_{r \in [t,T]}\big\vert Y^{\varepsilon,t,\xi}_{r} \big\vert^{2}\bigg],
     	\end{aligned}
     \end{equation}
where the positive constant $C$ depends on the arguments $T,L,c,\underline{\sigma}$. Combining (i) and Lemma \ref{lemma 3.1}\,(i), we derive that
     \begin{align}\label{(3.9)}
     	\hat{\mathbb{E}}_{t}\bigg[\sup\limits_{r \in [t,T]}\big\vert Y^{\varepsilon,t,\xi}_{r} \big\vert^{2}\bigg] 
     	\leq 
     	C(T,L,c,\underline{\sigma})\hat{\mathbb{E}}_{t}\bigg[\sup\limits_{r \in [t,T]}\big(1+\big\vert X^{\varepsilon,t,\xi}_{r} \big\vert^{2m+2}\big)\bigg]
     	\leq 
     	C(T,L,c,\underline{\sigma})\big(1+\vert \xi \vert^{2m+2}\big).
     \end{align}
Using the H\"{o}lder's inequality, \eqref{(3.4)} and \eqref{(3.9)}, we have
     	\begin{align}
     	\hat{\mathbb{E}}_{t}\bigg[\displaystyle\int\nolimits_t^{T}\big\vert Z^{\varepsilon,t,\xi}_{r}\big\vert^2dr\bigg] 
 	    &\leq
 	    C\big(1+\vert\xi\vert^{2m+2}\big)
 	    +C\big(1+\vert\xi \vert^{m+1}\big)
 	    \bigg(\hat{\mathbb{E}}_{t}\bigg[\int\nolimits_{t}^{T} \Big\vert  f\Big(\frac{r}{\varepsilon},X^{\varepsilon,t,\xi}_{r},0,0\Big) \Big\vert^{2} dr\bigg]\bigg)^{\frac{1}{2}}\notag\\ \notag
 	    &\leq
 	    C\big(1+\vert \xi \vert^{2m+2}\big)
 	    +C\big(1+\vert \xi \vert^{m+1}\big)
 	    \bigg(\hat{\mathbb{E}}_{t}\bigg[\int\nolimits_{t}^{T} C(L)\big(1+\big\vert X^{\varepsilon,t,\xi}_{r} \big\vert^{2m+2}\big) dr\bigg]\bigg)^{\frac{1}{2}}\\ \notag
 	    &\leq
 	    C\big(1+\vert \xi \vert^{2m+2}\big)
 	    +C\big(1+\vert \xi \vert^{m+1}\big)
 	    \bigg(\hat{\mathbb{E}}_{t}\bigg[\sup\limits_{r \in [t,T]} \big(1+\big\vert X^{\varepsilon,t,\xi}_{r} \big\vert^{2m+2}\big) \bigg]\bigg)^{\frac{1}{2}}\\ \notag
 	    &\leq C\big(1+\vert \xi \vert^{2m+2}\big),
 	 \end{align}
where the positive constant $C$ depends on $T,L,c$ and $\underline{\sigma}$. The proof is completed.
\end{proof}\par
Next, we give a probabilistic representation for the solution of an obstacle problem for a fully
nonlinear parabolic PDE by using the reflected $G$-BSDE with a small parameter $\varepsilon \in (0,1)$ in \eqref{(3.1)}.
In fact, for the linear case, martingale problem approach can be utilized to identify whether the limit process is a solution to the given equation. However, the nonlinear martingale problem still needs further study, it can not be utilized to verify the conclusion we want.\par
In order to deal with this difficulty, we adopt the nonlinear Feynman-Kac formula to establish the relation between reflected $G$-FBSDE  \eqref{(3.1)} and the corresponding obstacle problem for fully nonlinear parabolic PDE.
To this end, we define the function $u^{\varepsilon}$ satisfying
    \begin{equation}\label{(3.10)}
	u^{\varepsilon}(t,x) \coloneqq Y^{\varepsilon,t,x}_{t},\quad\forall(t,x) \in [0,T] \times \mathbb{R}^{n}.
	\end{equation}
In fact, it is important to note that $u^{\varepsilon}(t,x)$ is a deterministic function. And also, the following lemma is efficient to cope with the limit distribution of $G$-FBSDE \eqref{(3.1)} through the solution $u^{\varepsilon}$ of the following fully nonlinear PDE. More relevant details can be obtained in Peng \cite{Peng3} and Li et al. \cite{LPS}.
    \begin{lemma}\label{lemma 3.3}
    	Suppose that $\mathrm{(H1)}$-$\mathrm{(H3)}$ or $\mathrm{(H1)}$-$\mathrm{(H2)}\mathrm{,}$$\,\mathrm{(H4)}$ hold. Then $u^{\varepsilon}$ defined by \eqref{(3.10)} is the unique visco- sity solution of the following obstacle problem:
    	\begin{equation}\label{(3.11)}
    			\left\{\begin{matrix} 
    			\begin{aligned}     		
    				&\min\big(-\partial_{t}u^{\varepsilon}(t,x)-F\big(\tfrac{t}{\varepsilon},x,u^{\varepsilon},D_{x}u^{\varepsilon},D^{2}_{x}u^{\varepsilon}\big),\,u^{\varepsilon}(t,x)-S(t,x)\big)=0,\quad\hspace{0.5em} (t,x) \in (0,T)\times \mathbb{R}^{n},\\ 
    				&\,u^{\varepsilon}(T,x)=\varphi(x),\hspace{23.6em} x\in\mathbb{R}^{n},	
    		\end{aligned}
    	\end{matrix}
    		\right.
    	\end{equation}
where for $(t,x,v,p,A)\in \mathbb{R}^{+} \times \mathbb{R}^{n}\times \mathbb{R}\times \mathbb{R}^{1\times n} \times \mathbb{S}(n)$$\mathrm{,}$
         \begin{equation}\label{(3.12)}
         	\begin{aligned}
         	&F(t,x,v,p,A)=G(H(t,x,v,p,A))+pb(t,x)+f(t,x,v,p\sigma(t,x)),\\
	        &H_{ij}(t,x,v,p,A)=(\sigma(t,x)^\top A\sigma(t,x))_{ij}+2ph_{ij}(t,x)+2g_{ij}(t,x,v,p\sigma(t,x)).
	        \end{aligned}
         \end{equation}
Moreover$\mathrm{,}$ we also get
\begin{align}
        	u^{\varepsilon}(s,X^{\varepsilon,t,x}_{s})=Y^{\varepsilon,t,x}_{s},\ \forall\,0 \leq t \leq s \leq T,\,x \in \mathbb{R}^{n}.\notag
\end{align}
\end{lemma}

\section{The averaging principle}\label{section 4}

\noindent In this section, our aim is to establish the averaging principle regarding the dynamical system \eqref{(3.1)} perturbed by a small parameter $\epsilon \in (0,1)$ by using the nonlinear stochastic analysis technique and viscosity solution method. Thus, to ensure the subsequent discussions on the existence and uniqueness of viscosity solution to go on successfully, we will present the following assumption.\\
\textbf{($\mathbf{H5}$)} For each $(x,v,p,A) \in \mathbb{R}^{n} \times \mathbb{R} \times \mathbb{R}^{1 \times n} \times \mathbb{S}(n) $, the limit	\vspace{3pt}
    \begin{equation}
    	\bar{F}(x,v,p,A)
    	\coloneqq
    	\lim\limits_{s\rightarrow \infty}\frac{1}{s}\int\nolimits_{0}^{s}[pb(r,x)+f(r,x,v,p\sigma(r,x))+G(H(r,x,v,p,A))]dr\notag	\vspace{3pt}
	\end{equation}
${}\hspace{2.75em}$exists and is finite, where the function $H:\,\mathbb{R}^{+} \times \mathbb{R}^{n} \times \mathbb{R} \times \mathbb{R}^{1 \times n} \times \mathbb{S}(n) \rightarrow \mathbb{S}(d)$ is defined in \eqref{(3.12)}.\par
Next, we introduce an averaged PDE with the coefficient $\bar{F}$ defined above:	\vspace{2pt}
\begin{equation}
\begin{aligned}  
		\left\{\begin{matrix}
			\begin{split}
			&\min\big(-\partial_{t}\bar{u}(t,x)-\bar{F}(x,\bar{u},D_{x}\bar{u},D^{2}_{x}\bar{u}),\,\bar{u}(t,x)-S(t,x)\big)=0,\quad (t,x) \in (0,T)\times \mathbb{R}^{n},\\ 		\label{(4.1)}		
			&\,\bar{u}(T,x)=\varphi(x),\hspace{20.72em} x\in\mathbb{R}^{n}.
			\end{split}
		\end{matrix}\right.
\end{aligned}
\end{equation}
\vspace{-11pt}
\begin{remark}\label{remark 4.1}
	\rm{In} fact, by the probabilistic representation for the solution of an obstacle problem for a fully nonlinear parabolic PDE, if there exists some functions $\bar{b}(x),\,\bar{\sigma}(x),\,\bar{h}_{ij}(x),\,\bar{f}(x,y,z),\,\bar{g}_{ij}(x,y,z)$, such that the coefficient $\bar{F}$ has the following form:
	\begin{equation}\label{(4.2)}
		\begin{aligned}
			&\bar{F}(x,v,p,A)=G(\bar{H}(x,v,p,A))+p\bar{b}(x)+\bar{f}(x,v,p\bar{\sigma}(x)),\\
			&\bar{H}_{ij}(x,v,p,A)=(\bar{\sigma}(x)^\top A\bar{\sigma}(x))_{ij}+2p\bar{h}_{ij}(x)+2\bar{g}_{ij}(x,v,p\bar{\sigma}(x)),
		\end{aligned}
	\end{equation}
for $(x,v,p,A) \in \mathbb{R}^{n} \times \mathbb{R} \times \mathbb{R}^{1 \times n} \times \mathbb{S}(n)$, then from Lemma \ref{lemma 3.3}, we derive that $\bar{u}(t,x)=\bar{Y}^{t,x}_{t}$$\rm{,}$ where $\bar{Y}^{t,x}$ satisfies the following averaged $G$-FBSDE with the obstacle process $\{S(s,\bar{X}^{t,x}_{s})\}_{s \in [t,T]}$$\rm{:}$
\begin{equation}\label{(4.3)}
	\left\{
	\begin{matrix}
		 \bar{X}^{t,x}_{s}=x+\displaystyle\int_t^{s} \bar{b}(\bar{X}^{t,x}_{r})dr+\sum\limits_{i,j=1}^d\int_t^{s} \bar{h}_{ij}(\bar{X}^{t,x}_{r})d\langle B^{i},B^{j}\rangle_{r}+\int_t^{s} \bar{\sigma}(\bar{X}^{t,x}_{r})dB_{r},\hfill\hfill\\
		\begin{split}
			\bar{Y}^{t,x}_{s}
			=&\,\varphi(\bar{X}^{t,x}_{T})+\int_s^{T} \bar{f}\Big(\bar{X}^{t,x}_{r},\bar{Y}^{t,x}_{r},\bar{Z}^{t,x}_{r}\Big)dr-\int_s^{T}\bar{Z}^{t,x}_{r} dB_{r}+(\bar{A}^{t,x}_{T}-\bar{A}^{t,x}_{s})\\
			&+\sum_{i,j=1}^d\int_s^{T} \bar{g}_{ij}\Big(\bar{X}^{t,x}_{r},\bar{Y}^{t,x}_{r},\bar{Z}^{t,x}_{r}\Big)d\langle B^{i},B^{j}\rangle_{r}.\hfill\hfill
		\end{split}
	\end{matrix} \right. 
\end{equation}
\end{remark}\par	
Then we need to consider solutions of the above PDE \eqref{(4.1)} in the viscosity sense. To conclude the proper- ties regarding the solution of PDE, the best candidate to define
the notion of viscosity solution is by using the language of sub- and super-jets (see \cite{CIP}). So we try to introduce the following definitions.
\begin{definition}\label{definition 4.2}
	\rm{L}et $u \in C((0,T) \times \mathbb{R}^{n})$ and $(t,x) \in (0,T) \times \mathbb{R}^{n}$. We denote by $\mathcal{P}^{2,+}u(t,x)$ $\mathrm{(}$the ``parabolic superjet'' of $u$ at $(t,x)$) the set of triples $(p,q,X) \in \mathbb{R} \times \mathbb{R}^{1 \times n} \times \mathbb{S}(n) $ satisfying
\begin{align}
	u(s,y) \leq u(t,x)+p(s-t)+\langle q,y-x\rangle +\frac{1}{2}\langle X(y-x),y-x \rangle+o(\vert s-t \vert +\vert y-x \vert^{2}).\notag
\end{align}
Similarly, we define $\mathcal{P}^{2,-}u(t,x)$ $\mathrm{(}$the ``parabolic subjet'' of $u$ at $(t,x)$) by $\mathcal{P}^{2,-}u(t,x) \coloneqq -\mathcal{P}^{2,+}(-u)(t,x) $.
\end{definition}\par
Then, we present the definition of the viscosity solution of the obstacle problem \eqref{(4.1)}.
\begin{definition}\label{definition 4.3}
	\rm{I}t can be said that $u \in C([0,T]\times \mathbb{R}^{n})$ is a viscosity subsolution of \eqref{(4.1)} if $u(T,x) \leq \varphi(x),\,x \in \mathbb{R}^{n}$, and at any point $(t,x) \in (0,T) \times \mathbb{R}^{n}$, for any $(p,q,X) \in \mathcal{P}^{2,+}u(t,x)$,
	\begin{equation}
		\min(u(t,x)-S(t,x),-p-\bar{F}(x,u(t,x),q,X)) \leq 0.\notag
	\end{equation}
It can be said that $u \in C([0,T] \times \mathbb{R}^{n})$ is a viscosity supersolution of \eqref{(4.1)} if $u(T,x) \geq \varphi(x),\,x \in \mathbb{R}^{n}$, and at any point $(t,x) \in (0,T) \times \mathbb{R}^{n}$, for any $(p,q,X) \in \mathcal{P}^{2,-}u(t,x)$,
\begin{equation}
	\min(u(t,x)-S(t,x),-p-\bar{F}(x,u(t,x),q,X)) \geq 0.\notag
\end{equation}
$u \in C([0,T] \times \mathbb{R}^{n})$ is said to be a viscosity solution of \eqref{(4.1)} if it is both a viscosity sub- and supersolution.
\end{definition}\par
In order to better draw out the subsequent conclusions, we first give some priori estimates.

\begin{lemma}\label{lemma 4.4}
	Suppose that $\mathrm{(H1)}$-$\mathrm{(H3)}$ hold. Then for any $t,\,t_{1},\,t_{2} \in [0,T]$ and $x,\,x_{1},\,x_{2} \in \mathbb{R}^{n}\mathrm{,}$ we get\\
	${}\hspace{0.66em}$$\bm{(\mathbf{i})}$
	$\vert u^{\varepsilon}(t,x_{1})-u^{\varepsilon}(t,x_{2})\vert^{2} 
	\leq C[(1+\vert x_{1} \vert^{2m}+\vert x_{2} \vert^{2m}) \vert x_{1}-x_{2} \vert^{2}+(1+\vert x_{1} \vert^{m+1}+\vert x_{2} \vert^{m+1}) \vert x_{1}-x_{2} \vert]\mathbb{;}$\\
	${}\hspace{0.33em}$$\bm{(\mathbf{ii})}$
	$\vert u^{\varepsilon}(t_{2},x)-u^{\varepsilon}(t_{1},x)\vert 
	\leq C(1+\vert x \vert^{m+1}) \vert t_{2}-t_{1} \vert^{\frac{1}{2}}+C(1+\vert x \vert^{\frac{m+2}{2}}) \vert t_{2}-t_{1} \vert^{\frac{1}{4}}\mathbb{;}$\\
	$\bm{(\mathbf{iii})}$
	$\vert u^{\varepsilon}(t,x)\vert\leq C(1+\vert x\vert^{m+1})\mathrm{;}$\\
where the positive constant $C$ depends on $T,L,c$ and $\underline{\sigma}$.
\end{lemma}
\begin{proof}
	The first and third assertion can be directly derived from (i) and (ii) in Lemma \ref{lemma 3.2}. So we only need to prove the second statement. Inspired by Lemma 6.6 in \cite{LPS}, we fix the point $(t,x)\in [0,T] \times \mathbb{R}^{n}$ and utilize the reflected $G$-FBSDE defined in \eqref{(3.1)} to establish a new $G$-FBSDE with a new obstacle process, which plays an important role in carrying out the subsequent proof. To be more precise, this new $G$-FBSDE is a version which extends the time interval $[t,T]$ of the original stochastic equations to the time interval $[0,T]$.\par
	The main proof will proceed in the following three steps.\par
	$\mathbf{Step\ 1.}$ Firstly, we define $\widetilde{X}^{\varepsilon,t,x}_{s}=x,\,\widetilde{Y}^{\varepsilon,t,x}_{s}=Y^{\varepsilon,t,x}_{t},\,\widetilde{Z}^{\varepsilon,t,x}_{s}=0$ and $\widetilde{A}^{\varepsilon,t,x}_{s}=0$ for any $s\in [0,t]$. Then, we formulate the function $\widetilde{S}:[0,T] \times \mathbb{R}^{n} \rightarrow \mathbb{R}$, such that for any $x'\in \mathbb{R}^{n}$, 
\begin{equation}\label{(4.4)}
	\widetilde{S}(r,x')=\left\{
	\begin{matrix}
		{S}(t,x'),\quad r \in [0,t],\hfill\hfill\\
		\begin{split}
			{S}(r,x'),\quad r \in [t,T]. \hfill\hfill
		\end{split}
	\end{matrix} \right. 
\end{equation}\par	
Owing to assumptions (H2)-(H3) and the structure of $S(\cdot,\cdot)$, we can demonstrate that the function $\widetilde{S}(\cdot,\cdot)$ also satisfies assumption (H3). Namely, for any $x'\in \mathbb{R}^{n}$, $\widetilde{S}(r,x')$ satisfies the Lipschitz condition with respect to $x'$, $\widetilde{S} \leq c$ and $\varphi(x') \geq \widetilde{S}(T,x')=S(T,x')$. Furthermore, it is easy to check that $\widetilde{S}(r,x')$ is continuous with respect to the argument $r$ by using the continuity property of $S(\cdot,\cdot)$. \par
$\mathbf{Step\ 2.}$ We establish the following reflected $G$-FBSDE:
\begin{equation}\label{(4.5)}
	\left\{
	\begin{aligned}
		 \widetilde{X}^{\varepsilon,t,x}_{s}
		 =&\,\widetilde{X}^{\varepsilon,t,x}_{0}+\int_0^{s} b\Big(\frac{r}{\varepsilon},\widetilde{X}^{\varepsilon,t,x}_{r}\Big)\mathbb{I}_{[t,T]}(r)dr+\int_0^{s} \sigma\Big(\frac{r}{\varepsilon},\widetilde{X}^{\varepsilon,t,x}_{r}\Big)\mathbb{I}_{[t,T]}(r)dB_{r}\\
		&+\sum\limits_{i,j=1}^d\int_0^{s} h_{ij}\Big(\frac{r}{\varepsilon},\widetilde{X}^{\varepsilon,t,x}_{r}\Big)\mathbb{I}_{[t,T]}(r)d\langle B^{i},B^{j}\rangle_{r},\\
	    \widetilde{Y}^{\varepsilon,t,x}_{s}
	    =
	    &\,\varphi(\widetilde{X}^{\varepsilon,t,x}_{T})+\int_s^{T} f\Big(\frac{r}{\varepsilon},\widetilde{X}^{\varepsilon,t,x}_{r},\widetilde{Y}^{\varepsilon,t,x}_{r},\widetilde{Z}^{\varepsilon,t,x}_{r}\Big)\mathbb{I}_{[t,T]}(r)dr-\int_s^{T}\widetilde{Z}^{\varepsilon,t,x}_{r} dB_{r}+\big(\widetilde{A}^{\varepsilon,t,x}_{T}-\widetilde{A}^{\varepsilon,t,x}_{s}\big)\\
		&+\sum_{i,j=1}^d\int_s^{T} g_{ij}\Big(\frac{r}{\varepsilon},\widetilde{X}^{\varepsilon,t,x}_{r},\widetilde{Y}^{\varepsilon,t,x}_{r},\widetilde{Z}^{\varepsilon,t,x}_{r}\Big)\mathbb{I}_{[t,T]}(r)d\langle B^{i},B^{j}\rangle_{r},\hfill\hfill
	\end{aligned} \right. 
\end{equation}\vspace{-5pt}\\
where the functions $b,\,\sigma,\,h_{ij},\,f,\,g_{ij}$ and $\varphi$ are defined in \eqref{(3.1)}.\par
$\mathbf{Part\ 1.}$ We firstly demonstrate that the coefficients of the above equation \eqref{(4.5)} still satisfy assumptions (H1)-(H2). For simplicity, $\ \forall (r,x',y',z') \in [0,T] \times \mathbb{R}^{n} \times \mathbb{R} \times \mathbb{R}^{1 \times d} $, we redefine these coefficients:
\begin{equation}\label{(4.6)}
	\begin{aligned}
     b^{\varepsilon,t}_{1}(r,x')=b\Big(\frac{r}{\varepsilon},x'\Big)\mathbb{I}_{[t,T]}(r),\quad
     \sigma^{\varepsilon,t}_{1}(r,x')=\sigma\Big(\frac{r}{\varepsilon},x'\Big)\mathbb{I}_{[t,T]}(r),\quad
     h^{\varepsilon,t}_{1,ij}(r,x')=h_{ij}\Big(\frac{r}{\varepsilon},x'\Big)\mathbb{I}_{[t,T]}(r), \\
     f^{\varepsilon,t}_{1}(r,x',y',z')=f\Big(\frac{r}{\varepsilon},x',y',z'\Big)\mathbb{I}_{[t,T]}(r),\quad
     g^{\varepsilon,t}_{1,ij}(r,x',y',z')=g_{ij}\Big(\frac{r}{\varepsilon},x',y',z'\Big)\mathbb{I}_{[t,T]}(r),
	\end{aligned}
\end{equation}
From the properties of the coefficients $b,\,\sigma,\,h_{ij},\,f,\,g_{ij},\,\varphi$, it is not hard to check that the coefficients established in \eqref{(4.6)} also satisfy assumptions (H1)-(H2). \par

$\mathbf{Part\ 2.}$ We now establish a new obstacle process $\{\widetilde{L}^{\varepsilon,t,x}_{s}\}_{s \in [0,T]}$ by utilizing \eqref{(4.4)} and \eqref{(4.5)}. For each fixed $t \in [0,T]$, we consider the following stochastic process: 
\begin{equation}\label{(4.7)}
	\widetilde{L}^{\varepsilon,t,x}_{r}
	\coloneqq
	\widetilde{S}(r,\widetilde{X}^{\varepsilon,t,x}_{r})=\left\{
	\begin{matrix}
		{S}(t,x),\hspace{3.45em} r \in [0,t],\hfill\hfill\\
		\begin{split}
			{S}(r,X^{\varepsilon,t,x}_{r}),\hspace{1.45em} r \in [t,T].\hfill\hfill
		\end{split}
	\end{matrix} \right. \\
\end{equation}
In addition, with the help of the above analysis regarding $\widetilde{S}(\cdot,\cdot)$, we can get $\varphi(\widetilde{X}^{\varepsilon,t,x}_{T}) \geq \widetilde{S}(T,\widetilde{X}^{\varepsilon,t,x}_{T})=S(T,X^{\varepsilon,t,x}_{T})$.\par
$\mathbf{Part\ 3.}$ Among them, by using Theorem 5.1.3 in Peng \cite{Peng3} and Lemma \ref{lemma 5.1} in appendix, the above $G$-SDE in \eqref{(4.5)} admits the unique solution $\widetilde{X}^{\varepsilon,t,x} \in S^{2}_{G}(0,T)$. 
As for the reflected $G$-BSDE with the obstacle process $\{\widetilde{L}^{\varepsilon,t,x}_{s}\}_{s \in [0,T]}$ in \eqref{(4.5)}, there exists a unique triple $(\widetilde{Y}^{\varepsilon,t,x}_{s},\widetilde{Z}^{\varepsilon,t,x}_{s},\widetilde{A}^{\varepsilon,t,x}_{s})_{s\in [0,T]}$, such that\\
${}\hspace{0.40em}$(i) $\widetilde{Y}^{\varepsilon,t,x}_{s} \geq \widetilde{S}(s,\widetilde{X}^{\varepsilon,t,x}_{s}),\ 0 \leq s \leq T$;\\
${}\,$(ii) $\{\widetilde{A}^{\varepsilon,t,x}_{s}\}_{s \in [0,T]}$ is nondecreasing and continuous, and $\big\{-\int_0^{s} \big(\widetilde{Y}^{\varepsilon,t,x}_{r}-\widetilde{S}(r,\widetilde{X}^{\varepsilon,t,x}_{r})\big)d\widetilde{A}^{\varepsilon,t,x}_{r},0 \leq s \leq T\big\}$ is a \\
${}\hspace{2em}$nonincreasing $G$-martingale;\\
(iii)\,$(\widetilde{Y}^{\varepsilon,t,x}_{s},\widetilde{Z}^{\varepsilon,t,x}_{s},\widetilde{A}^{\varepsilon,t,x}_{s})_{s\in [0,T]} \in \mathcal{S}^{2}_{G}(0,T)$.\par
$\mathbf{Part\ 4.}$ We now redefine the form of solution to the above $G$-SDE in \eqref{(4.5)}. For any fixed $t \in [0,T]$, let $\widetilde{X}^{\varepsilon,t,0,x} \coloneqq \widetilde{X}^{\varepsilon,t,x}$. Essentially, the above $G$-SDE in \eqref{(4.5)} can be seen as a dynamical system starting from state $x \in \mathbb{R}^{n}$ at time $0$. And also, thanks to the uniqueness of the solution to the $G$-SDE in \eqref{(4.5)}, the estimates in Lemma \ref{lemma 3.1} are still valid for the $G$-SDE in \eqref{(4.5)}, which paves a convenient and easy way for the subsequent discussions. \par
$\mathbf{Step\ 3.}$ Without loss of generality, for the case $0 \leq t_{1} \leq t_{2} \leq T$, we can establish the connection between the solution $Y^{\varepsilon,t,x}$ and $\widetilde{Y}^{\varepsilon,t,x}$ by the structure of the $G$-FBSDE \eqref{(4.5)}, i.e. $Y^{\varepsilon,t,x}_{t}=\widetilde{Y}^{\varepsilon,t,x}_{0}$. In terms of \eqref{(3.10)} and Lemma \ref{lemma 5.5} in appendix, we have
	\begin{align}
		&\,\big\vert u^{\varepsilon}(t_{1},x)-u^{\varepsilon}(t_{2},x) \big\vert^{2}\notag\\\notag
		=
		&\,\big\vert Y^{\varepsilon,t_{1},x}_{t_{1}}-Y^{\varepsilon,t_{2},x}_{t_{2}} \big\vert^{2}
		=
		\big\vert \widetilde{Y}^{\varepsilon,t_{1},x}_{0}-\widetilde{Y}^{\varepsilon,t_{2},x}_{0} \big\vert^{2}\\\notag
		\leq
		&\,C\hat{\mathbb{E}}_{0}\bigg[\big\vert \varphi\big(\widetilde{X}^{\varepsilon,t_{1},x}_{T}\big)-\varphi\big(\widetilde{X}^{\varepsilon,t_{2},x}_{T}\big)\big\vert^{2}\bigg]\\\notag
		&+C\hat{\mathbb{E}}_{0}\bigg[\int\nolimits_{0}^{T}\Big\vert f\Big(\frac{r}{\varepsilon},\widetilde{X}^{\varepsilon,t_{1},x}_{r},\widetilde{Y}^{\varepsilon,t_{2},x}_{r},\widetilde{Z}^{\varepsilon,t_{2},x}_{r}\Big)\mathbb{I}_{[t_{1},T]}(r)-f\Big(\frac{r}{\varepsilon},\widetilde{X}^{\varepsilon,t_{2},x}_{r},\widetilde{Y}^{\varepsilon,t_{2},x}_{r},\widetilde{Z}^{\varepsilon,t_{2},x}_{r}\Big)\mathbb{I}_{[t_{2},T]}(r)\Big\vert^{2}dr\bigg]\\\notag
		&+C\bigg(\hat{\mathbb{E}}_{0}\bigg[\sup_{s \in [0,T]}  \big\vert \widetilde{S}\big(s,\widetilde{X}^{\varepsilon,t_{1},x}_{s}\big)-\widetilde{S}\big(s,\widetilde{X}^{\varepsilon,t_{2},x}_{s}\big) \big\vert^{2}\bigg]\bigg)^{\frac{1}{2}}(\Phi_{0,T})^{\frac{1}{2}},\notag
 	\end{align}
and
	\begin{align}
		\Phi_{0,T}=\ &
		\hat{\mathbb{E}}_{0}\bigg[\sup\limits_{s \in [0,T]}\hat{\mathbb{E}}_{s}\bigg[1+\big\vert \varphi\big(\widetilde{X}^{\varepsilon,t_{1},x}_{T}\big)\big\vert^{2} +\int\nolimits_{0}^{T} \Big\vert  f\Big(\frac{u}{\varepsilon},\widetilde{X}^{\varepsilon,t_{1},x}_{u},0,0\Big)\mathbb{I}_{[t_{1},T]}(u)\Big\vert^{2}du\bigg]\bigg]\notag\\ \notag
		&+\hat{\mathbb{E}}_{0}\bigg[\sup\limits_{s \in [0,T]}\hat{\mathbb{E}}_{s}\bigg[1+\big\vert \varphi \big(\widetilde{X}^{\varepsilon,t_{2},x}_{T}\big)\big\vert^{2} +\int\nolimits_{0}^{T} \Big\vert  f\Big(\frac{u}{\varepsilon},\widetilde{X}^{\varepsilon,t_{2},x}_{u},0,0\Big)\mathbb{I}_{[t_{2},T]}(u)\Big\vert^{2}du\bigg]\bigg].
	\end{align}
Given that $u^{\varepsilon}(t,x)$ is a deterministic function, by taking expectation on both sides, it is clear that
	\begin{align}\label{(4.8)}
		&\ \big\vert u^{\varepsilon}(t_{1},x)-u^{\varepsilon}(t_{2},x) \big\vert^{2}\notag\\
		\leq
		&\ C\hat{\mathbb{E}}\bigg[\big\vert \varphi\big(\widetilde{X}^{\varepsilon,t_{1},x}_{T}\big)-\varphi\big(\widetilde{X}^{\varepsilon,t_{2},x}_{T}\big)\big\vert^{2}\bigg]\notag\\
		&+C\hat{\mathbb{E}}\bigg[\int\nolimits_{0}^{T}\Big\vert f\Big(\frac{r}{\varepsilon},\widetilde{X}^{\varepsilon,t_{1},x}_{r},\widetilde{Y}^{\varepsilon,t_{2},x}_{r},\widetilde{Z}^{\varepsilon,t_{2},x}_{r}\Big)\mathbb{I}_{[t_{1},T]}(r)-f\Big(\frac{r}{\varepsilon},\widetilde{X}^{\varepsilon,t_{2},x}_{r},\widetilde{Y}^{\varepsilon,t_{2},x}_{r},\widetilde{Z}^{\varepsilon,t_{2},x}_{r}\Big)\mathbb{I}_{[t_{2},T]}(r)\Big\vert^{2}dr\bigg]\notag\\
		&+C\bigg(\hat{\mathbb{E}}\bigg[\sup_{s \in [0,T]}  \big\vert \widetilde{S}\big(s,\widetilde{X}^{\varepsilon,t_{1},x}_{s}\big)-\widetilde{S}\big(s,\widetilde{X}^{\varepsilon,t_{2},x}_{s}\big) \big\vert^{2}\bigg]\bigg)^{\frac{1}{2}}(\Phi'_{0,T})^{\frac{1}{2}},\notag\\
		\coloneqq
		&\ J_{1}+J_{2}+J_{3}(J_{4}+J_{5})^{\frac{1}{2}},
	\end{align}
and
	\begin{align}
		\Phi'_{0,T}=&\ 
		\hat{\mathbb{E}}\bigg[\sup\limits_{s \in [0,T]}\hat{\mathbb{E}}_{s}\bigg[1+\big\vert \varphi\big(\widetilde{X}^{\varepsilon,t_{1},x}_{T}\big)\big\vert^{2} +\int\nolimits_{0}^{T} \Big\vert  f\Big(\frac{u}{\varepsilon},\widetilde{X}^{\varepsilon,t_{1},x}_{u},0,0\Big)\mathbb{I}_{[t_{1},T]}(u)\Big\vert^{2}du\bigg]\bigg]\notag\\ \notag
		&+\hat{\mathbb{E}}\bigg[\sup\limits_{s \in [0,T]}\hat{\mathbb{E}}_{s}\bigg[1+\big\vert \varphi\big(\widetilde{X}^{\varepsilon,t_{2},x}_{T}\big)\big\vert^{2} +\int\nolimits_{0}^{T} \Big\vert  f\Big(\frac{u}{\varepsilon},\widetilde{X}^{\varepsilon,t_{2},x}_{u},0,0\Big)\mathbb{I}_{[t_{2},T]}(u)\Big\vert^{2}du\bigg]\bigg]\notag\\\notag
		\coloneqq
		&\ J_{4}+J_{5},
	\end{align}
where the positive constant $C$ depends on $L,\,T,\,c,\,\underline{\sigma}$.\\
As for $J_{1}$, from the structure of the new $G$-SDE in \eqref{(4.5)} and the uniqueness of solution to the previous $G$-SDE in \eqref{(3.1)}, it is easy to check that $\widetilde{X}^{\varepsilon,t_{1},x}_{r}={X}^{\varepsilon,t_{1},x}_{r}={X}^{\varepsilon,t_{2},{X}^{\varepsilon,t_{1},x}_{t_{2}}}_{r}$ for any $r \in [t_{2},T]$. At the same time, we also get $\widetilde{X}^{\varepsilon,t_{2},x}_{r}={X}^{\varepsilon,t_{2},x}_{r}$ when $r \in [t_{2},T]$. Thus, by Cauchy-Schwarz inequality and assumption (H2), we have 
\begin{align}
	J_{1}
	=
	&\,C\hat{\mathbb{E}}\bigg[\big\vert \varphi\big(\widetilde{X}^{\varepsilon,t_{1},x}_{T}\big)-\varphi\big(\widetilde{X}^{\varepsilon,t_{2},x}_{T}\big)\big\vert^{2}\bigg]\notag\\\notag
    \leq
    &\,C\bigg(\hat{\mathbb{E}}\bigg[\sup_{r \in [t_{2},T]}\Big(1+\big\vert \widetilde{X}^{\varepsilon,t_{1},x}_{r}\big\vert^{4m}+\big\vert \widetilde{X}^{\varepsilon,t_{2},x}_{r}\big\vert^{4m}\Big)\bigg]\bigg)^{\frac{1}{2}}
    \bigg(\hat{\mathbb{E}}\bigg[\sup_{r \in [t_{2},T]}\big\vert \widetilde{X}^{\varepsilon,t_{1},x}_{r}-\widetilde{X}^{\varepsilon,t_{2},x}_{r}\big\vert^{4}\bigg]\bigg)^{\frac{1}{2}}\notag
\end{align}
\begin{align}
		\leq
		&\,C\bigg(\hat{\mathbb{E}}\bigg[\sup_{r \in [t_{2},T]}\Big(1+\Big\vert {X}^{\varepsilon,t_{2},{X}^{\varepsilon,t_{1},x}_{t_{2}}}_{r}\Big\vert^{4m}+\big\vert X^{\varepsilon,t_{2},x}_{r}\big\vert^{4m}\Big)\bigg]\bigg)^{\frac{1}{2}}
		\bigg(\hat{\mathbb{E}}\bigg[\sup_{r \in [t_{2},T]}\Big\vert {X}^{\varepsilon,t_{2},{X}^{\varepsilon,t_{1},x}_{t_{2}}}_{r}-X^{\varepsilon,t_{2},x}_{r}\Big\vert^{4}\bigg]\bigg)^{\frac{1}{2}}.\notag
\end{align}
By Lemma \ref{lemma 3.1} and a similar analysis as \eqref{(3.6)}, we deduce that
	\begin{align}\label{(4.9)}
		J_{1}
		\leq
		&\,C\bigg(1+\vert x\vert^{4m}+\hat{\mathbb{E}}\bigg[\sup_{r \in [t_{1},t_{2}]}\big\vert X^{\varepsilon,t_{1},x}_{r}\big\vert^{4m}\bigg]\bigg)^{\frac{1}{2}}
		\bigg(\hat{\mathbb{E}}\bigg[\sup_{r \in [t_{1},t_{2}]}\big\vert X^{\varepsilon,t_{1},x}_{r}-x\big\vert^{4}\bigg]\bigg)^{\frac{1}{2}}\notag\\
		\leq
		&\,C(1+\vert x\vert^{2m+2})\vert t_{2}-t_{1} \vert.
	\end{align}
Owing to the emergence of indicator function, the following analysis regarding $J_{2}$ need to be conducted in three different time intervals, i.e., $[0,t_{1}],\,[t_{1},t_{2}]$ and $[t_{2},T]$. Thus, we get
\begin{align}
	J_{2}
	=
	&\,C\hat{\mathbb{E}}\bigg[\int\nolimits_{0}^{T}\Big\vert f\Big(\frac{r}{\varepsilon},\widetilde{X}^{\varepsilon,t_{1},x}_{r},\widetilde{Y}^{\varepsilon,t_{2},x}_{r},\widetilde{Z}^{\varepsilon,t_{2},x}_{r}\Big)\mathbb{I}_{[t_{1},T]}(r)-f\Big(\frac{r}{\varepsilon},\widetilde{X}^{\varepsilon,t_{2},x}_{r},\widetilde{Y}^{\varepsilon,t_{2},x}_{r},\widetilde{Z}^{\varepsilon,t_{2},x}_{r}\Big)\mathbb{I}_{[t_{2},T]}(r)\Big\vert^{2}dr\bigg]\notag\\\notag
	\leq
	&\,C\hat{\mathbb{E}}\bigg[\int\nolimits_{t_{1}}^{t_{2}}\Big\vert f\Big(\frac{r}{\varepsilon},\widetilde{X}^{\varepsilon,t_{1},x}_{r},\widetilde{Y}^{\varepsilon,t_{2},x}_{r},\widetilde{Z}^{\varepsilon,t_{2},x}_{r}\Big)\Big\vert^{2}dr\bigg]\\\notag
	&+C\hat{\mathbb{E}}\bigg[\int\nolimits_{t_{2}}^{T}\Big\vert f\Big(\frac{r}{\varepsilon},\widetilde{X}^{\varepsilon,t_{1},x}_{r},\widetilde{Y}^{\varepsilon,t_{2},x}_{r},\widetilde{Z}^{\varepsilon,t_{2},x}_{r}\Big)-f\Big(\frac{r}{\varepsilon},\widetilde{X}^{\varepsilon,t_{2},x}_{r},\widetilde{Y}^{\varepsilon,t_{2},x}_{r},\widetilde{Z}^{\varepsilon,t_{2},x}_{r}\Big)\Big\vert^{2}dr\bigg]\\\notag
	\coloneqq
	&\,J_{21}+J_{22}.
\end{align}
Likewise, as for $J_{21}$, for any $r \in [t_{1},t_{2}]$, we can immediately see that $\widetilde{X}^{\varepsilon,t_{1},x}_{r}={X}^{\varepsilon,t_{1},x}_{r},\,\widetilde{Y}^{\varepsilon,t_{2},x}_{r}={Y}^{\varepsilon,t_{2},x}_{t_{2}}$ and $\widetilde{Z}^{\varepsilon,t_{2},x}_{r}=0$. By assumption (H2) and Lemmas \ref{lemma 3.1}-\ref{lemma 3.2}, we obtain
	\begin{align}
    J_{21}
    =
    &\,C\hat{\mathbb{E}}\bigg[\int\nolimits_{t_{1}}^{t_{2}}\Big\vert f\Big(\frac{r}{\varepsilon},\widetilde{X}^{\varepsilon,t_{1},x}_{r},\widetilde{Y}^{\varepsilon,t_{2},x}_{r},\widetilde{Z}^{\varepsilon,t_{2},x}_{r}\Big)\Big\vert^{2}dr\bigg]\notag\\\notag
    \leq
    &\,C\hat{\mathbb{E}}\bigg[\int\nolimits_{t_{1}}^{t_{2}}
    \Big(\big(1+\big\vert\widetilde{X}^{\varepsilon,t_{1},x}_{r}\big\vert^{2m}\big)\big\vert\widetilde{X}^{\varepsilon,t_{1},x}_{r}\big\vert^{2}
    +\big\vert \widetilde{Y}^{\varepsilon,t_{2},x}_{r} \big\vert^{2}
    +\big\vert \widetilde{Z}^{\varepsilon,t_{2},x}_{r} \big\vert^{2}\Big)dr\bigg]\notag
    +C\hat{\mathbb{E}}\bigg[\int\nolimits_{t_{1}}^{t_{2}}\Big\vert f\Big(\frac{r}{\varepsilon},0,0,0\Big)\Big\vert^{2}dr\bigg]\\\notag
    \leq
    &\,C\vert t_{2}-t_{1} \vert\hat{\mathbb{E}}\bigg[\sup\limits_{r \in [t_{1},t_{2}]}
    \big(1+\big\vert\widetilde{X}^{\varepsilon,t_{1},x}_{r}\big\vert^{2m}\big)\sup\limits_{r \in [t_{1},t_{2}]}\big\vert\widetilde{X}^{\varepsilon,t_{1},x}_{r}\big\vert^{2}\bigg]\notag
    +C\vert t_{2}-t_{1}\vert\hat{\mathbb{E}}\bigg[\sup\limits_{r \in [t_{1},t_{2}]}\big\vert \widetilde{Y}^{\varepsilon,t_{2},x}_{r} \big\vert^{2}\bigg]
    +C\vert t_{2}-t_{1}\vert\notag \\\notag
    \leq
    &\,C\vert t_{2}-t_{1} \vert\hat{\mathbb{E}}\bigg[\sup\limits_{r \in [t_{1},t_{2}]}
    \big(1+\vert X^{\varepsilon,t_{1},x}_{r}\vert^{2m+2}\big)\bigg]
    +C\vert t_{2}-t_{1}\vert \vert Y^{\varepsilon,t_{2},x}_{t_{2}} \vert^{2}
    +C\vert t_{2}-t_{1}\vert \\\notag
    \leq
    &\,C(1+\vert x\vert ^{2m+2})\vert t_{2}-t_{1}\vert.
\end{align}
Similar to the idea of $J_{1}$, using assumption (H2), Cauchy-Schwarz inequality and Lemma \ref{lemma 3.1} can yield that
\begin{align}
		J_{22}
		=
		&\,C\hat{\mathbb{E}}\bigg[\int\nolimits_{t_{2}}^{T}\Big\vert f\Big(\frac{r}{\varepsilon},\widetilde{X}^{\varepsilon,t_{1},x}_{r},\widetilde{Y}^{\varepsilon,t_{2},x}_{r},\widetilde{Z}^{\varepsilon,t_{2},x}_{r}\Big)-f\Big(\frac{r}{\varepsilon},\widetilde{X}^{\varepsilon,t_{2},x}_{r},\widetilde{Y}^{\varepsilon,t_{2},x}_{r},\widetilde{Z}^{\varepsilon,t_{2},x}_{r}\Big)\Big\vert^{2}dr\bigg]\notag\\\notag
		\leq
		&\,C\hat{\mathbb{E}}\bigg[\sup\limits_{r \in [t_{2},T]}
		\big(1+\big\vert\widetilde{X}^{\varepsilon,t_{1},x}_{r}\big\vert^{2m}+\big\vert\widetilde{X}^{\varepsilon,t_{2},x}_{r}\big\vert^{2m}\big)\sup\limits_{r \in [t_{2},T]}\big\vert\widetilde{X}^{\varepsilon,t_{1},x}_{r}-\widetilde{X}^{\varepsilon,t_{2},x}_{r}\big\vert^{2}\bigg]\\\notag
		\leq
		&\,C\hat{\mathbb{E}}\bigg[\sup\limits_{r \in [t_{2},T]}
		\big(1+\vert X^{\varepsilon,t_{1},x}_{r}\vert^{2m}+\vert X^{\varepsilon,t_{2},x}_{r}\vert^{2m}\big)\sup\limits_{r \in [t_{2},T]}\vert X^{\varepsilon,t_{1},x}_{r}-X^{\varepsilon,t_{2},x}_{r}\vert^{2}\bigg]\notag\\\notag
		\leq
		&\,C\bigg(\hat{\mathbb{E}}\bigg[\sup\limits_{r \in [t_{2},T]}
		\Big(1+\Big\vert X^{\varepsilon,t_{2},X^{\varepsilon,t_{1},x}_{t_{2}}}_{r}\Big\vert^{4m}+\vert X^{\varepsilon,t_{2},x}_{r}\vert^{4m}\Big)\bigg]\bigg)^{\frac{1}{2}}
		\bigg(\hat{\mathbb{E}}\bigg[\sup\limits_{r \in [t_{2},T]}\Big\vert X^{\varepsilon,t_{2},X^{\varepsilon,t_{1},x}_{t_{2}}}_{r}-X^{\varepsilon,t_{2},x}_{r}\Big\vert^{4}\bigg]\bigg)^{\frac{1}{2}}\notag\\\notag
		\leq
		&\,C\bigg(1+\vert x \vert ^{4m}+\hat{\mathbb{E}}\bigg[\sup\limits_{r \in [t_{1},t_{2}]}
		\vert X^{\varepsilon,t_{1},x}_{r}\vert^{4m}\bigg]\bigg)^{\frac{1}{2}}
		\bigg(\hat{\mathbb{E}}\bigg[\sup\limits_{r \in [t_{1},t_{2}]}\vert X^{\varepsilon,t_{1},x}_{r}-x\vert^{4}\bigg]\bigg)^{\frac{1}{2}}\notag\\
		\leq
		&\,C(1+ \vert x \vert ^{2m+2}) \vert t_{2}-t_{1}\vert.\notag
\end{align}
Thus, we can conclude that
	\begin{align}\label{(4.10)}
		J_{2}\leq J_{21}+J_{22} \leq C(1+ \vert x \vert ^{2m+2}) \vert t_{2}-t_{1}\vert. 
	\end{align}
As for $J_{3}$, note that the properties of the function $\widetilde{S}$ we have concluded, together with similar analysis as before, we get
\begin{equation}\label{(4.11)}
	\begin{aligned}
    J_{3}
    =&\,C\bigg(\hat{\mathbb{E}}\bigg[\sup_{s \in [0,T]}  \big\vert \widetilde{S}\big(s,\widetilde{X}^{\varepsilon,t_{1},x}_{s}\big)-\widetilde{S}\big(s,\widetilde{X}^{\varepsilon,t_{2},x}_{s}\big) \big\vert^{2}\bigg]\bigg)^{\frac{1}{2}}\\
    \leq
    &\,C\bigg(\hat{\mathbb{E}}\bigg[\sup_{s \in [t_{1},t_{2}]}  \big\vert X^{\varepsilon,t_{1},x}_{s}-x \big\vert^{2}\bigg]\bigg)^{\frac{1}{2}}
    +C\bigg(\hat{\mathbb{E}}\bigg[\sup_{s \in [t_{2},T]}  \Big\vert X^{\varepsilon,t_{2},X^{\varepsilon,t_{1},x}_{t_{2}}}_{s}-X^{\varepsilon,t_{2},x}_{s} \Big\vert^{2}\bigg]\bigg)^{\frac{1}{2}}\\
    \leq
    &\,C(1+ \vert x \vert ) \vert t_{2}-t_{1}\vert^{\frac{1}{2}}
    +C\bigg(\hat{\mathbb{E}}\bigg[\sup_{s \in [t_{1},t_{2}]}  \big\vert X^{\varepsilon,t_{1},x}_{s}-x \big\vert^{2}\bigg]\bigg)^{\frac{1}{2}}\\
    \leq
    &\,C(1+ \vert x \vert) \vert t_{2}-t_{1}\vert^{\frac{1}{2}}.
\end{aligned}
\end{equation}
As for $J_4$, by means of assumption (H2), Lemma \ref{lemma 3.1} and $\widetilde{X}^{\varepsilon,t_{i},0,x}
\coloneqq \widetilde{X}^{\varepsilon,t_{i},x},\,i=1,2$, we deduce that
\begin{align}
	J_{4}
	=
    &\,\hat{\mathbb{E}}\bigg[\sup\limits_{s \in [0,T]}\hat{\mathbb{E}}_{s}\bigg[1+\big\vert \varphi\big(\widetilde{X}^{\varepsilon,t_{1},x}_{T}\big)\big\vert^{2} +\int\nolimits_{0}^{T} \Big\vert  f\Big(\frac{u}{\varepsilon},\widetilde{X}^{\varepsilon,t_{1},x}_{u},0,0\Big)\mathbb{I}_{[t_{1},T]}(u)\Big\vert^{2}du\bigg]\bigg]\notag\\\notag
    \leq
    &\,C\hat{\mathbb{E}}\bigg[\sup\limits_{s \in [0,T]}\hat{\mathbb{E}}_{s}\Big[1+\big\vert \widetilde{X}^{\varepsilon,t_{1},0,x}_{T}\big\vert^{2m+2}\Big]\bigg]
    +\hat{\mathbb{E}}\bigg[\sup\limits_{s \in [0,T]}\hat{\mathbb{E}}_{s}\bigg[\int\nolimits_{0}^{T} C\Big(1+\big\vert \widetilde{X}^{\varepsilon,t_{1},0,x}_{u}\big\vert^{2m+2}\Big)du+CT\bigg]\bigg]\\\notag
    \coloneqq 
    &\,J_{41}+J_{42}.
\end{align}
By applying $\widetilde{X}^{\varepsilon,t_{i},0,x}_{s}=\widetilde{X}^{\varepsilon,t_{i},r,\widetilde{X}^{\varepsilon,t_{i},0,x}_{r}}_{s},\,i=1,2,\,\forall  s \in [r,T]$, we have
\begin{align}
		J_{41}
		=&\,C\hat{\mathbb{E}}\bigg[\sup\limits_{s \in [0,T]}\hat{\mathbb{E}}_{s}\Big[1+\big\vert \widetilde{X}^{\varepsilon,t_{1},0,x}_{T}\big\vert^{2m+2}\Big]\bigg]\notag\\\notag
		\leq
		&\,C\hat{\mathbb{E}}\bigg[\sup\limits_{s \in [0,T]}\hat{\mathbb{E}}_{s}\bigg[\sup\limits_{r \in [s,T]}\Big(1+\Big\vert \widetilde{X}^{\varepsilon,t_{1},s,\widetilde{X}^{\varepsilon,t_{1},0,x}_{s}}_{r}\Big\vert^{2m+2}\Big)\bigg]\bigg]\\\notag
		\leq
		&\,C\hat{\mathbb{E}}\bigg[\sup\limits_{s \in [0,T]}\Big(1+\big\vert \widetilde{X}^{\varepsilon,t_{1},0,x}_{s}\big\vert^{2m+2}\Big)\bigg]\\\notag
		\leq
		&\,C(1+\vert x \vert^{2m+2}),\notag
\end{align}
and
	\begin{align}
		J_{42}
		=&\,\hat{\mathbb{E}}\bigg[\sup\limits_{s \in [0,T]}\hat{\mathbb{E}}_{s}\bigg[\int\nolimits_{0}^{T} C\Big(1+\big\vert \widetilde{X}^{\varepsilon,t_{1},0,x}_{u}\big\vert^{2m+2}\Big)du+CT\bigg]\bigg]\notag\\\notag
		\leq
		&\,\hat{\mathbb{E}}\bigg[\sup\limits_{s \in [0,T]}\hat{\mathbb{E}}_{s}\bigg[\int\nolimits_{s}^{T} C\Big(1+\big\vert \widetilde{X}^{\varepsilon,t_{1},0,x}_{u}\big\vert^{2m+2}\Big)du\bigg]\bigg]
		+\hat{\mathbb{E}}\bigg[\int\nolimits_{0}^{T} C\Big(1+\big\vert \widetilde{X}^{\varepsilon,t_{1},0,x}_{u}\big\vert^{2m+2}\Big)du\bigg]+CT\\\notag
		\leq
		&\,C\hat{\mathbb{E}}\bigg[\sup\limits_{s \in [0,T]}\hat{\mathbb{E}}_{s}\bigg[\sup\limits_{u \in [s,T]} \Big(1+\big\vert \widetilde{X}^{\varepsilon,t_{1},0,x}_{u}\big\vert^{2m+2}\Big)\bigg]\bigg]
		+C\hat{\mathbb{E}}\bigg[\sup\limits_{u \in [0,T]} \Big(1+\big\vert \widetilde{X}^{\varepsilon,t_{1},0,x}_{u}\big\vert^{2m+2}\Big)\bigg]
		+CT\notag\\\notag
 	    \leq
 	    &\,C\hat{\mathbb{E}}\bigg[\sup\limits_{s \in [0,T]}\hat{\mathbb{E}}_{s}\bigg[\sup\limits_{u \in [s,T]} \Big(1+\Big\vert \widetilde{X}^{\varepsilon,t_{1},s,\widetilde{X}^{\varepsilon,t_{1},0,x}_{s}}_{u}\Big\vert^{2m+2}\Big)\bigg]\bigg]
    	+C(1+\vert x\vert^{2m+2})
 	    +CT\notag\\\notag
		\leq
		&\,C\hat{\mathbb{E}}\bigg[\sup\limits_{s \in [0,T]} \Big(1+\big\vert \widetilde{X}^{\varepsilon,t_{1},0,x}_{s}\big\vert^{2m+2}\Big)\bigg]
		+C(1+\vert x\vert^{2m+2})
		+CT\\\notag
		\leq
		&\,C(1+\vert x\vert^{2m+2}).
	\end{align}
Thus, it is clear that
	\begin{align}\label{(4.12)}
		J_{4} \leq J_{41}+J_{42} \leq C(1+\vert x\vert^{2m+2}).
	\end{align}
Similarly, $J_{5} \leq C(1+\vert x\vert^{2m+2})$ also holds.
To sum up, we reorganize the above estimates to substitute into \eqref{(4.8)}, such that
\begin{align}
		\big\vert u^{\varepsilon}(t_{1},x)-u^{\varepsilon}(t_{2},x) \big\vert^{2}
		\leq 
		&\,J_{1}+J_{2}+J_{3}(J_{4}+J_{5})^{\frac{1}{2}}\notag\\\notag
		\leq
		&\,C(1+\vert x\vert^{2m+2})\vert t_{2}-t_{1}\vert
		+C(1+\vert x\vert^{m+2})\vert t_{2}-t_{1}\vert^{\frac{1}{2}},
\end{align}
where the constant $C$ depends on $T,L,c$ and $\underline{\sigma}$. Thus, we obtain the desired result.
\end{proof}\par
\begin{lemma}\label{lemma 4.5}
	Suppose that $\mathrm{(H1)}$-$\mathrm{(H2)}$ and $\mathrm{(H4)}$ hold. Then for any $s,t \in [0,T]$ and $x,\,x' \in \mathbb{R}^{n}\mathrm{,}$ we have\\
	${}\hspace{0.25em}$$\bm{(\mathbf{i})}$
	$\vert u^{\varepsilon}(t,x)-u^{\varepsilon}(t,x')\vert
	\leq C(1+\vert x \vert^{m\vee 2}+\vert x' \vert^{m\vee 2}) \vert x-x' \vert$$\mathrm{;}$\\
	$\bm{(\mathbf{ii})}$
	$\vert u^{\varepsilon}(t,x)-u^{\varepsilon}(s,x)\vert
	\leq C[(1+\vert x \vert^{(m+1)\vee 3}) \vert t-s \vert^{\frac{1}{2}}+\vert S(t,x)-S(s,x)\vert]$$\mathrm{;}$\\
	where the constant $C$ depends on $T,L,c$ and $\underline{\sigma}$.
\end{lemma}
\begin{proof}
	For the proof of this lemma, one can refer to Lemmas 6.4 and 6.6 in \cite{LPS}.
\end{proof}
\begin{lemma}\label{lemma 4.6}
	Suppose that $\mathrm{(H1)}$-$\mathrm{(H3)}$ hold. Then there exists a sequence $\varepsilon _{k} \downarrow 0\mathrm{,}$ $k \geq 1$ such that $u^{\varepsilon _{k}}$ converges to $u^{\ast} \in C([0,T] \times \mathbb{R}^{n})$. Moreover$\mathrm{,}$ for any $t,\,t_{1},\,t_{2} \in [0,T]$ and $x,\,x_{1},\,x_{2} \in \mathbb{R}^{n}$$\mathrm{,}$ we deduce the following estimates:\\
	${}\hspace{0.66em}$$\bm{(\mathbf{i})}$
	$\vert u^{\ast}(t,x_{1})-u^{\ast}(t,x_{2})\vert^{2} 
	\leq C[(1+\vert x_{1} \vert^{2m}+\vert x_{2} \vert^{2m}) \vert x_{1}-x_{2} \vert^{2}+(1+\vert x_{1} \vert^{m+1}+\vert x_{2} \vert^{m+1}) \vert x_{1}-x_{2} \vert]\mathrm{;}$\\
	${}\hspace{0.33em}$$\bm{(\mathbf{ii})}$
	$\vert u^{\ast}(t_{2},x)-u^{\ast}(t_{1},x)\vert 
	\leq C(1+\vert x \vert^{m+1}) \vert t_{2}-t_{1} \vert^{\frac{1}{2}}+C(1+\vert x \vert^{\frac{m+2}{2}}) \vert t_{2}-t_{1} \vert^{\frac{1}{4}}\mathrm{;}$\\
	$\bm{(\mathbf{iii})}$
	$\vert u^{\ast}(t,x)\vert\leq C(1+\vert x\vert^{m+1})\mathrm{;}$\\
	where the positive constant $C$ depends on $T,L,c$ and $\underline{\sigma}$.
\end{lemma}
\begin{proof}
	By the estimates in Lemma \ref{lemma 4.4}, we can directly deduce that $u^{\varepsilon }$ is uniformly bounded and equi-continuity on each compact subset of $[0,T] \times \mathbb{R}^{n}$. According to the Arzel\`{a}-Ascoli theorem, there exists a subsequence $\varepsilon _{k} \downarrow 0$, such that $u^{\varepsilon _{k}}$ uniformly converges to $u^{\ast} \in C([0,T] \times \mathbb{R}^{n})$ on each compact subset of $[0,T] \times \mathbb{R}^{n}$. Thus, it is not hard to check that $u^{\ast}$ has the above estimates. The proof is completed.
\end{proof}\par
Next, we present the main theorem in this paper.
\begin{theorem}\label{Theorem 4.7}
	Suppose that $\mathrm{(H1)}$-$\mathrm{(H3)}$ and $\mathrm{(H5)}$ hold. Then the averaged PDE \eqref{(4.1)} admits a unique viscos-ity solution $\bar{u}$
	such that for each $(t,x) \in [0,T] \times \mathbb{R}^{n}\mathrm{,}$
	\begin{align}\label{(4.13)}
	     \lim\limits_{\varepsilon \rightarrow 0} u^{\varepsilon}(t,x)=\bar{u}(t,x).
	\end{align}
\end{theorem}\par
Before giving the proof of Theorem \ref{Theorem 4.7}, we conclude the following Lemmas \ref{(4.8)}-\ref{(4.10)}.

\begin{lemma}\label{Lemma 4.8}
	Suppose that $\mathrm{(H1)}$-$\mathrm{(H3)}$ hold. Then for any $s \in [0,T]$ and $\xi \in L^{2m+3}_{G}(\Omega_{s};\mathbb{R}^{n})\mathrm{,}$ we have 
	\begin{align}\label{(4.14)}
		u^{\ast}(s,\xi)\geq S(s,\xi).
	\end{align}
\end{lemma}
\begin{proof}
	Inspired by the proof of Theorem 5.3.5 in Peng \cite{Peng3}, the main procedures will conducted in the following four steps.\par
	$\mathbf{Step\ 1.}$ In this step we assume that $\xi$ is a constant, i.e., $\xi=x \in \mathbb{R}^{n}$.\par
	Due to \eqref{(3.10)} and the structure of the reflected $G$-BSDE in \eqref{(3.1)}, for any $\varepsilon \in (0,1)$, it is not hard to check that  
	\begin{align}
		u^{\varepsilon}(s,x) = Y^{\varepsilon,s,x}_{s} \geq S(s,x).\notag
	\end{align}
And also, owing to Lemma \ref{lemma 4.6}, for the point $(s,x)$, there always exists a subsequence of $(u^{\varepsilon})_{\varepsilon \in (0,1)}$, such that $u^{\varepsilon_{j}}(s,x) \rightarrow u^{\ast}(s,x)$ as $j \rightarrow \infty$, which implies $u^{\ast}(s,x) \geq S(s,x)$.\par
	
$\mathbf{Step\ 2.}$ In this step we assume that $\xi=\sum\limits_{i=1}^{N} x_{i} \mathbb{I}_{A_i} \in Lip(\Omega_{s};\mathbb{R}^{n})$, where the constant $x_i \in \mathbb{R}^{n}$ holds for any $i=1,2,\ldots,N$ and $(A_i)_{i=1}^N$ is a $\mathcal{B}(\Omega_{s})$-partition.\par
Note that $ u^{\ast}(s,x_i) \geq S(s,x_i)$ for any $i=1,2,\ldots,N$, we can derive that 
\begin{align}
		 u^{\ast}(s,\xi)
		 =\sum\limits_{i=1}^{N} u^{\ast}(s, x_{i}) \mathbb{I}_{A_i}
		 \geq \sum\limits_{i=1}^{N} S(s,x_i) \mathbb{I}_{A_i}
		 \geq S(s,\xi).\notag
\end{align}\par

$\mathbf{Step\ 3.}$ In this step we assume that the random variable $\xi$ is bounded by some constant $\rho$ and $\xi \in Lip(\Omega_{s};\mathbb{R}^{n})$.\par
Without loss of generality, we can establish a sequence $\{\xi_n \}_{n\geq 1}$ of simple random variables by using $\xi$. More precisely, for each $n \geq 1$, let
\begin{align}
		\xi_n=\sum\limits_{i=-n}^{n} x_i \mathbb{I}_{A_i}(\xi),\notag
\end{align}
with $x_i=\frac{i\rho}{n}$, $A_i =\small\left[\frac{i\rho}{n},\frac{(i+1)\rho}{n}\small\right)$ for $i=-n,-n+1,\ldots,n-1$ and $x_n=\rho$, $A_n=\{\rho \}$, which can imply that $\vert \xi_n-\xi \vert \leq \frac{\rho}{n} $. Note that for each $n \geq 1$, $u^{\ast}(s,\xi_n) \geq S(s,\xi_n)$ holds, by assumption (H3), we get 
\begin{align}
\hat{\mathbb{E}}[\vert S(s,\xi_n)-S(s,\xi)\vert^{2}] \notag
\leq 
L^2\hat{\mathbb{E}}[\vert \xi_n-\xi \vert^{2}]
\leq
\frac{L^2\rho^2}{n^2}.
\end{align}
Given that $\xi_n$ and $\xi$ are bounded by $\rho$, together with Lemma \ref{lemma 4.6}, we have
	\begin{align}
	&\,\hat{\mathbb{E}}[\vert u^{\ast}(s,\xi_n) -u^{\ast}(s,\xi)\vert^{2}] \notag\\\notag
	\leq
	&\,C\hat{\mathbb{E}}[(1+\vert \xi_n \vert^{2m}+\vert \xi\vert^{2m})\vert \xi_n-\xi \vert^{2}]
	+C\hat{\mathbb{E}}[(1+\vert \xi_n \vert^{m+1}+\vert \xi\vert^{m+1})\vert \xi_n-\xi] \vert\notag\\\notag
	\leq
	&\,C(1+\rho^{2m})\frac{\rho^{2}}{n^2}
	+C(1+\rho^{m+1})\frac{\rho}{n},
	\end{align}
where the constant $C$ depends on $L,T,c,\underline{\sigma}$.
Since the constant $n$ is an arbitrary constant, using $u^{\ast}(s,\xi_n) \geq S(s,\xi_n)$ and letting $n\rightarrow \infty$ can yield that $u^{\ast}(s,\xi) \geq S(s,\xi)$ in $L^{2}_{G}(\Omega_{s};\mathbb{R}^{n})$.\par
	
$\mathbf{Step\ 4.}$ In this step we assume that the random variable $\xi \in L^{2m+3}_{G}(\Omega_{s};\mathbb{R}^{n})$.\par
With the help of Exercise 3.10.4 in Peng \cite{Peng3}, there exists a bounded sequence of random variables $\xi_k \in Lip(\Omega_{s};\mathbb{R}^{n})$, such that $\hat{\mathbb{E}}[\vert \xi_k -\xi\vert^{2m+3}] \rightarrow 0$ as $k$ tends to $0$. Note that $u^{\ast}(s,\xi_k) \geq S(s,\xi_k)$ holds for any $k \geq 1$, the aim of the following procedures is to verify that $u^{\ast}(s,\xi_k) \rightarrow u^{\ast}(s,\xi)$ and $S(s,\xi_k) \rightarrow S(s,\xi)$ hold in $L^{2}_{G}(\Omega_{s};\mathbb{R}^{n})$ when $k$ tends to $\infty$.\par
Firstly, for an arbitrary constant $N \,\textgreater\, 0$, by Lemma \ref{lemma 4.6}, we get
	\begin{align}
		&\,\hat{\mathbb{E}}\big[\vert u^{\ast}(s,\xi_k)-u^{\ast}(s,\xi) \vert^{2}\big]\notag\\\notag
		\leq
		&\,C\hat{\mathbb{E}}\big[\vert u^{\ast}(s,\xi_k)-u^{\ast}(s,\xi) \vert^{2}\,\mathbb{I}_{\{ \vert \xi_k \vert \vee \vert \xi \vert \leq N\}}\big]
		+C\hat{\mathbb{E}}\big[\vert u^{\ast}(s,\xi_k)-u^{\ast}(s,\xi) \vert^{2}\,\mathbb{I}_{\{ \vert \xi_k \vert \vee \vert \xi \vert \,\textgreater\, N\}}\big]\notag
	\end{align}
\begin{align}
		\leq
		&\,C\hat{\mathbb{E}}\big[(1+N^{2m})\vert \xi_k-\xi\vert^{2} \,\mathbb{I}_{\{ \vert \xi_k \vert \vee \vert \xi \vert \leq N\}}\big]
		+C\hat{\mathbb{E}}\big[(1+N^{m+1})\vert \xi_k-\xi\vert \,\mathbb{I}_{\{ \vert \xi_k \vert \vee \vert \xi \vert \leq N\}}\big]\notag\\\notag
		&+C\hat{\mathbb{E}}\big[(1+\vert \xi_k\vert^{2m+2}+\vert \xi\vert^{2m+2})\,\mathbb{I}_{\{ \vert \xi_k \vert \vee \vert \xi \vert \,\textgreater\, N\}}\big]\\ \notag
		\leq
		&\,C(1+N^{2m})\hat{\mathbb{E}}\big[\vert \xi_k-\xi\vert^{2} \big]
		+C(1+N^{m+1})\hat{\mathbb{E}}\big[\vert \xi_k-\xi\vert \big]
		+\frac{C}{N}\hat{\mathbb{E}}\big[(1+\vert \xi_k-\xi\vert^{2m+3}+\vert \xi\vert^{2m+3})\big],
	\end{align}
where the constant $C$ depends on $L,T,c,\underline{\sigma}$. Then, letting $k \rightarrow \infty$, we deduce that
	\begin{align}
	\limsup_{k \rightarrow \infty}\hat{\mathbb{E}}\big[\vert u^{\ast}(s,\xi_k)-u^{\ast}(s,\xi) \vert^{2}\big] \notag
	\leq
	\frac{C}{N}\hat{\mathbb{E}}\big[(1+\vert \xi\vert^{2m+3})\big].
    \end{align}
Since the constant $N$ can be arbitrary large, sending $N \rightarrow \infty$, we have
	\begin{align}\label{(4.15)}
		\lim_{k \rightarrow \infty}\hat{\mathbb{E}}\big[\vert u^{\ast}(s,\xi_k)-u^{\ast}(s,\xi) \vert^{2}\big]=0. 
	\end{align}\par
Next, from assumption (H3), it is not difficult to prove that
	\begin{align}\label{(4.16)}
	\lim\limits_{k\rightarrow \infty}\hat{\mathbb{E}}\big[\vert S(s,\xi_k)-S(s,\xi) \vert^{2}\big]
	\leq
	\lim\limits_{k\rightarrow \infty} L^{2}\hat{\mathbb{E}}\big[\vert \xi_k-\xi\vert^{2}\big]
	=0.
    \end{align}
Consequently, \eqref{(4.15)} and \eqref{(4.16)} together with $u^{\ast}(s,\xi_k) \geq S(s,\xi_k)$ can imply the desired result.
\end{proof}\par
Now we start to introduce the so-called stochastic ``backward semigroup'' formulated by Peng \cite{Peng4}, which plays an important role in handling with the problem regarding the viscosity solution of the fully nonlinear PDE \eqref{(4.1)}. For each $(t,x) \in [0,T] \times \mathbb{R}^{n}$, positive real number $\delta \leq T-t$, and $\eta \in L^{p}_{G}(\Omega_{t+\delta})$ with $p \,\textgreater\, 2$, we define
the following stochastic backward semigroup:
\begin{align}\label{(4.17)}
		\mathbb{G}^{\varepsilon,t,x}_{t,t+\delta}[\eta]
		\coloneqq
		Y^{\varepsilon,t,x,\delta}_{t},
\end{align}
where $(Y^{\varepsilon,t,x,\delta},Z^{\varepsilon,t,x,\delta},A^{\varepsilon,t,x,\delta}) \in \mathcal{S}^{2}_{G}(t,t+\delta)$ is the solution of the following
reflected $G$-BSDE with obstacle process $S(s,X^{\varepsilon,t,x}_{s})$:
\begin{equation}\label{(4.18)}
	\begin{aligned}
		Y^{\varepsilon,t,x,\delta}_{s}=&\,\eta+\int_s^{t+\delta} f\Big(\frac{r}{\varepsilon},X^{\varepsilon,t,x}_{r},Y^{\varepsilon,t,x,\delta}_{r},Z^{\varepsilon,t,x,\delta}_{r}\Big)dr
		-\int_s^{t+\delta}Z^{\varepsilon,t,x,\delta}_{r}dB_{r}
		+\big(A^{\varepsilon,t,x,\delta}_{t+\delta}-A^{\varepsilon,t,x,\delta}_{s}\big)\\
		&+\sum\limits_{i,j=1}^d\int_s^{t+\delta} g_{ij}\Big(\frac{r}{\varepsilon},X^{\varepsilon,t,x}_{r},Y^{\varepsilon,t,x,\delta}_{r},Z^{\varepsilon,t,x,\delta}_{r}\Big)d\langle B^{i},B^{j}\rangle_{r}.
    \end{aligned}
\end{equation}\\[-25.68bp]
\begin{lemma}\label{Lemma 4.9}
	Suppose that $\mathrm{(H1)}$-$\mathrm{(H3)}$ hold. Then there exists a sequence $ \varepsilon_{k} \downarrow 0\mathrm{,}$ such that when $u^{\varepsilon_{k}} \rightarrow u^{\ast} \in C([0,T] \times \mathbb{R}^{n})\mathrm{,}$ we can deduce that
		\begin{align}\label{(4.19)}
			\lim\limits_{k \rightarrow \infty} \Big\vert \mathbb{G}^{\varepsilon_k,t,x}_{t,t+\delta}\big[u^{\varepsilon_{k}}(t+\delta,X^{\varepsilon_{k},t,x}_{t+\delta})\big]-\mathbb{G}^{\varepsilon_k,t,x}_{t,t+\delta}\big[u^{\ast}(t+\delta,X^{\varepsilon_{k},t,x}_{t+\delta})\big] \Big\vert =0.	
	    \end{align}
\end{lemma}
\begin{proof}
	From Lemma \ref{lemma 4.6}, we can directly get that there exists a sequence $ \varepsilon_{k}$ such that $u^{\varepsilon_{k}}$ converges to $u^{\ast} \in C([0,T] \times \mathbb{R}^{n})$ as $ \varepsilon_{k} \downarrow 0$.
Owing to Lemmas \ref{lemma 3.1} and \ref{Lemma 4.8}, it is clear that $X^{\varepsilon_{k},t,x}_{t+\delta}\in L^{p}_{G}(\Omega_{t+\delta})$ holds for any $p \geq 2$, which implies $u^{\ast}(t+\delta,X^{\varepsilon_{k},t,x}_{t+\delta})\geq S(t+\delta,X^{\varepsilon_{k},t,x}_{t+\delta})$. Furthermore, from the structure of the reflected $G$-BSDE with obstacle process $S(s,X^{\varepsilon,t,x}_{s})$ in \eqref{(3.1)}, we obtain that $u^{\varepsilon_{k}}(t+\delta,X^{\varepsilon_k,t,x}_{t+\delta}) \geq S(t+\delta,X^{\varepsilon_{k},t,x}_{t+\delta})$.
Only in this way can we ensure that the terminal values $u^{\varepsilon_{k}}(t+\delta,X^{\varepsilon_k,t,x}_{t+\delta})$ and $u^{\ast}(t+\delta,X^{\varepsilon_{k},t,x}_{t+\delta})$ correspond to the restriction of the lower obstacle.
Thus, by \eqref{(4.18)} and Lemma \ref{lemma 5.5} in appendix, we have
	\begin{equation}\label{(4.20)}
		\begin{aligned}
			&\,\Big\vert \mathbb{G}^{\varepsilon_k,t,x}_{t,t+\delta}\big[u^{\varepsilon_{k}}(t+\delta,X^{\varepsilon_{k},t,x}_{t+\delta})\big]
			-\mathbb{G}^{\varepsilon_k,t,x}_{t,t+\delta}\big[u^{\ast}(t+\delta,X^{\varepsilon_{k},t,x}_{t+\delta})\big] \Big\vert\\
			\leq
			&\,C(L,T,c,\underline{\sigma})\Big(\hat{\mathbb{E}}_{t}\Big[\big\vert u^{\varepsilon_{k}}(t+\delta,X^{\varepsilon_{k},t,x}_{t+\delta})- u^{\ast}(t+\delta,X^{\varepsilon_{k},t,x}_{t+\delta})\big\vert^{2}\Big]\Big)^{\frac{1}{2}}.		
	    \end{aligned}
    \end{equation}
For each $N \,\textgreater\, 0$, by Lemmas \ref{lemma 3.1}, \ref{lemma 4.4} and \ref{lemma 4.6}, we deduce that
	\begin{align}
	&\,\hat{\mathbb{E}}_{t}\Big[\big\vert u^{\varepsilon_{k}}(t+\delta,X^{\varepsilon_{k},t,x}_{t+\delta})- u^{\ast}(t+\delta,X^{\varepsilon_{k},t,x}_{t+\delta})\big\vert^{2}\Big]\notag\\\notag
	\leq
	&\,C\hat{\mathbb{E}}_{t}\bigg[\big\vert u^{\varepsilon_{k}}(t+\delta,X^{\varepsilon_{k},t,x}_{t+\delta})- u^{\ast}(t+\delta,X^{\varepsilon_{k},t,x}_{t+\delta})\big\vert^{2}\,\mathbb{I}_{\big\{\vert X^{\varepsilon_{k},t,x}_{t+\delta} \vert \leq N\big\}}\bigg]\\\notag
	&+C\hat{\mathbb{E}}_{t}\bigg[\big(1+\big\vert X^{\varepsilon_{k},t,x}_{t+\delta}\big\vert^{2m+2}\big)\,\mathbb{I}_{\big\{\vert X^{\varepsilon_{k},t,x}_{t+\delta} \vert \geq N\big\}}\bigg]\\\notag
	\leq
	&\sup_{s \in [0,T] \atop \vert x \vert \leq N }\vert u^{\varepsilon_{k}}(s,x)-u^{\ast}(s,x) \vert^{2}
	+\frac{C}{N}\hat{\mathbb{E}}_{t}\Big[1+\big\vert X^{\varepsilon_{k},t,x}_{t+\delta}\big\vert^{2m+3}\Big]\\\notag
	\leq
	&\sup_{s \in [0,T] \atop \vert x \vert \leq N }\vert u^{\varepsilon_{k}}(s,x)-u^{\ast}(s,x) \vert^{2}
	+\frac{C}{N}(1+\vert x \vert^{2m+3}),
    \end{align}
where the constant $C$ depends on $L,T,c,\underline{\sigma}$.
Thus, letting $k \rightarrow \infty$, we have
\begin{align}
     \limsup_{k \rightarrow \infty}	\hat{\mathbb{E}}_{t}\Big[\big\vert u^{\varepsilon_{k}}(t+\delta,X^{\varepsilon_{k},t,x}_{t+\delta})- u^{\ast}(t+\delta,X^{\varepsilon_{k},t,x}_{t+\delta})\big\vert^{2}\Big]
     \leq
     \frac{C}{N}(1+\vert x \vert^{2m+3}).\notag
 \end{align}
Consequently, it is easy to check that
\begin{align}
\lim\limits_{k \rightarrow \infty}	\hat{\mathbb{E}}_{t}\Big[\big\vert u^{\varepsilon_{k}}(t+\delta,X^{\varepsilon_{k},t,x}_{t+\delta})- u^{\ast}(t+\delta,X^{\varepsilon_{k},t,x}_{t+\delta})\big\vert^{2}\Big]=0,\notag
\end{align}
 as $N$ tends to $\infty$, which together with \eqref{(4.20)} yields the desired result.
\end{proof}\par
Owing to the restriction of the lower obstacle, it will require us to consider whether the terminal state satisfies the lower obstacle of the given stochastic backward semigroup. But in reality, in many cases we can not directly guarantee the validity of this terminal condition based on existing results. To this end, the main purpose of the following discussions is to eliminate the influence of this obstacle process by formulating two auxiliary stochastic backward semigroups. More precisely, we can analyze the properties of auxiliary stochastic backward semigroups to obtain the relevant information regarding the original equation \eqref{(4.18)}.\par
Now we shall present two types of auxiliary stochastic backward semigroups, which are based on the general backward stochastic differential equations.
For each $(t,x) \in [0,T] \times \mathbb{R}^{n}$, positive real number $\delta \leq T-t$, $n \geq 1$ and $\eta_{1},\eta_{2} \in L^{p}_{G}(\Omega_{t+\delta})$ with $p \,\textgreater\, 2$, we define the following stochastic backward semigroups:	\vspace{4pt}
\begin{align}\label{(4.21)}
		\underline{\mathbb{G}}^{\varepsilon,t,x}_{t,t+\delta}[\eta_{1}]
		\coloneqq
		\underline{Y}^{\varepsilon,t,x,\delta}_{t},
		\quad
		\underline{\mathbb{G}}^{n,\varepsilon,t,x}_{t,t+\delta}[\eta_2]
		\coloneqq
		\underline{Y}^{n,\varepsilon,t,x,\delta}_{t},\vspace{4.75bp}
\end{align}
where $(\underline{Y}^{\varepsilon,t,x,\delta},\underline{Z}^{\varepsilon,t,x,\delta},\underline{K}^{\varepsilon,t,x,\delta}) \in \mathfrak{S}^{2}_{G}(t,t+\delta)$ and $(\underline{Y}^{n,\varepsilon,t,x,\delta},\underline{Z}^{n,\varepsilon,t,x,\delta},\underline{K}^{n,\varepsilon,t,x,\delta}) \in \mathfrak{S}^{2}_{G}(t,t+\delta)$ are the unique solution of the following $G$-BSDEs on the time interval $[t,t+\delta]$, respectively:
\begin{equation}
	\begin{aligned}\label{(4.22)}
		\underline{Y}^{\varepsilon,t,x,\delta}_{s}=&\,\eta_{1}+\int_s^{t+\delta} f\Big(\frac{r}{\varepsilon},X^{\varepsilon,t,x}_{r},\underline{Y}^{\varepsilon,t,x,\delta}_{r},\underline{Z}^{\varepsilon,t,x,\delta}_{r}\Big)dr
		-\int_s^{t+\delta}\underline{Z}^{\varepsilon,t,x,\delta}_{r}dB_{r}
		-\big(\underline{K}^{\varepsilon,t,x,\delta}_{t+\delta}-\underline{K}^{\varepsilon,t,x,\delta}_{s}\big)\\
		&+\sum\limits_{i,j=1}^d\int_s^{t+\delta} g_{ij}\Big(\frac{r}{\varepsilon},X^{\varepsilon,t,x}_{r},\underline{Y}^{\varepsilon,t,x,\delta}_{r},Z^{\varepsilon,t,x,\delta}_{r}\Big)d\langle B^{i},B^{j}\rangle_{r},
		\end{aligned}
	\end{equation}
\begin{equation}
	\begin{aligned}\label{(4.23)}
		\underline{Y}^{n,\varepsilon,t,x,\delta}_{s}=&\,\eta_{2}+\int_s^{t+\delta} f\Big(\frac{r}{\varepsilon},X^{\varepsilon,t,x}_{r},\underline{Y}^{n,\varepsilon,t,x,\delta}_{r},\underline{Z}^{n,\varepsilon,t,x,\delta}_{r}\Big)dr
		+\int_s^{t+\delta} n\big(\underline{Y}^{n,\varepsilon,t,x,\delta}_{r}-S(r,X^{\varepsilon,t,x}_{r})\big)^{-}dr\\
		&+\sum\limits_{i,j=1}^d\int_s^{t+\delta} g_{ij}\Big(\frac{r}{\varepsilon},X^{\varepsilon,t,x}_{r},\underline{Y}^{n,\varepsilon,t,x,\delta}_{r},\underline{Z}^{n,\varepsilon,t,x,\delta}_{r}\Big)d\langle B^{i},B^{j}\rangle_{r}
		-\int_s^{t+\delta}\underline{Z}^{n,\varepsilon,t,x,\delta}_{r}dB_{r}\\
		&-\big(\underline{K}^{n,\varepsilon,t,x,\delta}_{t+\delta}-\underline{K}^{n,\varepsilon,t,x,\delta}_{s}\big).\\[3pt]
	\end{aligned}
\end{equation}
Note that the structures of the above $G$-BSDEs in \eqref{(4.22)} and \eqref{(4.23)}, by the comparison theorem of $G$-BSDEs in \cite{HJP2}, it is clear that for each $n \geq 1$ and $\eta \in L^{p}_{G}(\Omega_{t+\delta})$ with $p \,\textgreater\, 2$
\begin{align}\label{(4.24)}
	\underline{\mathbb{G}}^{\varepsilon,t,x}_{t,t+\delta}[\eta]
	\leq
	\underline{\mathbb{G}}^{n,\varepsilon,t,x}_{t,t+\delta}[\eta].
\end{align}
In particular, when $\eta$ also satisfies $\eta \geq S(t+\delta,X^{\varepsilon,t,x}_{t+\delta})$, by applying the approximation method via penalization (see \cite{LPS}), we can deduce that as $n \rightarrow \infty$
\begin{align}\label{(4.25)}
	\underline{\mathbb{G}}^{n,\varepsilon,t,x}_{t,t+\delta}[\eta]
	\uparrow
	\mathbb{G}^{\varepsilon,t,x}_{t,t+\delta}[\eta].
\end{align}
Thus, we can immediately see that 
	\begin{align}\label{(4.26)}
		\underline{\mathbb{G}}^{\varepsilon,t,x}_{t,t+\delta}[\eta]
		\leq
		\mathbb{G}^{\varepsilon,t,x}_{t,t+\delta}[\eta]
    \end{align}
holds for the case $\eta \geq S(t+\delta,X^{\varepsilon,t,x}_{t+\delta}) $ and $\eta \in L^{p}_{G}(\Omega_{t+\delta})$ with $p \,\textgreater\, 2$, which establishes the relation between the stochastic backward semigroups \eqref{(4.17)} and \eqref{(4.21)}. Essentially, it is equivalent to transforming the stochastic backward semigroup \eqref{(4.17)} into the general backward semigroup \eqref{(4.21)}, which ensures that the terminal state is not constrained by the obstacle process, so we can adopt the original methods in \cite{HJW}  to handle with the existence of viscosity supersolution to \eqref{(4.1)}. 

\begin{lemma}\label{Lemma 4.10}
	Suppose that $\mathrm{(H1)}$-$\mathrm{(H3)}$ and $\mathrm{(H5)}$ hold. Then $u^{\ast}$ is the unique viscosity solution to the averaged PDE \eqref{(4.1)}.
\end{lemma}
\begin{proof}
The proof regarding the uniqueness of viscosity solution to \eqref{(4.1)} can refer to Lemma \ref{lemma 5.6} in appendix.
Given that the restriction of obstacle process, we need to conduct the discussions on the existence of the viscosity solution from two different perspectives.\par
	
$\mathbf{Step\ 1.}$ In this step we show that $u^{\ast}$ is a viscosity supersolution of PDE \eqref{(4.1)}. Without loss of generality, assume that $h_{ij}=g_{ij}=0,\,i,j=1,2,\ldots,d$, since the other cases can be discussed in a similar method. For each fixed $(t,x)\in[0,T] \times \mathbb{R}^{n}$, we assume that the test function $\phi \in C^{2,3}([0,T] \times \mathbb{R}^{n})$ satisfies that $\phi \leq u^{\ast}$, $\phi(t,x) = u^{\ast}(t,x)$ and all partial derivatives of order from 1 to 3 are bounded by $C(L,T)(1+\vert x \vert ^{m})$. \par
In fact, from Lemma \ref{Lemma 4.8}, $\phi(t,x)=u^{\ast}(t,x) \geq S(t,x)$ holds, so we only need to verify that
\begin{align}\label{(4.27)}
		-\partial_{t}\phi(t,x)-\bar{F}(x,\phi(t,x),D_{x}\phi(t,x),D^{2}_{x}\phi(t,x)) \geq 0.
\end{align}
Since we can not ensure that $\phi(t+\delta,X^{\varepsilon_{k},t,x}_{t+\delta}) \geq S(t+\delta,X^{\varepsilon_{k},t,x}_{t+\delta})$ is valid, the method in \cite{HJW} is no longer applicable. More precisely, the stochastic backward semigroup with the terminal value $\phi(t+\delta,X^{\varepsilon_{k},t,x}_{t+\delta})$ in \eqref{(4.17)} can not be utilized directly. Thus, we need to draw support from \eqref{(4.21)}-\eqref{(4.26)} to transform \eqref{(4.17)} into the second stochastic backward semigroup in \eqref{(4.21)}. The main procedures are divided into two parts.\par 
$\mathbf{Part\ 1}\,$(Dynamic programming principle). According to Lemma \ref{lemma 4.6}, there exists a subsequence $u^{\varepsilon_{k}}$, such that $u^{\varepsilon _{k}}$ uniformly converges to $u^{\ast} \in C([0,T] \times \mathbb{R}^{n})$ on each compact subset of $[0,T] \times \mathbb{R}^{n}$. Recalling Lemma \ref{lemma 3.3}, for each small enough $\delta \in (0,1)$ and $k \geq 1$, we have
	\begin{align}
		u^{\varepsilon_{k}}(t,x)\notag
		=
		\mathbb{G}^{\varepsilon_{k},t,x}_{t,t+\delta}\big[Y^{\varepsilon_{k},t,x}_{t+\delta}\big]
		=
		\mathbb{G}^{\varepsilon_{k},t,x}_{t,t+\delta}\big[u^{\varepsilon_{k}}(t+\delta,X^{\varepsilon_{k},t,x}_{t+\delta})\big].
	\end{align}
By means of Lemma \ref{Lemma 4.9}, it is easy to derive that
	\begin{align}\label{(4.28)}
		u^{\ast}(t,x)
		=
		\lim\limits_{k \rightarrow \infty}\mathbb{G}^{\varepsilon_{k},t,x}_{t,t+\delta}\big[u^{\ast}(t+\delta,X^{\varepsilon_{k},t,x}_{t+\delta})\big].
	\end{align}
Note that $\underline{\mathbb{G}}^{\varepsilon_{k},t,x}_{t,t+\delta}\big[u^{\ast}(t+\delta,X^{\varepsilon_{k},t,x}_{t+\delta})\big] \geq \underline{\mathbb{G}}^{\varepsilon_{k},t,x}_{t,t+\delta}\big[\phi(t+\delta,X^{\varepsilon_{k},t,x}_{t+\delta})\big]$, by \eqref{(4.26)} and Lemma \ref{Lemma 4.8}, we obtain that  
\begin{align}
	\mathbb{G}^{\varepsilon_{k},t,x}_{t,t+\delta}\big[u^{\ast}(t+\delta,X^{\varepsilon_{k},t,x}_{t+\delta})\big] 
	\geq \underline{\mathbb{G}}^{\varepsilon_{k},t,x}_{t,t+\delta}\big[u^{\ast}(t+\delta,X^{\varepsilon_{k},t,x}_{t+\delta})\big]\notag
	\geq
	\underline{\mathbb{G}}^{\varepsilon_{k},t,x}_{t,t+\delta}\big[\phi(t+\delta,X^{\varepsilon_{k},t,x}_{t+\delta})\big],
\end{align}
which together with \eqref{(4.28)} can imply that
\begin{align}\label{(4.29)}
		\phi(t,x)
		=u^{\ast}(t,x) 
		\geq 
		\limsup_{k \rightarrow \infty}\underline{\mathbb{G}}^{\varepsilon_{k},t,x}_{t,t+\delta}\big[\phi(t+\delta,X^{\varepsilon_{k},t,x}_{t+\delta})\big].
\end{align}\par
$\mathbf{Part\ 2}\,$(Approximating). For each $k \geq 1$, denote by $(\widetilde{Y}^{\varepsilon_{k}},\widetilde{Z}^{\varepsilon_{k}},\widetilde{K}^{\varepsilon_{k}})$ the solution of $G$-BSDE \eqref{(4.22)} with terminal condition $\phi(t+\delta,X^{\varepsilon_{k},t,x}_{t+\delta})$ on the time interval $[t,t+\delta]$. For each $s \in [t,t+\delta] $, we define 
	\begin{align}
		\widehat{Y}^{\varepsilon_{k}}_s=\widetilde{Y}^{\varepsilon_{k}}_s-\phi(s,X^{\varepsilon_{k},t,x}_{s}),\quad
		\widehat{Z}^{\varepsilon_{k}}_s=\widetilde{Z}^{\varepsilon_{k}}_s-D_x\phi(s,X^{\varepsilon_{k},t,x}_{s})\sigma(\tfrac{s}{\varepsilon_{k}},X^{\varepsilon_{k},t,x}_{s}),\quad
		\widehat{K}^{\varepsilon_{k}}_s=\widetilde{K}^{\varepsilon_{k}}_s.\notag
 \end{align}
Applying $G$-It\^o's formula to $\widetilde{Y}^{\varepsilon_{k}}_s-\phi(s,X^{\varepsilon_{k},t,x}_{s})$, we can derive that $(\widehat{Y}^{\varepsilon_{k}},\widehat{Z}^{\varepsilon_{k}},\widehat{K}^{\varepsilon_{k}})$ satisfies the following $G$-BSDE:
	\begin{align}
		\widehat{Y}^{\varepsilon_k}_{s}=&\,\int_s^{t+\delta}\Big[f^{\varepsilon_k}\big(r,X^{\varepsilon_k,t,x}_{r},\widehat{Y}^{\varepsilon_k}_{r},\widehat{Z}^{\varepsilon_k}_{r}\big)+\partial_{t}\phi\big(r,X^{\varepsilon_k,t,x}_{r}\big)+D_x\phi\big(r,X^{\varepsilon_k,t,x}_{r}\big)b^{\varepsilon_k}\big(r,X^{\varepsilon_k,t,x}_{r}\big)\Big]dr \notag\\\notag
        &-\int_s^{t+\delta}\widehat{Z}^{\varepsilon_k}_{r}dB_{r}
        +\frac{1}{2}\int_s^{t+\delta} \sigma^{\varepsilon_k}\big(r,X^{\varepsilon_k,t,x}_{r}\mathbf{\big)}^\top D^{2}_x\phi\big(r,X^{\varepsilon_k,t,x}_{r}\big) \sigma^{\varepsilon_k}\big(r,X^{\varepsilon_k,t,x}_{r}\big)d\langle B\rangle_{r}
        -\big(\widehat{K}^{\varepsilon_k}_{t+\delta}-\widehat{K}^{\varepsilon_k}_{s}\big),
  \end{align}
where $f^{\varepsilon_k}(r,x,y,z)=f(\frac{r}{\varepsilon_k},x,y+\phi(r,x),z+D_x\phi(r,x)\sigma(\frac{r}{\varepsilon_k},x))$, $b^{\varepsilon_k}(r,x)=b(\frac{r}{\varepsilon_k},x)$ and $\sigma^{\varepsilon_k}(r,x)=\sigma(\frac{r}{\varepsilon_k},x)$.
Then, we can establish the following approximating $G$-BSDE on the time interval $[t,t+\delta]$:
\begin{equation}
	\begin{aligned}
		\bar{Y}^{\varepsilon_k}_{s}=&\,\int_s^{t+\delta}\Big[\bar{f}^{\varepsilon_k}\big(r,x,\bar{Y}^{\varepsilon_k}_{r},\bar{Z}^{\varepsilon_k}_{r}\big)+\partial_{t}\phi(t,x)+D_x\phi(t,x)b^{\varepsilon_k}(r,x)\Big]dr \notag\\
		&-\int_s^{t+\delta}\bar{Z}^{\varepsilon_k}_{r}dB_{r}
		+\frac{1}{2}\int_s^{t+\delta} \sigma^{\varepsilon_k}(r,x\mathbf{)}^\top D^{2}_x\phi(t,x) \sigma^{\varepsilon_k}(r,x)d\langle B\rangle_{r}
		-(\bar{K}^{\varepsilon_k}_{t+\delta}-\bar{K}^{\varepsilon_k}_{s}),
	\end{aligned}
\end{equation}
where $\bar{f}^{\varepsilon_k}(r,x',y,z)=f(\frac{r}{\varepsilon_k},x',y+\phi(t,x),z+D_x\phi(t,x)\sigma(\frac{r}{\varepsilon_k},x'))$. The remaining procedures are similar to the proof of Lemma 4.6 in \cite{HJW}, so we omit it. Consequently, we have
	\begin{align}
		-\partial_{t}\phi(t,x)-\bar{F}(x,\phi(t,x),D_{x}\phi(t,x),D^{2}_{x}\phi(t,x)) \geq 0,\notag
	\end{align}
which directly implies $u^{\ast}$ is a viscosity supersolution of \eqref{(4.1)}.\par

$\mathbf{Step\ 2.}$ In this step we demonstrate that $u^{\ast}$ is a viscosity subsolution of \eqref{(4.1)}. We define a test function $\phi^{1}$ with similar settings to $\phi$. In fact, we assume that the test function $\phi^{1} \in C^{2,3}([0,T] \times \mathbb{R}^{n})$ satisfies that $\phi^1 \geq u^{\ast}$, $\phi^{1}(t,x) = u^{\ast}(t,x)$ and all partial derivatives of order from 1 to 3 are bounded by $C(L,T)(1+\vert x \vert ^{m})$. And also, we still consider the case where $h_{ij}=g_{ij}=0,\,i,j=1,2,\ldots,d$. The main proof is divided into the following three parts.\par
$\mathbf{Part\ 1\,}$(Dynamic programming principle). From Lemma \ref{Lemma 4.8}, we can obtain that $\phi^1(t+\delta,X^{\varepsilon_{k},t,x}_{t+\delta}) \geq S(t+\delta,X^{\varepsilon_{k},t,x}_{t+\delta})$. In addition, by a similar analysis as \eqref{(4.28)} and \eqref{(4.29)}, we can immediately conclude that
\begin{align}\label{(4.30)}
	\phi^{1}(t,x)
	 =u^{\ast}(t,x) 
	 \leq 
	 \limsup_{k \rightarrow \infty}\mathbb{G}^{\varepsilon_{k},t,x}_{t,t+\delta}\big[\phi^{1}(t+\delta,X^{\varepsilon_{k},t,x}_{t+\delta})\big].
 \end{align}\par
$\mathbf{Part\ 2}\,$(Approximating). For each $k \geq 1$, we denote by $(\widetilde{Y}^{1,\varepsilon_{k}},\widetilde{Z}^{1,\varepsilon_{k}},\widetilde{A}^{1,\varepsilon_{k}})$ the solution to the reflected $G$-BSDE \eqref{(4.18)} with terminal condition $\phi^{1}(t+\delta,X^{\varepsilon_{k},t,x}_{t+\delta})$ on the time interval $[t,t+\delta]$. For each $s \in [t,t+\delta] $, we define 
	\begin{align}
		\widehat{Y}^{1,\varepsilon_{k}}_s=\widetilde{Y}^{1,\varepsilon_{k}}_s-\phi^{1}(s,X^{\varepsilon_{k},t,x}_{s}),\ \,
		\widehat{Z}^{1,\varepsilon_{k}}_s=\widetilde{Z}^{1,\varepsilon_{k}}_s-D_x\phi^{1}(s,X^{\varepsilon_{k},t,x}_{s})\sigma(\tfrac{s}{\varepsilon_{k}},X^{\varepsilon_{k},t,x}_{s}),\ \,
		\widehat{A}^{1,\varepsilon_{k}}_s=\widetilde{A}^{1,\varepsilon_{k}}_s.\notag
	\end{align}
Applying $G$-It\^o's formula to $\widetilde{Y}^{1,\varepsilon_{k}}_s-\phi^{1}(s,X^{\varepsilon_{k},t,x}_{s})$, we can immediately see that $(\widehat{Y}^{1,\varepsilon_{k}},\widehat{Z}^{1,\varepsilon_{k}},\widehat{A}^{1,\varepsilon_{k}})$ satisfies the following reflected $G$-BSDE:
\begin{equation}\label{(4.31)}
	\begin{aligned}
		\widehat{Y}^{1,\varepsilon_k}_{s}=&\,\int_s^{t+\delta}\Big[f^{1,\varepsilon_k}\big(r,X^{\varepsilon_k,t,x}_{r},\widehat{Y}^{1,\varepsilon_k}_{r},\widehat{Z}^{1,\varepsilon_k}_{r}\big)+\partial_{t}\phi^{1}\big(r,X^{\varepsilon_k,t,x}_{r}\big)+D_x\phi^{1}\big(r,X^{\varepsilon_k,t,x}_{r}\big)b^{\varepsilon_k}\big(r,X^{\varepsilon_k,t,x}_{r}\big)\Big]dr \\
		&+\frac{1}{2}\int_s^{t+\delta} \sigma^{\varepsilon_k}\big(r,X^{\varepsilon_k,t,x}_{r}\mathbf{\big)}^\top D^{2}_x\phi^{1}\big(r,X^{\varepsilon_k,t,x}_{r}\big) \sigma^{\varepsilon_k}\big(r,X^{\varepsilon_k,t,x}_{r}\big)d\langle B\rangle_{r}\\
		&-\int_s^{t+\delta}\widehat{Z}^{1,\varepsilon_k}_{r}dB_{r}
		+\big(\widehat{A}^{1,\varepsilon_k}_{t+\delta}-\widehat{A}^{1,\varepsilon_k}_{s}\big),
	\end{aligned}
\end{equation}
where the coefficients $f^{1,\varepsilon_k}(r,x,y,z)=f(\frac{r}{\varepsilon_k},x,y+\phi^{1}(r,x),z+D_x\phi^{1}(r,x)\sigma(\frac{r}{\varepsilon_k},x))$, $b^{\varepsilon_k}(r,x)=b(\frac{r}{\varepsilon_k},x)$ and $\sigma^{\varepsilon_k}(r,x)=\sigma(\frac{r}{\varepsilon_k},x)$.
Next, we construct the following approximating reflected $G$-BSDE on the time interval $[t,t+\delta]$:
\begin{equation}\label{(4.32)}
	\begin{aligned}
		\bar{Y}^{1,\varepsilon_k}_{s}=&\,\int_s^{t+\delta}\Big[\bar{f}^{1,\varepsilon_k}\big(r,x,\bar{Y}^{1,\varepsilon_k}_{r},\bar{Z}^{1,\varepsilon_k}_{r}\big)+\partial_{t}\phi^{1}(t,x)+D_x\phi^{1}(t,x)b^{\varepsilon_k}(r,x)\Big]dr \\
		&-\int_s^{t+\delta}\bar{Z}^{1,\varepsilon_k}_{r}dB_{r}
		+\frac{1}{2}\int_s^{t+\delta} \sigma^{\varepsilon_k}(r,x\mathbf{)}^\top D^{2}_x\phi^{1}(t,x) \sigma^{\varepsilon_k}(r,x)d\langle B\rangle_{r}
		+(\bar{A}^{1,\varepsilon_k}_{t+\delta}-\bar{A}^{1,\varepsilon_k}_{s}),
	\end{aligned}
\end{equation}
where $\bar{f}^{1,\varepsilon_k}(r,x',y,z)=f(\frac{r}{\varepsilon_k},x',y+\phi^{1}(t,x),z+D_x\phi^{1}(t,x)\sigma(\frac{r}{\varepsilon_k},x'))$.\par
Under the structures of the two equations given above, the main purpose of the following discussions is to ensure that \eqref{(4.31)} and \eqref{(4.32)} correspond to the setting of the reflected $G$-BSDE. To be more precise, we need to establish two obstacle processes, which really belong to \eqref{(4.31)} and \eqref{(4.32)}, respectively. The specific structures of two lower obstacle processes provide a convenient and significant way to cope with estimation of the difference between the solutions of \eqref{(4.31)} and \eqref{(4.32)}.\par
As for $\widehat{Y}^{1,\varepsilon_k}$, by Lemma \ref{Lemma 4.8} and the definition of $u^{\ast}$, it is not hard to get that
\begin{align}\label{(4.33)}
	\phi^{1}(s,X^{\varepsilon_k,t,x}_{s})
	\geq
	u^{\ast}(s,X^{\varepsilon_k,t,x}_{s})
	\geq
	S(s,X^{\varepsilon_k,t,x}_{s}),\quad \forall s \in[t,t+\delta].
\end{align}
Thus we shall formulate the following new frameworks:
\begin{align}
		&\hspace{0.5em}S_{1}(s,x') \coloneqq S(s,x')-\phi^{1}(s,x'),\quad\forall (s,x') \in [t,t+\delta] \times \mathbb{R}^{n},\notag\\\notag
		L^{1,\varepsilon_k,t,x}_{s} &\coloneqq S_{1}(s,X^{\varepsilon_k,t,x}_{s})= S(s,X^{\varepsilon_k,t,x}_{s})-\phi^{1}(s,X^{\varepsilon_k,t,x}_{s}), \quad \forall s \in [t,t+\delta].
\end{align}\par
It is clear that $S_{1}(s,X^{\varepsilon_k,t,x}_{s}) \leq 0$ for any $s \in [t,t+\delta]$, which indicates that $0=\widehat{Y}^{1,\varepsilon_k}_{t+\delta} \geq S_{1}(t+\delta,X^{\varepsilon_k,t,x}_{t+\delta})$. In addition, note that $\{L^{1,\varepsilon_k,t,x}_{s}\}_{s \in [t,t+\delta]} \in S^{p}_{G}(t,t+\delta)$, by Lemma \ref{lemma 5.1} in appendix, so \eqref{(4.31)} can be considered as a reflected $G$-BSDE with obstacle process $\{L^{1,\varepsilon_k,t,x}_{s}\}_{s \in [t,t+\delta]}$. Moreover, \eqref{(4.31)} also satisfies the following statements:\\
${}\hspace{0.40em}$(i) $\widehat{Y}^{1,\varepsilon_k}_{s} \geq S_{1}(s,X^{\varepsilon_k,t,x}_{s})$ for any $s \in [t,t+\delta]$;\vspace{2pt}\\
${}\,$(ii) $\big\{-\int\nolimits_{t}^{s} (\widehat{Y}^{1,\varepsilon_k}_{r} - S_{1}(r,X^{\varepsilon_k,t,x}_{r})) d\widehat{A}^{1,\varepsilon_k}_{r} \big\}_{s \in [t,t+\delta]}$ is a nonincreasing $G$-martingale.\vspace{3pt}\par
In fact, the above statements can be directly obtained by using the structure of $\widetilde{Y}^{1,\varepsilon_k}_{r}$ and the reflected $G$-BSDE in \eqref{(4.18)}. To be more precise, $\{-\int\nolimits_{t}^{s} (\widetilde{Y}^{1,\varepsilon_k}_{r}- S(r,X^{\varepsilon_k,t,x}_{r})) d\widetilde{A}^{1,\varepsilon_k}_{r}\}_{s \in [t,t+\delta]}$ is a nonincreasing $G$-martingale. And also, it is clear that
\begin{align}
	\widehat{Y}^{1,\varepsilon_k}_{s}- S_{1}(s,X^{\varepsilon_k,t,x}_{s})
	&=
	\widetilde{Y}^{1,\varepsilon_k}_{s}- S(s,X^{\varepsilon_k,t,x}_{s})
	\geq 0,\notag\\
	-\int\nolimits_{t}^{s} (\widehat{Y}^{1,\varepsilon_k}_{r} - S_{1}(r,X^{\varepsilon_k,t,x}_{r})) d\widehat{A}^{1,\varepsilon_k}_{r}
	&=
	-\int\nolimits_{t}^{s} (\widetilde{Y}^{1,\varepsilon_k}_{r}- S(r,X^{\varepsilon_k,t,x}_{r})) d\widetilde{A}^{1,\varepsilon_k}_{r}.\notag
\end{align}
That is to say, if we consider the above statements by using the equation \eqref{(4.18)} instead of the information that \eqref{(4.31)} is a reflected $G$-BSDE, we can also ensure the correctness of the above statements (i)-(ii). Thus, no matter from which perspective we discuss it, it can indicate that the obstacle process $L^{1,\varepsilon_k,t,x}_{s}$ we established is really reasonable, the equation \eqref{(4.31)} is indeed a reflected $G$-BSDE with the obstacle process $S_{1}(s,X^{\varepsilon_k,t,x}_{s})$.\par Similarly, as for $\bar{Y}^{1,\varepsilon_k}$, fixing a point $x \in \mathbb{R}^{n}$, we set
\begin{align}
	L^{2,\varepsilon_k,t,x}_{s}
	\coloneqq
	S_1(s,x)
	=
	S(s,x)-\phi^{1}(s,x), \quad \forall s \in [t,t+\delta].\notag
\end{align}
In fact, $L^{2,\varepsilon_k,t,x}_{s}=S_{1}(s,x)$ is a deterministic and continuous function regarding the argument $s$. Furthermore, using $\phi^{1} \geq u^{\ast}$ and Lemma \ref{Lemma 4.8} can yield that 
\begin{equation}\label{(4.34)}
	\begin{aligned}
	\phi^{1}(s,x)
	&\geq
	u^{\ast}(s,x)
	\geq	 
	S(s,x), \quad \forall s \in [t,t+\delta],\\
	S_{1}(t+\delta,x)
	&=S(t+\delta,x)-\phi^{1}(t+\delta,x)
	\leq 0=\bar{Y}^{1,\varepsilon_k}_{t+\delta},
	\end{aligned}
\end{equation}
which together with Lemma \ref{lemma 5.1} in appendix can illustrate that \eqref{(4.32)} is recognized as a reflected $G$-BSDE with the obstacle process $\{L^{2,\varepsilon_k,t,x}_{s}\}_{s \in [t,t+\delta]}$. In addition, it is easy to obtain that\\
${}\hspace{0.40em}$(i) $\bar{Y}^{1,\varepsilon_k}_{s} \geq S_{1}(s,x)$ for any $s \in [t,t+\delta]$;\vspace{2pt}\\
${}\,$(ii) $\{-\int\nolimits_{t}^{s} (\bar{Y}^{1,\varepsilon_k}_{r} - S_{1}(r,x)) d\bar{A}^{1,\varepsilon_k}_{r} \}_{s \in [t,t+\delta]}$ is a nonincreasing $G$-martingale.\par

Next we present some estimates regarding the coefficients of \eqref{(4.31)} and \eqref{(4.32)}. By assumptions (H1)-(H2), one can check
that 
\begin{align}
	\Gamma^{1}_r=&\,\big\vert f^{1,\varepsilon_k}(r,X^{\varepsilon_k,t,x}_{r} ,\bar{Y}^{1,\varepsilon_k}_{r},\bar{Z}^{1,\varepsilon_k}_{r})
	-\bar{f}^{1,\varepsilon_k}(r,x,\bar{Y}^{1,\varepsilon_k}_{r},\bar{Z}^{1,\varepsilon_k}_{r})\big\vert \notag\\\label{(4.35)}
	\leq
	&\,C(L,T)(1+\vert x \vert^{m+1}+\vert X^{\varepsilon_k,t,x}_{r} \vert^{m+1})(\vert r-t \vert+\vert X^{\varepsilon_k,t,x}_{r}-x\vert),\\ 
	\Gamma^{2}_r=&\,\big\vert \partial_{t}\phi^{1}(r,X^{\varepsilon_k,t,x}_{r})
	+D_x \phi^{1}(r,X^{\varepsilon_k,t,x}_{r})b^{\varepsilon_k}(r,X^{\varepsilon_k,t,x}_{r})-\partial_{t}\phi^{1}(t,x)-D_x \phi^{1}(t,x)b^{\varepsilon_k}(r,x)\big\vert \notag\\\label{(4.36)}
	\leq
	&\,C(L,T)(1+\vert x \vert^{m+1}+\vert X^{\varepsilon_k,t,x}_{r} \vert^{m+1})(\vert r-t \vert+\vert X^{\varepsilon_k,t,x}_{r}-x\vert),\\
	\Gamma^{3}_r=&\,\big\vert \sigma^{\varepsilon_k}(r,X^{\varepsilon_k,t,x}_{r}\mathbf{)}^\top
	D^{2}_x \phi^{1}(r,X^{\varepsilon_k,t,x}_{r})\sigma^{\varepsilon_k}(r,X^{\varepsilon_k,t,x}_{r})
	-\sigma^{\varepsilon_k}(r,x\mathbf{)}^\top
	D^{2}_x \phi^{1}(t,x)\sigma^{\varepsilon_k}(r,x)\big\vert \notag\\\label{(4.37)}
	\leq
	&\,C(L,T)(1+\vert x \vert^{m+2}+\vert X^{\varepsilon_k,t,x}_{r} \vert^{m+2})(\vert r-t \vert+\vert X^{\varepsilon_k,t,x}_{r}-x\vert).
\end{align}
Thus, we can immediately derive that assumptions (H1)-(H2) in \cite{L} are weaker than those made for the coefficients in this paper. In particular, by \eqref{(4.33)}-\eqref{(4.34)}, one can choose $\xi=0$ and $I_{s}=0$ for each $s \in [t,t+\delta]$, which are proposed in assumption (H3) of \cite{L}. Then, the coefficients of \eqref{(4.31)} and \eqref{(4.32)} correspond to assumption (H3) in \cite{L}. Thus, by Lemma 3.5 in \cite{L}, for any $r \in [t,t+\delta]$, we obtain the following estimate:
\begin{equation}\label{(4.38)}
	\begin{aligned}
 	&\ \hat{\mathbb{E}}\Big[\big\vert \widehat{Y}^{1,\varepsilon_k}_{r}-\bar{Y}^{1,\varepsilon_k}_{r} \big\vert^{2}\Big]\\
	\leq
	&\,C\hat{\mathbb{E}}\bigg[\sup_{s \in [t,t+\delta]}\hat{\mathbb{E}}_{s}\bigg[ \bigg(\int\nolimits_{t}^{t+\delta} \big(\Gamma^{1}_r+\Gamma^{2}_r+\Gamma^{3}_r \big) dr \bigg)^{3}     \bigg]            \bigg]\\
	&+C\bigg(\hat{\mathbb{E}}\bigg[\sup_{s \in [t,t+\delta]}\hat{\mathbb{E}}_{s}\bigg[ \bigg(\int\nolimits_{t}^{t+\delta} \big(\Gamma^{1}_r+\Gamma^{2}_r+\Gamma^{3}_r \big) dr \bigg)^{3}     \bigg]            \bigg]\bigg)^{\frac{2}{3}},
    \end{aligned}
\end{equation}
where the constant $C$ depends on $T,\underline{\sigma},\bar{\sigma},L$. Moreover, we only need to deal with the first term on the right side of the inequality \eqref{(4.38)}, since the second term can be researched in a similar method. By \eqref{(4.35)}-\eqref{(4.37)}, we have
\begin{align}\label{(4.39)}
	{}\hspace{-2.5em}C\hat{\mathbb{E}}\bigg[\sup_{s \in [t,t+\delta]}\hat{\mathbb{E}}_{s}\bigg[ \bigg(\int\nolimits_{t}^{t+\delta} \big(\Gamma^{1}_r+\Gamma^{2}_r+\Gamma^{3}_r \big) dr \bigg)^{3}     \bigg] \bigg]\notag
\end{align}
\begin{align}
	\hspace{5.2em}\leq
	&\,C\delta^{3}\hat{\mathbb{E}}\bigg[\sup_{s \in [t,t+\delta]}\hat{\mathbb{E}}_{s}\bigg[ \bigg(\int\nolimits_{t}^{t+\delta} 
	\big(1+\vert x \vert^{m+2}+\vert X^{\varepsilon_k,t,x}_{r} \vert^{m+2}\big)   dr \bigg)^{3}     \bigg]            \bigg]\notag\\\notag
	&+C\hat{\mathbb{E}}\bigg[\sup_{s \in [t,t+\delta]}\hat{\mathbb{E}}_{s}\bigg[ \bigg(\int\nolimits_{t}^{t+\delta} 
	\big(1+\vert x \vert^{m+2}+\vert X^{\varepsilon_k,t,x}_{r} \vert^{m+2}\big)\vert X^{\varepsilon_k,t,x}_{r}-x\vert   dr \bigg)^{3}     \bigg]            \bigg]\\
	\coloneqq
	&\,M_{1}+M_{2}.
\end{align}
As for $M_1$, applying Lemma \ref{lemma 3.1} and \eqref{(3.6)}, we have
\begin{equation}\label{(4.40)}
	\begin{aligned}
	M_1
	=&\,C\delta^{3}\hat{\mathbb{E}}\bigg[\sup_{s \in [t,t+\delta]}\hat{\mathbb{E}}_{s}\bigg[ \bigg(\int\nolimits_{t}^{t+\delta} 
	\big(1+\vert x \vert^{m+2}+\vert X^{\varepsilon_k,t,x}_{r} \vert^{m+2}\big)   dr \bigg)^{3}     \bigg]            \bigg]\\
	\leq
	&\,C\delta^{6}\hat{\mathbb{E}}\bigg[\sup_{r \in [t,t+\delta]}
	\Big(1+\vert x \vert^{3m+6}+\vert X^{\varepsilon_k,t,x}_{r} \vert^{3m+6}\Big)         \bigg]\\
	&+C\delta^{6}\hat{\mathbb{E}}\bigg[\sup_{s \in [t,t+\delta]}\hat{\mathbb{E}}_{s}\bigg[ \sup_{r \in [s,t+\delta]}
	\Big(1+\vert x \vert^{3m+6}+\big\vert X^{\varepsilon_k,s,X^{\varepsilon_k,t,x}_{s}}_{r} \big\vert^{3m+6}\Big)       \bigg]            \bigg]\\
	\leq
	&\,C\delta^{6}(1+\vert x \vert^{3m+6})
	+C\delta^{6}\hat{\mathbb{E}}\bigg[\sup_{s \in [t,t+\delta]}
	\Big(1+\vert x \vert^{3m+6}+\vert X^{\varepsilon_k,t,x}_{s} \vert^{3m+6}\Big)            \bigg]\\
	\leq
	&\,C\delta^{6}(1+\vert x \vert^{3m+6}).
    \end{aligned}
\end{equation}
As for $M_2$, given that $X^{\varepsilon_k,t,x}_{r}=X^{\varepsilon_k,s,X^{\varepsilon_k,t,x}_{s}}_{r}$ holds  for any $r \in [s,t+\delta]$, again with H{\"o}lder inequality and Lemma \ref{lemma 3.1}, it is easy to get
\begin{align}
	\hat{\mathbb{E}}_{s}\bigg[\sup_{r\in [s,t+\delta]}\big\vert X^{\varepsilon_k,t,x}_{r}-x\big\vert ^6\bigg]\notag
	\leq
	&\,\hat{\mathbb{E}}_{s}\bigg[\sup_{r\in [s,t+\delta]}\big\vert X^{\varepsilon_k,s,X^{\varepsilon_k,t,x}_{s}}_{r}-X^{\varepsilon_k,t,x}_{s}+X^{\varepsilon_k,t,x}_{s}-x\big\vert ^6\bigg] \notag\\\notag
	\leq
	&\,C\Big[\big(1+\big\vert X^{\varepsilon_k,t,x}_{s}\big\vert^{6}\big)\delta^{3}   + \big\vert X^{\varepsilon_k,t,x}_{s}-x\big\vert^{6} \Big].
\end{align}
And also, it is clear to see that
\begin{equation}\label{(4.41)}
	\begin{aligned}
	M_2
	=&\,C\hat{\mathbb{E}}\bigg[\sup_{s \in [t,t+\delta]}\hat{\mathbb{E}}_{s}\bigg[ \bigg(\int\nolimits_{t}^{t+\delta} 
	\big(1+\vert x \vert^{m+2}+\vert X^{\varepsilon_k,t,x}_{r} \vert^{m+2}\big)\vert X^{\varepsilon_k,t,x}_{r}-x\vert   dr \bigg)^{3}     \bigg]            \bigg]\\
	\leq
	&\,C\delta^{3}\bigg(\hat{\mathbb{E}}\bigg[\sup_{r \in [t,t+\delta]}
	\Big(1+\vert x \vert^{6m+12}+\vert X^{\varepsilon_k,t,x}_{r} \vert^{6m+12}\Big)  \bigg]\bigg)^\frac{1}{2}
	\bigg(\hat{\mathbb{E}}\bigg[\sup_{r \in [t,t+\delta]}\big\vert X^{\varepsilon_k,t,x}_{r}-x\big\vert^{6}           \bigg]\bigg)^\frac{1}{2}\\
	&+C\delta^{3}\hat{\mathbb{E}}\bigg[\sup_{s \in [t,t+\delta]}\hat{\mathbb{E}}_{s}\bigg[ \sup_{r\in [s,t+\delta]} 
	\big(1+\vert x \vert^{3m+6}+\vert X^{\varepsilon_k,t,x}_{r} \vert^{3m+6}\big)\sup_{r\in [s,t+\delta]}\big\vert X^{\varepsilon_k,s,X^{\varepsilon_k,t,x}_{s}}_{r}-x\big\vert ^3\bigg]       \bigg]\\
	\leq		
    &\,C\delta^{\frac{9}{2}}(1+\vert x \vert^{3m+9})\\
    &+C\delta^{3}\hat{\mathbb{E}}\bigg[\sup_{s \in [t,t+\delta]}\Big(1+\vert x \vert^{3m+6}+\vert X^{\varepsilon_k,t,x}_{s} \vert^{3m+6}\Big)
    \Big(  
    \big(1+\big\vert X^{\varepsilon_k,t,x}_{s}\big\vert^{6}\big)\delta^{3}   + \big\vert X^{\varepsilon_k,t,x}_{s}-x\big\vert^{6}  \Big)^\frac{1}{2} \bigg]\\
    \leq	
    &\,C\delta^{\frac{9}{2}}(1+\vert x \vert^{3m+9})\\
    &+C\delta^{3}(1+\vert x \vert^{3m+6})
    \bigg(\delta^{3}  \hat{\mathbb{E}}\bigg[\sup_{s \in [t,t+\delta]}  
    \Big(1+\big\vert X^{\varepsilon_k,t,x}_{s}\big\vert^{6}\Big) \bigg]+\hat{\mathbb{E}}\bigg[\sup_{s \in [t,t+\delta]}\big\vert X^{\varepsilon_k,t,x}_{s}-x\big\vert^{6}\bigg]\bigg)^\frac{1}{2} \\
    \leq
    &\,C\delta^{\frac{9}{2}}(1+\vert x \vert^{3m+9}).
    \end{aligned}
\end{equation}
Thus, \eqref{(4.38)}-\eqref{(4.41)} can yield that
	\begin{align}
	C\hat{\mathbb{E}}\bigg[\sup_{s \in [t,t+\delta]}\hat{\mathbb{E}}_{s}\bigg[ \bigg(\int\nolimits_{t}^{t+\delta} \big(\Gamma^{1}_r+\Gamma^{2}_r+\Gamma^{3}_r \big) dr \bigg)^{3}     \bigg]            \bigg]
	\leq 
	C\delta^{\frac{9}{2}}(1+\vert x \vert^{3m+9}),\notag
    \end{align}
which implies that 
\begin{align}\label{(4.42)}
	\hat{\mathbb{E}}\Big[\big\vert \widehat{Y}^{1,\varepsilon_k}_{r}-\bar{Y}^{1,\varepsilon_k}_{r} \big\vert \Big]
	\leq
	C(T,\underline{\sigma},\bar{\sigma},L)\delta^{\frac{3}{2}}\big(1+\vert x \vert^{\frac{3m+9}{2}}\big),
	\quad \forall r \in [t,t+\delta].
\end{align}\par
On the other hand, we can directly see that $\bar{Z}^{1,\varepsilon_k}_{s}=0$ and $\bar{Y}^{1,\varepsilon_k}_{s}$ is the solution of the following ODE on the time interval $[t,t+\delta]$:
\begin{equation}\label{(4.43)}
	\begin{aligned}
	\bar{Y}^{1,\varepsilon_k}_{s}
	=&\int_s^{t+\delta}\big[\bar{f}^{1,\varepsilon_k}\big(r,x,\bar{Y}^{1,\varepsilon_k}_{r},0\big)+\partial_{t}\phi^{1}(t,x)+D_x\phi^{1}(t,x)b^{\varepsilon_k}(r,x)\big]dr \\
	&+\int_s^{t+\delta}  G\big(\sigma^{\varepsilon_k}(r,x\mathbf{)}^\top D^{2}_x\phi^{1}(t,x) \sigma^{\varepsilon_k}(r,x)\big) dr,
    \end{aligned}
\end{equation}
and		
	\begin{align}
	\bar{A}^{1,\varepsilon_k}_{s}
	=
	\int_t^{s}  G\big(\sigma^{\varepsilon_k}(r,x\mathbf{)}^\top D^{2}_x\phi^{1}(t,x) \sigma^{\varepsilon_k}(r,x)\big)    dr
    -\frac{1}{2}\int_t^{s} \sigma^{\varepsilon_k}(r,x\mathbf{)}^\top D^{2}_x\phi^{1}(t,x) \sigma^{\varepsilon_k}(r,x)d\langle B\rangle_{r}.\notag
    \end{align}
It is easy to check that $\{-\bar{A}^{1,\varepsilon_k}_{s}\}_{s \in [t,t+\delta]  }$ is a nonincreasing $G$-martingale.\par

And also, note that $\bar{Y}^{1,\varepsilon_k}_{r} \geq S_1(r,x)$ for each $r \in [t,t+\delta]$, by Lemma 3.4 in Hu et al. \cite{HJP1}, we can immediately demonstrate that $\{-\int\nolimits_{t}^{s}(\bar{Y}^{1,\varepsilon_k}_{r} - S_1(r,x) )d\bar{A}^{1,\varepsilon_k}_{r} \}_{s \in [t,t+\delta]  }$ is a nonincreasing $G$-martingale, which implies that the structure of ODE \eqref{(4.43)} is valid.
Thus, combining \eqref{(4.42)}-\eqref{(4.43)}, we have
\begin{equation}\label{(4.44)}
	\begin{aligned}
	\hat{\mathbb{E}}\Big[\widehat{Y}^{1,\varepsilon_k}_{t}\Big]
	\leq
	&\,\hat{\mathbb{E}}\Big[\big\vert\widehat{Y}^{1,\varepsilon_k}_{t}-\bar{Y}^{1,\varepsilon_k}_{t}\big\vert\Big]
	+\hat{\mathbb{E}}\big[\bar{Y}^{1,\varepsilon_k}_{t}\big]\\
	\leq
	&\,C(T,\underline{\sigma},\bar{\sigma},L)\delta^{\frac{3}{2}}\big(1+\vert x \vert^{\frac{3m+9}{2}}\big)
	+\hat{\mathbb{E}}\bigg[\int_t^{t+\delta}\bar{f}^{1,\varepsilon_k}\big(r,x,\bar{Y}^{1,\varepsilon_k}_{r},0\big)    dr\bigg]\\
	&+\hat{\mathbb{E}}\bigg[\int_t^{t+\delta}\Big(\partial_{t}\phi^{1}(t,x)+D_x\phi^{1}(t,x)b^{\varepsilon_k}(r,x)+G\big(\sigma^{\varepsilon_k}(r,x\mathbf{)}^\top D^{2}_x\phi^{1}(t,x) \sigma^{\varepsilon_k}(r,x)\big)\Big)dr
	\bigg].
    \end{aligned}
\end{equation}
Then, we can check that $X \leq \vert X\vert \textbf{I}_d$ and $G(X) \leq \frac{1}{2}\bar{\sigma}^{2}d\vert X\vert$  for each matrix $X \in \mathbb{S}(d)$, where the matrix norm $\vert X \vert \coloneqq \sqrt{tr\{XX^{\top}\}} $. Thus, it is clear that
\begin{align}
	 G\big(\sigma^{\varepsilon_k}(u,x\mathbf{)}^\top D^{2}_x\phi^{1}(t,x) \sigma^{\varepsilon_k}(u,x)\big)
	\leq
	\frac{\bar{\sigma}^2d}{2}\big\vert \sigma^{\varepsilon_k}(u,x\mathbf{)}^\top D^{2}_x\phi^{1}(t,x)\sigma^{\varepsilon_k}(u,x) \big\vert^{2}.\notag
\end{align}
In addition, from Remark \ref{remark 5.3} in appendix, we note that the lower obstacle $S_{1}(s,x)$ of $G$-BSDE \eqref{(4.32)} is bounded by $c=0$ from above, which together with assumptions (H1)-(H2) can indicate that
\begin{align}
	\vert \bar{Y}^{1,\varepsilon_k}_{r} \vert^{2}
	\leq
	&\,C(T,L,\underline{\sigma})\hat{\mathbb{E}}_r\bigg[   \int_r^{t+\delta}\big\vert \bar{f}^{1,\varepsilon_k}(u,x,0,0)+\partial_{t}\phi^{1}(t,x)+D_x\phi^{1}(t,x)b^{\varepsilon_k}(u,x)  \big\vert^{2}  du \bigg] \notag\\
	&+C(T,L,\underline{\sigma})\hat{\mathbb{E}}_r\bigg[\int_r^{t+\delta}  \frac{\bar{\sigma}^2d}{2}\big\vert \sigma^{\varepsilon_k}(u,x\mathbf{)}^\top D^{2}_x\phi^{1}(t,x) \sigma^{\varepsilon_k}(u,x) \big\vert^{2}   du\bigg]\notag\\\notag
	 \leq
	 &\,C(T,L,\underline{\sigma},\bar{\sigma},d)\delta(1+\vert x \vert^{2m+4}),
\end{align}
and
\begin{align}\label{(4.45)}
	\big\vert \bar{f}^{1,\varepsilon_k}\big(r,x,\bar{Y}^{1,\varepsilon_k}_{r},0\big)
	-\bar{f}^{1,\varepsilon_k}(r,x,0,0)  \big\vert
	\leq 
	&\,C(T,L,\underline{\sigma},\bar{\sigma},d)\delta^{\frac{1}{2}}(1+\vert x \vert^{m+2}),\quad \forall r \in [t,t+\delta].
\end{align}
Thus, by \eqref{(4.44)}-\eqref{(4.45)}, we can derive that
\begin{align}\label{(4.46)}
	\hat{\mathbb{E}}\Big[ \widehat{Y}^{1,\varepsilon_k}_{t}   \Big]
	\leq
	&\int_t^{t+\delta}\Big[\partial_{t}\phi^{1}(t,x)+\bar{f}^{1,\varepsilon_k}(r,x,0,0)+D_x\phi^{1}(t,x)b^{\varepsilon_k}(r,x)+G\big(\sigma^{\varepsilon_k}(r,x\mathbf{)}^\top D^{2}_x\phi^{1}(t,x) \sigma^{\varepsilon_k}(r,x)\big)\Big]dr\notag\\
	&+C(T,L,\underline{\sigma},\bar{\sigma})\delta^{\frac{3}{2}}\big(1+\vert x \vert^{\frac{3m+9}{2}}\big)
	+C(T,L,\underline{\sigma},\bar{\sigma},d)\delta^{\frac{3}{2}}(1+\vert x \vert^{m+2}).
    \end{align}\par
Owing to \eqref{(4.30)}, for a fixed $\delta$, and for an arbitrary positive real number $\epsilon$ we can always choose a large enough integer $k$, such that 
\begin{align}
     \mathbb{G}^{\varepsilon_{k},t,x}_{t,t+\delta}\big[\phi^{1}\big(t+\delta,X^{\varepsilon_{k},t,x}_{t+\delta}\big)\big]
     \,\textgreater\,
     \phi^{1}(t,x)-\epsilon.\notag
\end{align}
Without loss of generality, for each $n \geq 1$, letting $\epsilon_{n}=\frac{1}{2^{n}}$, there exists a decreasing subsequence $\{\varepsilon_{k_n}\}_{n} \downarrow 0$ of $\{\varepsilon_{k}\}_{k}$, such that
	\begin{align}\label{(4.47)}
		\mathbb{G}^{\varepsilon_{k_n},t,x}_{t,t+\delta}\big[\phi^{1}\big(t+\delta,X^{\varepsilon_{k_n},t,x}_{t+\delta}\big)\big]
		\,\textgreater\,
		\phi^{1}(t,x)-\frac{1}{2^{n}},\quad
		\forall n \geq 1.
	\end{align}
Thus, by \eqref{(4.46)}-\eqref{(4.47)} and $\widehat{Y}^{1,\varepsilon_k}_{t}=\mathbb{G}^{\varepsilon_{k},t,x}_{t,t+\delta}\big[\phi^{1}\big(t+\delta,X^{\varepsilon_{k},t,x}_{t+\delta}\big)\big]-\phi^{1}(t,x)$, we obtain that for this fixed $\delta$
\begin{align}
		&\,\int_t^{t+\delta}\Big[\partial_{t}\phi^{1}(t,x)+\bar{f}^{1,\varepsilon_{k_{n}}}(r,x,0,0)+D_x\phi^{1}(t,x)b^{\varepsilon_{k_{n}}}(r,x)+G\big(\sigma^{\varepsilon_{k_{n}}}(r,x \mathbf{)}^\top D^{2}_x\phi^{1}(t,x) \sigma^{\varepsilon_{k_{n}}}(r,x)\big)\Big]dr\notag\\\label{(4.48)}
		\geq
		&-C\delta^{\frac{3}{2}}\big(1+\vert x \vert^{\frac{3m+9}{2}}\big)
		-\frac{1}{2^{n}}, \quad 
		\forall n \geq 1.
\end{align}\par
$\mathbf{Part\ 3\,}$(Asymptotic analysis).
By means of the variable transformation method, recalling the definition of the coefficients $\bar{f}^{1,\varepsilon_{k_{n}}}, \sigma^{\varepsilon_{k_{n}}},b^{\varepsilon_{k_{n}}}$, we have
\begin{align}
	&\,\delta\partial_{t}\phi^{1}(t,x)\notag\\\label{(4.49)}
	&+\varepsilon_{k_{n}}\int_{\frac{t}{\varepsilon_{k_{n}}}}^{\frac{t+\delta}{\varepsilon_{k_{n}}}}\Big[f\big(r,x,\phi^{1}(t,x),D_x\phi^{1}(t,x) \sigma(r,x)\big)+D_x\phi^{1}(t,x)b(r,x)+G\big(\sigma(r,x \mathbf{)}^\top D^{2}_x\phi^{1}(t,x) \sigma(r,x)\big)\Big]dr\notag\\
	\geq
	&-C\delta^{\frac{3}{2}}\big(1+\vert x \vert^{\frac{3m+9}{2}}\big)
	-\frac{1}{2^{n}}, \quad 
	\forall n \geq 1.
\end{align}
And also, from assumption (H5), we can derive that
\begin{align}\label{(4.50)}
	&\lim_{n \rightarrow \infty}\varepsilon_{k_{n}}\int_{\frac{t}{\varepsilon_{k_{n}}}}^{\frac{t+\delta}{\varepsilon_{k_{n}}}}\Big[f\big(r,x,\phi^{1}(t,x),D_x\phi^{1}(t,x) \sigma(r,x)\big)+D_x\phi^{1}(t,x)b(r,x)+G\big(\sigma(r,x \mathbf{)}^\top D^{2}_x\phi^{1}(t,x) \sigma(r,x)\big)\Big]dr\notag\\
	 =&
	\,\delta \bar{F}\big(x,\phi^{1}(t,x),D_x\phi^{1}(t,x),D^{2}_x\phi^{1}(t,x)\big).
\end{align}
Furthermore, for this fixed $\delta$, with aid of \eqref{(4.49)}-\eqref{(4.50)}, letting $n \rightarrow \infty$, we get
\begin{align}\label{(4.51)}
	\,\partial_{t}\phi^{1}(t,x)
	+\bar{F}\big(x,\phi^{1}(t,x),D_x\phi^{1}(t,x),D^{2}_x\phi^{1}(t,x)\big)
	\geq
	-C\delta^{\frac{1}{2}}\big(1+\vert x \vert^{\frac{3m+9}{2}}\big). 
\end{align}\par
Likewise, for any $\delta \in (0,1)$, we can always  deduce the results similar to \eqref{(4.47)}-\eqref{(4.50)}, which indicates that \eqref{(4.51)} still holds. Thus, sending $\delta \rightarrow 0$, we get
\begin{align}
	\partial_{t}\phi^{1}(t,x)\notag
	+\bar{F}\big(x,\phi^{1}(t,x),D_x\phi^{1}(t,x),D^{2}_x\phi^{1}(t,x)\big)
	\geq
	0,
\end{align}
which implies the desired result.
\end{proof}\par
Now we shall formulate the proof of Theorem \ref{Theorem 4.7}.\\
\begin{proof of Theorem 4.7}
	The proof now follows from the uniqueness of the averaged PDE \eqref{(4.1)}. In fact, by means of Lemmas \ref{lemma 4.6} and \ref{Lemma 4.10}, for any sequence $(\varepsilon_{l})_{l \geq 1}$ converges to $0$, we always find a subsequence $(\varepsilon_{l_k})_{k \geq 1}$ of $(\varepsilon_{l})_{l \geq 1}$, such that $u^{\varepsilon_{l_k}} \rightarrow u^{\ast}$ as $k \rightarrow \infty$, where $u^{\ast}$ is a viscosity solution of the averaged PDE \eqref{(4.1)}. In addition, note that the averaged PDE \eqref{(4.1)} admits a unique viscosity solution, thus we derive that $ \bar{u} \equiv u^{\ast}$, which completes the proof of Theorem \ref{Theorem 4.7}.
\end{proof of Theorem 4.7}

\section{Appendix}\label{section 5}

\noindent In what follows, we shall recall some basic results about reflected $G$-BSDEs for reader's convenience. The following processes are established under the condition $d=1$, but the results and methods still hold for the case $d \,\textgreater\, 1$. We are given the following data: the generators $f$ and $g$, the obstacle process $\{L_t\}_{t \in [0,T]}$ and the terminal value $\xi$, where $f$ and $g$ are maps $f(t,\omega,y,z),\,g(t,\omega,y,z)$: $[0,T] \times \Omega_{T} \times \mathbb{R} \times \mathbb{R} \rightarrow \mathbb{R}$. \par
We firstly present some assumptions regarding the coefficients of the reflected $G$-BSDE: there exists some $\beta \,\textgreater\, 2$ such that\\
\textbf{($\mathbf{A1}$)} for any $y,z$, $f(\cdot,\cdot,y,z),\,g(\cdot,\cdot,y,z) \in M^{\beta}_{G}(0,T)$;\\
\textbf{($\mathbf{A2}$)} There exists a constant $L\,\textgreater\,0$ such that, for any $t \in [0,T]$ and $y,\,y'$, $z,z' \in \mathbb{R}$,
\begin{equation}
	\vert f(t,\omega,y,z)-f(t,\omega,y',z')\vert+\vert g(t,\omega,y,z)-g(t,\omega,y',z')\vert \leq L(\vert y-y'\vert+\vert z-z'\vert); \notag
\end{equation}
\textbf{($\mathbf{A3}$)} $\xi \in L^{\beta}_{G}(\Omega_T)$ and $\xi \geq L_T$ q.s.;\\
\textbf{($\mathbf{A4}$)} There exists a constant $c$ such that $(L_{t})_{t \in [0,T]} \in S^{\beta}_{G}(0,T)$ and $L_{t} \leq c$ for each $t \in [0,T]$;\\
\textbf{($\mathbf{A5}$)} $(L_{t})_{t \in [0,T]}$ has the following form:
\begin{equation}
	L_t=L_0+\int\nolimits_{0}^{t} b(s)ds+\int\nolimits_{0}^{t} l(s)d\langle B\rangle_{s}+\int\nolimits_{0}^{t} \sigma(s)dB_s,\notag
\end{equation}
where $\{b(t)\}_{t\in [0,T]}$ and $\{l(t)\}_{t\in [0,T]}$ belong to $M^{\beta}_{G}(0,T)$ and $\{\sigma(t)\}_{t\in [0,T]}$ belongs to $H^{\beta}_{G}(0,T)$.\par
For some $1  \,\textless\, \alpha \leq \beta $, there exists a unique triple $(Y,Z,A) \in \mathcal{S}^{2}_{G}(0,T)$ under the above assumptions, which solves the following reflected $G$-BSDE with a lower obstacle on the time interval $[0,T]$. And also, $(Y,Z,A)$ also satisfies the following properties:\\
${}\hspace{0.22em}$ (i) $(Y,Z,A) \in \mathcal{S}^{\alpha}_{G}(0,T)$, $Y_t \geq L_t, 0 \leq t \leq T$;\vspace{1.5pt}\\
${}\hspace{0.001em}$ (ii) $Y_{t}=\xi+\int_t^{T} f(s,Y_{s},Z_{s})ds+\int_t^{T} g(s,Y_{s},Z_{s})d\langle B\rangle_{s}-\int_t^{T}Z_{s} dB_{s}+(A_{T}-A_{t})$;\vspace{1pt}\\
(iii) $\{A_{t}\}_{t\in[0,T]}$ is continuous nondecreasing, and $\{-\int_0^{t} (Y_{s}-L_{s})dA_{s}\}_{t \in [0,T]}$ is a nonincreasing $G$-martingale.\\
For simplicity, we denote by $\mathcal{S}^{\alpha}_{G}(0,T)$ the collection of $(Y,Z,A)$ such that $Y \in S^{\alpha}_{G}(0,T),\,Z \in H^{\alpha}_{G}(0,T),\,A \in S^{\alpha}_{G}(0,T)$ and $A_{0}=0$.\par
Now we shall give the existence and uniqueness of the solution to the above reflected $G$-BSDE with a lower obstacle $\{L_t\}_{t\in[0,T]}$.
\begin{lemma}\label{lemma 5.1}
  Suppose that $\xi,\,f\,$and $g$ satisfy $\mathrm{(A1)}$-$\mathrm{(A3)}\mathrm{,}$ $L$ satisfies $\mathrm{(A4)}$ or $\mathrm{(A5)}$. Then$\mathrm{,}$ the reflected $G$-BSDE with data $(\xi,f,g,L)$ admits a unique solution $(Y,Z,A)$. In addition$\mathrm{,}$ for any $2 \leq \alpha \,\textless\, \beta\mathrm{,}$ we have $Y\in S^{\alpha}_{G}(0,T),\,Z \in H^{\alpha}_{G}(0,T)$ and $A \in S^{\alpha}_{G}(0,T)$.
\end{lemma}\par
Next, we present some relevant estimates regarding the case where $g \equiv 0$ and $l \equiv 0$, similar results and methods still hold for the cases $g \neq 0,\,l \neq 0$.
\begin{lemma}\label{lemma 5.2}
	Suppose that $(\xi,f,L)$ satisfy $\mathrm{(A1)}$-$\mathrm{(A4)}\mathrm{,}$ and $(Y,Z,A) \in \mathcal{S}^{\alpha}_{G}(0,T)\mathrm{,}$ for some $2 \leq \alpha \,\textless\,\beta \mathrm{,}$ is a solution of the reflected $G$-BSDE with data $(\xi,f,L)$. Then$\mathrm{,}$ we have
	\begin{equation}
		\begin{aligned}
		\vert Y_{t}\vert^{\alpha}
		\leq 
		C\hat{\mathbb{E}}_{t}\bigg[1+\vert \xi \vert^{\alpha}+\int\nolimits_{t}^{T} \vert f(s,0,0)\vert^{\alpha} ds\bigg],\notag
		\end{aligned}
	\end{equation}
where the positive constant $C$ depends on $\alpha, T,L,\underline{\sigma},c$.		
\end{lemma}
\begin{remark}\label{remark 5.3}
	\rm{I}n particular, when the upper bound $c$ of the obstacle process is equal to $0$, under the same assumptions as Lemma \ref{lemma 5.2}, we get
	\begin{align}
		\vert Y_{t}\vert^{\alpha}\notag
		\leq 
		C\hat{\mathbb{E}}_{t}\bigg[\vert \xi \vert^{\alpha}+\int\nolimits_{t}^{T} \vert f(s,0,0)\vert^{\alpha} ds\bigg],
	\end{align}
\end{remark}
where the positive constant $C$ depends on $\alpha, T,L,\underline{\sigma}$.
\begin{lemma}\label{lemma 5.4}
	Let $f$ satisfy $\mathrm{(A1)}$-$\mathrm{(A2)}$$\mathrm{,}$ and $(Y,Z,A)$ satisfy:
	\begin{align}
	Y_{t}=\xi+\displaystyle\int_t^{T} f(s,Y_{s},Z_{s})ds-\int_t^{T}Z_{s} dB_{s}+(A_{T}-A_{t}),\notag
	\end{align}
    where $(Y,Z,A) \in \mathcal{S}^{\alpha}_{G}(0,T)$ with $\alpha \,\textgreater\, 1 \mathrm{,}$ then for each $t\in [0,T]\mathrm{,}$
    \begin{align}
    &\,\hat{\mathbb{E}}_{t}\bigg[  \bigg(\int\nolimits_{t}^{T} \vert Z_{s}\vert^{2} ds\bigg)^{\frac{\alpha}{2}}  \bigg]\notag\\\notag
    \leq 
    &\,C\bigg\{\hat{\mathbb{E}}_{t}\bigg[\sup_{s \in [t,T]} \vert Y_s\vert^{\alpha}\bigg]+\bigg(\hat{\mathbb{E}}_{t}\bigg[\sup_{s \in [t,T]} \vert Y_s\vert^{\alpha}\bigg]\bigg)^{\frac{1}{2}}\bigg(\hat{\mathbb{E}}_{t}\bigg[\bigg(\int\nolimits_{t}^{T} \vert f(s,0,0)\vert ds\bigg)^{\alpha}\bigg]\bigg)^{\frac{1}{2}}\bigg\},\\
    &\,\hat{\mathbb{E}}_{t}\big[ \vert A_{T}-A_t\vert^{\alpha}\big] 
    \leq
    C\bigg\{\hat{\mathbb{E}}_{t}\bigg[\sup_{s \in [t,T]} \vert Y_s\vert^{\alpha}\bigg]+\hat{\mathbb{E}}_{t}\bigg[\bigg(\int\nolimits_{t}^{T} \vert f(s,0,0)\vert ds\bigg)^{\alpha}\bigg]\bigg\},\notag
    \end{align}
where the positive contsant $C$ depends on $\alpha, T,L,\underline{\sigma}$.
\end{lemma}
\begin{lemma}\label{lemma 5.5}
	Let $(\xi^1,f^1,L^1)$ and $(\xi^2,f^2,L^2)$ be two sets of data$\mathrm{,}$ each one satisfying assumptions $\mathrm{(A1)}$-$\mathrm{(A4)}\mathrm{,}$ and let $(Y^i,Z^i,A^i) \in \mathcal{S}^{\alpha}_{G}(0,T)$ be the solutions of the reflected $G$-BSDEs with data $(\xi^i,f^i,L^i)\mathrm{,}$\,$i=1,2\mathrm{,}$ respectively$\mathrm{,}$ with $2 \leq \alpha \,\textless\, \beta$. Set $\hat{Y}_t=Y^{1}_{t}-Y^{2}_{t}\mathrm{,}$ $\hat{L}_t=L^{1}_{t}-L^{2}_{t}\mathrm{,}$
	$\hat{\xi}=\xi^{1}-\xi^{2}$. Then$\mathrm{,}$ there exists a constant $C \coloneqq C(\alpha,T,L,\underline{\sigma},c) \,\textgreater\, 0$ such that
	 \begin{equation}
		\begin{aligned}
		    \vert \hat{Y}_t \vert^{\alpha}
			&\leq 
			C\bigg\{\hat{\mathbb{E}}_{t}\bigg[\vert \hat{\xi}\vert^{\alpha}+\int\nolimits_{t}^{T} \vert \hat{\lambda}_{s} \vert ^{\alpha} ds\bigg]+\bigg(\hat{\mathbb{E}}_{t}\bigg[\sup_{s \in [t,T]} \vert \hat{L}_s\vert^{\alpha}\bigg]\bigg)^{\frac{1}{\alpha}}\Psi_{t,T}^{\frac{\alpha-1}{\alpha}}\bigg\},\notag\\
		\end{aligned}
	\end{equation}
where $\hat{\lambda}_{s}=\vert f^1(s,Y^2_{s},Z^2_{s})-f^2(s,Y^2_{s},Z^2_{s})\vert$ and
	\begin{align}
    \Psi_{t,T}&=\sum_{i=1}^{2}\hat{\mathbb{E}}_{t}\bigg[\sup\limits_{s \in [t,T]}  \hat{\mathbb{E}}_{s}\bigg[1+\vert \xi^i \vert^{\alpha} + \int\nolimits_{t}^{T} \vert f^i(r,0,0)\vert^{\alpha} dr  \bigg]    \bigg].\notag
	\end{align}
\end{lemma}\par
For the proof of Lemmas \ref{lemma 5.1}-\ref{lemma 5.2}, \ref{lemma 5.4}-\ref{lemma 5.5} and Remark \ref{remark 5.3}, one can refer to \cite{LPS}. Next, we shall state the comparison theorem for the averaged PDE \eqref{(4.1)}, which can be utilized to ensure that the averaged PDE \eqref{(4.1)} admits the unique viscosity solution.
\begin{lemma}\label{lemma 5.6}
The obstacle problem \eqref{(4.1)} has at most one viscosity solution in the class of continuous functions which grow at most polynomially at infinity.
\end{lemma}
\begin{proof}
	The basic idea of the proof comes from Theorem 8.6 in \cite{EKPP} and Theorem 2.2 in Appendix C of \cite{Peng3}. For reader's convenience, we shall give the sketch of the proof.\par
	We assume that both $u$ and $v$ are the viscosity solutions to the averaged PDE \eqref{(4.1)}. And also, $u$ and $v$ satisfy the polynomial growth condition and $u(T,x)=v(T,x)=\varphi(x)$, $x \in \mathbb{R}^{n} $. Without loss of generality, we only need to prove that $u \leq v $ when $u$ is the viscosity subsolution and $v$ is viscosity supersolution.\par
For some constant $\lambda \,\textgreater\, 0$ to be chosen below, let
	\begin{align}
			\widetilde{u}(t,x)&=u(t,x)e^{\lambda t}\xi^{-1}(x),\notag\quad
			\widetilde{v}(t,x)=v(t,x)e^{\lambda t}\xi^{-1}(x),\notag\\\notag
			\widetilde{S}(t,x)&=S(t,x)e^{\lambda t}\xi^{-1}(x),\hspace{1.6em}
			\widetilde{\varphi}(x)=\varphi(x)e^{\lambda T}\xi^{-1}(x),
	\end{align}
	where $\xi(x)=(1+\vert x\vert^{2})^{\frac{k}{2}}$. In addition, choosing an appropriate integer $k \in \mathbb{N}$ can derive that $\widetilde{u}$ and $\widetilde{v}$ are bounded. It is easy to check that $\widetilde{u}$ (resp. $\widetilde{v}$) is a bounded viscosity subsolution (resp. supersolution) of the following obstacle problem:
	\begin{equation}\label{(5.1)}
		\left\{\begin{matrix} 
			\begin{aligned}     		
				&\min\big(-\partial_{t}\widetilde{u}(t,x)+\lambda\widetilde{u}(t,x)-\bar{F}^{*}(x,\widetilde{u},D_{x}\widetilde{u},D^{2}_{x}\widetilde{u}),\,\widetilde{u}(t,x)-\widetilde{S}(t,x)\big)=0,\quad (t,x) \in (0,T)\times \mathbb{R}^{n},\\ 
				&\,\widetilde{u}(T,x)=\widetilde{\varphi}(x),\hspace{25.7em} x\in\mathbb{R}^{n},    	   	
			\end{aligned}
		\end{matrix}
		\right.
	\end{equation}
where the function $\bar{F}^{*}(x,v,p,A)$ is given by
\begin{align}
			e^{\lambda t}\xi^{-1}(x)\bar{F}(x,e^{-\lambda t}v\xi(x),e^{-\lambda t}(p\xi(x)+vD\xi(x)),e^{-\lambda t}(A\xi(x)+p\otimes D\xi(x)+ D\xi(x)\otimes p +vD^{2}\xi(x))),\notag
\end{align}
	for any $(x,v,p,A) \in \mathbb{R}^{n} \times \mathbb{R} \times \mathbb{R}^{1 \times n}  \times \mathbb{S}(n)$ with $p\otimes D\xi(x)=[p^{i}D\xi^{j}(x)]_{i,j}$. We note that there exists a constant $L \,\textgreater\, 0$ such that
	\begin{equation}\label{(5.2)}
		\begin{aligned}
			\vert \xi^{-1}(x)-\xi^{-1}(y)\vert &\leq L\vert x-y\vert,\\
			\vert \xi(x)-\xi(y)\vert \vee \vert D\xi(x)-D\xi(y) \vert &\leq L(1+\vert x\vert^{k}+\vert y\vert^{k})\vert x-y\vert.\\
		\end{aligned}
	\end{equation}
Without loss of generality, we assume that $h_{ij}=g_{ij}=0,\,i,j=1,\dots,d$. For $k \geq 2$, we establish the following functions:
	\begin{equation}\label{(5.3)}
		\begin{aligned}
			\eta(x) &\coloneqq \xi^{-1}(x)D\xi(x)=k(1+\vert x \vert^{2})^{-1}x,\\
			\kappa(x) &\coloneqq \xi^{-1}(x)D^{2}\xi(x)=k(1+\vert x \vert^{2})^{-1}\textbf{I}_{n}+k(k-2)(1+\vert x \vert^{2})^{-2}x\otimes x,
		\end{aligned}
	\end{equation}
where $D\xi$ denotes the gradient of $\xi$, and $D^{2}\xi$ denotes the matrix of second order partial derivatives of $\xi$. It is easy to check that $\eta(x)$ and $\kappa(x)$ are uniformly bounded functions satisfying globally Lipschitz condition.
For the convenience of the subsequent proof, we set $\widetilde{F}(x,v,p,A)\coloneqq \lambda v-\bar{F}^{*}(x,v,p,A)$, the following procedures mainly start from exploring the properties of the functions $\bar{F}^{\ast}$ and $\widetilde{F}$.\par
Firstly, we shall verify that the function $\bar{F}^{\ast}(x,v,p,A)$ satisfies Lipschitz condition with respect to $v$. For any $v,v' \in \mathbb{R}$, $(x,p,A) \in \mathbb{R}^{n} \times \mathbb{R}^{1 \times n}  \times \mathbb{S}(n)$, by the definition of $\bar{F}^{\ast}$ and $\bar{F}$, we derive that
	\begin{align}\label{(5.4)}
			&\,\bar{F}^{\ast}(x,v,p,A)-\bar{F}^{\ast}(x,v',p,A)\notag\\\notag
			=&\,e^{\lambda t}\xi^{-1}(x)\bar{F}(x,e^{-\lambda t}v\xi(x),\,e^{-\lambda t}(p\xi(x)+vD\xi(x)),\,e^{-\lambda t}(A\xi(x)+p\otimes D\xi(x) +D\xi(x)\otimes p+vD^{2}\xi(x)  )    )\\\notag
			&-e^{\lambda t}\xi^{-1}(x)\bar{F}(x,e^{-\lambda t}v'\xi(x),e^{-\lambda t}(p\xi(x)+v'D\xi(x)), e^{-\lambda t}(A\xi(x)+p\otimes D\xi(x) +D\xi(x)\otimes p+v'D^{2}\xi(x) ) )\\\notag
			\leq
			&\,e^{\lambda t}\xi^{-1}(x)\limsup_{s \rightarrow \infty}\frac{1}{s} \int_{0}^{s} e^{-\lambda t}(v-v')D\xi(x)b(r,x) dr\\\notag
			&+e^{\lambda t}\xi^{-1}(x)\limsup_{s \rightarrow \infty}\frac{1}{s} \int_{0}^{s} [\beta(r,x,v,p)-\beta(r,x,v',p)]dr\\\notag
			&+e^{\lambda t}\xi^{-1}(x)\limsup_{s \rightarrow \infty}\frac{1}{s} \int_{0}^{s} [G(\gamma(r,x,v,p,A))-G(\gamma(r,x,v',p,A))]dr\\
			\coloneqq
			&\,O_{1}+O_{2}+O_{3},
	\end{align}
where the coefficients $\beta(r,x,v,p)$ and $\gamma(r,x,v,p,A)$ satisfy
	\begin{align}
			\beta(r,x,v,p)&=f(r,x,e^{-\lambda t}v\xi(x),e^{-\lambda t}(p\xi(x)+vD\xi(x))\sigma(r,x)),\notag\\\notag
			\gamma(r,x,v,p,A)&=\sigma(r,x\mathbf{)}^\top\big[e^{-\lambda t}(A\xi(x)+p\otimes D\xi(x) +D\xi(x)\otimes p+vD^{2}\xi(x))\big] \sigma(r,x).
	\end{align}
According to assumption (H1) and \eqref{(5.3)}, it is easy to check that
\begin{align}\label{(5.5)}
			\vert \eta(x)b(r,x)\vert \leq Lk,\quad
			\vert \eta(x)\sigma(r,x)\vert \leq Lk.
\end{align}
As for $O_{1}$ and $O_{2}$, we derive that
	\begin{align}
			O_{1}=e^{\lambda t}\xi^{-1}(x)\limsup_{s \rightarrow \infty}\frac{1}{s} \int_{0}^{s} e^{-\lambda t}(v-v')D\xi(x)b(r,x) dr\notag
			\leq
		    C(L,k)\vert v-v'\vert,
	\end{align}
and
	\begin{align}
			O_{2}
			=
			&\,e^{\lambda t}\xi^{-1}(x)\limsup_{s \rightarrow \infty}\frac{1}{s} \int_{0}^{s} [\beta(r,x,v,p)-\beta(r,x,v',p)]dr\notag\\\notag
			\leq
			&\limsup_{s \rightarrow \infty}\frac{1}{s} \int_{0}^{s} L(1+\vert \eta(x)\sigma(r,x)\vert)\vert v-v'\vert dr\\\notag
			\leq
			&\,C(L,k)\vert v-v'\vert.
		\end{align}
In addition, we note that for any $x \in \mathbb{R}^{n}$,
	\begin{align}\label{(5.6)}
			\vert \sigma(r,x\mathbf{)}^\top \kappa(x) \sigma(r,x) \vert
			\leq
			Ck\vert \textbf{I}_{n}\vert+Ck(k-2)(1+\vert x \vert^{2})^{-1}\vert x\otimes x\vert
			\leq
			C(L,k,n),
	\end{align}
where the matrix norm $\vert X \vert \coloneqq \big(\sum\limits_{i=1}^{n}\sum\limits_{j=1}^{n}\vert x_{ij}\vert^2\big)^{\frac{1}{2}}$.\\
As for $O_{3}$, owing to $X \leq \vert X\vert \textbf{I}_d$ for each matrix $X \in \mathbb{S}(d)$, it is clear that $G(X) \leq \frac{1}{2}\bar{\sigma}^{2}d\vert X \vert$, where the matrix norm $\vert X \vert \coloneqq \sqrt{tr\{XX^{\top}\}} $. By the sublinear property of $G$, we have
		\begin{align}
			O_{3}
			=
			&\,e^{\lambda t}\xi^{-1}(x)\limsup_{s \rightarrow \infty}\frac{1}{s} \int_{0}^{s} [G(\gamma(r,x,v,p,A))-G(\gamma(r,x,v',p,A))]dr\notag\\\notag
			\leq
			&\limsup_{s \rightarrow \infty}\frac{1}{s} \int_{0}^{s} G\big(\sigma(r,x\mathbf{)}^\top[(v-v')\kappa(x)] \sigma(r,x)\big)dr\\\notag
			\leq
			&\,C(L,d,n,k,\bar{\sigma})\vert v-v'\vert.\notag
		\end{align}
Similarly, repeating the above discussions for the term $\bar{F}^{\ast}(x,v',p,A)-\bar{F}^{\ast}(x,v,p,A)$, we can get
	\begin{align}\label{(5.7)}
			\vert \bar{F}^{\ast}(x,v,p,A)-\bar{F}^{\ast}(x,v',p,A)\vert
			\leq
			C(L,d,n,k,\bar{\sigma})\,\vert v-v'\vert.
	\end{align}
Then, for any $(x,p,A) \in \mathbb{R}^{n} \times  \mathbb{R}^{1 \times n} \times \mathbb{S}(n) $, $v,v' \in \mathbb{R}$ and $v \geq v'$, by using \eqref{(5.7)}, we have
	\begin{align}\label{(5.8)}
			\widetilde{F}(x,v,p,A)-\widetilde{F}(x,v',p,A)
			=
			\lambda (v-v')-(\bar{F}^{*}(x,v,p,A)-\bar{F}^{*}(x,v',p,A))
			\geq
			(\lambda-C) (v-v'),
	\end{align}
where the positive constant $C$ depends on $L,d,n,k,\bar{\sigma}$.
Thus, choosing a large enough $\lambda$ can ensure that the function $\widetilde{F}(x,v,p,A)$ is strictly increasing with respect to the argument $v$.\par
Next, we will check that the function $\widetilde{F}(x,v,p,A)$ is degenerate elliptic regarding the matrix $A \in \mathbb{S}(n)$. For any matrix $A_{1}, A_{2} \in \mathbb{S}(n)$ satisfying $A_{1} \geq A_{2}$, we immediately see that 
	\begin{align}
			\widetilde{F}(x,v,p,A_{1})-\widetilde{F}(x,v,p,A_{2})
			=
			&\,\bar{F}^{*}(x,v,p,A_{2})-\bar{F}^{*}(x,v,p,A_{1})\notag\\
			=
			&\,e^{\lambda t}\xi^{-1}(x)\limsup_{s \rightarrow \infty}\frac{1}{s} \int_{0}^{s} [G(\gamma(r,x,v,p,A_{2}))-G(\gamma(r,x,v,p,A_{1}))]dr,\notag
	\end{align}
where the function $\gamma$ is defined in \eqref{(5.4)}. When $A_{1} \geq A_{2}$ holds, note that $\gamma(r,x,v,p,A_{1})\geq \gamma(r,x,v,p,A_{2})$, together with the sublinear property of $G$, we obtain that 
\begin{align}
		G(\gamma(r,x,v,p,A_{1}))\geq G(\gamma(r,x,v,p,A_{2})),\notag
\end{align}
which directly implies that
	\begin{align}\label{(5.9)}
			\widetilde{F}(x,v,p,A_{1}) 
			\leq
			\widetilde{F}(x,v,p,A_{2}).
		\end{align}
Thus, the function $\widetilde{F}(x,v,p,A)$ is equiped with the properties, i.e., degenerate ellipticity and \eqref{(5.8)}, we say $\widetilde{F}$ is ``proper'' in the terminology of \cite{CIP}.\par
From now on, we make a last modification of $\widetilde{v}$. For any $\varepsilon \,\textgreater\, 0$, the ``transformed'' function $\widetilde{v}_{1}$ satisfies the following structure:${}\vspace{0.85em}$
	\begin{equation}\label{(5.10)}
		\widetilde{v}_{1}(t,x)=
		\left\{\begin{matrix} 
			\begin{aligned}     		
				&\widetilde{v}(t,x)+\frac{\varepsilon}{t},\ \,t \in \left(0,T\right],\\ 	
				&\,\infty,\hspace{3.95em} t=0.	   		
			\end{aligned}
		\end{matrix}
		\right.\vspace{0.85em}
	\end{equation}
In addition, it is easy to check that $\widetilde{v}_{1}$ is a continuous function with respect to $t$, uniformly in $x$. Moreover, since $\widetilde{v}$ is the viscosity supersolution of \eqref{(5.1)} and $\widetilde{F}$ is ``proper'', for any $(t,x) \in (0,T)\times \mathbb{R}^{n}$, we have
\begin{equation}
\begin{aligned}\label{(5.11)}  		
		-(\partial_{t}\widetilde{v}_{1}+\frac{\varepsilon}{t^{2}})+\widetilde{F}(x,\widetilde{v}_{1},D_{x}\widetilde{v}_{1},D^{2}_{x}\widetilde{v}_{1}) 
		\geq 
		&-(\partial_{t}\widetilde{v}_{1}+\frac{\varepsilon}{t^{2}})+\widetilde{F}(x,\widetilde{v},D_{x}\widetilde{v}_{1},D^{2}_{x}\widetilde{v}_{1}) \\
		=&-\partial_{t}\widetilde{v}+\widetilde{F}(x,\widetilde{v},D_{x}\widetilde{v},D^{2}_{x}\widetilde{v}) \\
		\geq &\ 0.
\end{aligned}
\end{equation}
Thus, we only need to verify that the ``transformed'' function $\widetilde{v}_{1}$ satisfies $\widetilde{u}(t,x) \leq \widetilde{v}_{1}(t,x)$ for any $(t,x) \in [0,T] \times \mathbb{R}^n$. For any $R \,\textgreater\, 0$, define $B_{R} \coloneqq\{x \in \mathbb{R}^{n}; \vert x\vert \leq R\}$, we only verify that for any $R \,\textgreater\, 0$,
	\begin{align}  \label{(5.12)}		
			\sup_{[0.T] \times B_{\!R}}(\widetilde{u}-\widetilde{v}_{1})^{+} (t,x)
			\leq 
			\!\!\!\sup_{[0,T] \times \partial\!B_{\!R} }(\widetilde{u}-\widetilde{v}_{1})^{+} (t,x),
	\end{align}
since the right-hand side tends to zero as $R \rightarrow \infty$.\par
To this end, let us suppose that for some $R \,\textgreater\,  0$, there exists $(t_{0},x_{0}) \in [0,T] \times B_{R}$ such that
	\begin{equation}\label{(5.13)}
		\delta=\widetilde{u}(t_{0},x_{0})-\widetilde{v}_{1}(t_{0},x_{0}) 
		\,\textgreater 
		\!\!\!\sup_{[0,T] \times \partial\!{B}_{\!R} }(\widetilde{u}-\widetilde{v}_{1})^{+} (t,x)
		\geq 0,
	\end{equation}
and we will find a contradiction. Thus, for each $\alpha \,\textgreater\, 0 $, we can establish the following function $\Gamma_{\alpha}$:
	\begin{equation}\label{(5.14)}
		\Gamma_{\alpha}(t,x,y)
		=\widetilde{u}(t,x)-\widetilde{v}_{1}(t,y)-\frac{\alpha}{2}\vert x-y\vert^{2}.
	\end{equation}\par
Note that for each $\alpha \,\textgreater\, 0 $, $\Gamma_{\alpha}$ is a continuous function regarding the argument $t,\,x,\,y$, we can directly see that there exists $(t_{\alpha},x_{\alpha},y_{\alpha})$ in the compact set of $[0,T] \times \bar{B}_{R} \times \bar{B}_{R}$, such that $\Gamma_{\alpha}$ can reach its maximum value at this point. And also, since $\widetilde{u}$ is bounded, we can get that $\Gamma_{\alpha}(t_{0},x_{0},x_{0})=\delta \,\textgreater\,0$ and $\Gamma_{\alpha}(0,x,y)=-\infty$ hold for any $(x,y) \in \mathbb{R}^{n} \times \mathbb{R}^{n}$, which can imply $t_{0},t_{\alpha} \,\textgreater\, 0$. Given that  $\widetilde{u},\widetilde{v}$ are bounded, we have
\begin{align}  
	-\infty 
	\textless
	\,\Gamma_{\alpha}(t_{\alpha},x_{\alpha},y_{\alpha})= &\,\widetilde{u}(t_{\alpha},x_{\alpha})-\widetilde{v}_{1}(t_{\alpha},y_{\alpha})-\frac{\alpha}{2}\vert x_{\alpha}-y_{\alpha}\vert^{2} \notag\\\notag
=&\,\widetilde{u}(t_{\alpha},x_{\alpha})-\widetilde{v}(t_{\alpha},y_{\alpha})-\frac{\varepsilon}{t_{\alpha}}-\frac{\alpha}{2}\vert x_{\alpha}-y_{\alpha}\vert^{2}\notag\\\notag
\leq
&\,\widetilde{u}(t_{\alpha},x_{\alpha})-\widetilde{v}(t_{\alpha},y_{\alpha})\\
\textless
&\,C.\notag
\end{align}
Thus, motivated by Lemma 3.1 in \cite{CIP}, we get the following results:
	\begin{equation}\label{(5.15)}
		\left\{\begin{matrix} 
			\begin{aligned}  		
				&{}\hspace{0.40em}\mathrm{(i)}\ \mathrm{For}\ \alpha\ \mathrm{large}\ \mathrm{enough},\  (t_{\alpha},x_{\alpha},y_{\alpha}) \in (0,T) \times B_{R} \times B_{R}; \\ 				
				&\,\mathrm{(ii)}\,\lim\limits_{\alpha \rightarrow \infty}\alpha \vert x_{\alpha}-y_{\alpha} \vert^{2}=0,\   \lim\limits_{\alpha \rightarrow \infty}\vert x_{\alpha}-y_{\alpha} \vert^{2}=0;\\
				&\mathrm{(iii)}\ \widetilde{u}(t_{\alpha},x_{\alpha})
				\geq
				\widetilde{v}_{1}(t_{\alpha},y_{\alpha})+\delta;\\
				&\mathrm{(iv)}\,\lim\limits_{\alpha \rightarrow \infty} \widetilde{u}(t_{\alpha},x_{\alpha})-\widetilde{v}_{1}(t_{\alpha},y_{\alpha})-\frac{\alpha}{2}\vert x_{\alpha}-y_{\alpha}\vert^{2}
				=
				\!\!\!\!\sup_{[0,T] \times \bar{B}_{\!R} }(\widetilde{u}-\widetilde{v}_{1})(t,x)=(\widetilde{u}-\widetilde{v}_{1})(t^{\ast},x^{\ast});
			\end{aligned}
		\end{matrix}
		\right.
	\end{equation}
where the point $(t^{\ast},x^{\ast},x^{\ast})$ is the limit point of the sequence$(t_{\alpha},x_{\alpha},y_{\alpha}) $ as $\alpha \rightarrow \infty$. We note that
	\begin{align}
		(\widetilde{u}-\widetilde{v}_{1})(t^{\ast},x^{\ast})
		=&\!\!\!\!\sup_{[0,T] \times\bar{B}_{\!R} }(\widetilde{u}-\widetilde{v}_{1})(t,x)
		\geq
		(\widetilde{u}-\widetilde{v}_{1})(t_{0},x_{0})=\delta \,\textgreater\,0,\notag\\\notag
		(&\widetilde{u}-\widetilde{v}_{1})(0,x)=-\infty,\quad\quad \forall x \in \mathbb{R}^n.\notag
	\end{align}
which indicates $t^{\ast} \,\textgreater\,0$.\par
By Theorem 8.3 in \cite{CIP}, there exists $(p,X,Y) \in \mathbb{R} \times \mathbb{S}(n) \times \mathbb{S}(n)$, such that
		\begin{align}
			(p,\alpha(x_{\alpha}-y_{\alpha}),X) &\in \mathcal{P}^{2,+}\widetilde{u}(t_{\alpha},x_{\alpha}),\notag\\\notag
			(p,\alpha(x_{\alpha}-y_{\alpha}),Y) &\in \mathcal{P}^{2,-}\widetilde{v}_{1}(t_{\alpha},y_{\alpha}),
		\end{align}
and
	\begin{align}\label{(5.16)}
		\begin{pmatrix}
			X & \ 0\\
			0 & -Y
		\end{pmatrix}
		\leq 3\alpha
		\begin{pmatrix}
			\,\textbf{I}_{n} & -\textbf{I}_{n}\\
			-\textbf{I}_{n} & \,\textbf{I}_{n}
		\end{pmatrix}.
	\end{align}\par
Given that $\widetilde{v}$ is a viscosity supersolution, from (iii) in \eqref{(5.15)}, we get $\widetilde{u}(t_{\alpha},x_{\alpha}) \geq \widetilde{S}(t_{\alpha},y_{\alpha})$. In addition, since the obstacle $\widetilde{S}$ is uniformly continuous on compact sets, for $\alpha$ large enough, it is easy to check that $\widetilde{u}(t_{\alpha},x_{\alpha}) \,\textgreater\, \widetilde{S}(t_{\alpha},x_{\alpha})$. Hence since $\widetilde{u}$ is a viscosity subsolution, we have
	\begin{align}  		
			-p+\widetilde{F}(x_{\alpha},\widetilde{u}(t_{\alpha},x_{\alpha}),\alpha(x_{\alpha}-y_{\alpha}),X) 
			\leq 0,\notag
	\end{align}
and by \eqref{(5.11)}, we obtain
	\begin{align}  	\label{(5.17)}
		-p+\widetilde{F}(y_{\alpha},\widetilde{v}_{1}(t_{\alpha},y_{\alpha}),\alpha(x_{\alpha}-y_{\alpha}),Y) 
		\geq \frac{\varepsilon}{t^{2}}.
	\end{align}
Since $\widetilde{F}$ is ``proper'', and by (iii) in \eqref{(5.15)}, we can derive that
	\begin{equation}
		\begin{aligned}  		
			-p+\widetilde{F}(x_{\alpha},\widetilde{v}_{1}(t_{\alpha},y_{\alpha}),\alpha(x_{\alpha}-y_{\alpha}),X) 
			\leq 
			-p+\widetilde{F}(x_{\alpha},\widetilde{u}(t_{\alpha},x_{\alpha}),\alpha(x_{\alpha}-y_{\alpha}),X) \notag
			\leq
			0,
		\end{aligned}
	\end{equation}
which together with \eqref{(5.17)} implies that
	\begin{align} \label{(5.18)}
			\frac{\varepsilon}{t^{2}}
			\leq
			\widetilde{F}(y_{\alpha},\widetilde{v}_{1}(t_{\alpha},y_{\alpha}),\alpha(x_{\alpha}-y_{\alpha}),Y) 
			-\widetilde{F}(x_{\alpha},\widetilde{v}_{1}(t_{\alpha},y_{\alpha}),\alpha(x_{\alpha}-y_{\alpha}),X).
	\end{align}\par
Finally, we shall process the right side of \eqref{(5.18)}, by the definition of $\widetilde{F}$, it is easy to get
	\begin{equation}\label{(5.19)}
		\begin{aligned}  
			&\,\widetilde{F}(y_{\alpha},\widetilde{v}_{1}(t_{\alpha},y_{\alpha}),\alpha(x_{\alpha}-y_{\alpha}),Y) 
			-\widetilde{F}(x_{\alpha},\widetilde{v}_{1}(t_{\alpha},y_{\alpha}),\alpha(x_{\alpha}-y_{\alpha}),X)\\
			=
			&\,\bar{F}^{\ast}(x_{\alpha},\widetilde{v}_{1}(t_{\alpha},y_{\alpha}),\alpha(x_{\alpha}-y_{\alpha}),X) 
			-\bar{F}^{\ast}(y_{\alpha},\widetilde{v}_{1}(t_{\alpha},y_{\alpha}),\alpha(x_{\alpha}-y_{\alpha}),Y)\\
			=
			&\,e^{\lambda t}\xi^{-1}(x_{\alpha})\bar{F}(x_{\alpha},\beta_{1},p_{1},A_{1}) 
			-e^{\lambda t}\xi^{-1}(y_{\alpha})\bar{F}(y_{\alpha},\beta_{2},p_{2},A_{2}),
		\end{aligned}
	\end{equation}
where the above coefficients are given by
	\begin{align}  
		\beta_{1}=
\ &\widetilde{v}_{1}(t_{\alpha},y_{\alpha})e^{-\lambda t}\xi(x_{\alpha}),\ 
			\beta_{2}=\widetilde{v}_{1}(t_{\alpha},y_{\alpha})e^{-\lambda t}\xi(y_{\alpha}),\notag\\\notag
			p_{1}=\ &e^{-\lambda t}\xi(x_{\alpha})\alpha(x_{\alpha}-y_{\alpha})+e^{-\lambda t}\widetilde{v}_{1}(t_{\alpha},y_{\alpha})D\xi(x_{\alpha}),\\\notag
			p_{2}=\ &e^{-\lambda t}\xi(y_{\alpha})\alpha(x_{\alpha}-y_{\alpha})+e^{-\lambda t}\widetilde{v}_{1}(t_{\alpha},y_{\alpha})D\xi(y_{\alpha}),\\\notag
			A_{1}=\ &e^{-\lambda t}\xi(x_{\alpha})X+e^{-\lambda t}(D\xi(x_{\alpha}) \otimes \alpha(x_{\alpha}-y_{\alpha}))+e^{-\lambda t}(\alpha(x_{\alpha}-y_{\alpha}) \otimes D\xi(x_{\alpha}))
			+e^{-\lambda t}\widetilde{v}_{1}(t_{\alpha},y_{\alpha})D^{2}\xi(x_{\alpha}),\\\notag
			A_{2}=\ &e^{-\lambda t}\xi(y_{\alpha})Y+e^{-\lambda t}(D\xi(y_{\alpha}) \otimes \alpha(x_{\alpha}-y_{\alpha}))+e^{-\lambda t}(\alpha(x_{\alpha}-y_{\alpha}) \otimes D\xi(y_{\alpha}))
			+e^{-\lambda t}\widetilde{v}_{1}(t_{\alpha},y_{\alpha})D^{2}\xi(y_{\alpha}).\notag
		\end{align}
Thus, by assumptions (H1)-(H3) and (H5), we have
\begin{align} \label{(5.20)}
			&\,\widetilde{F}(y_{\alpha},\widetilde{v}_{1}(t_{\alpha},y_{\alpha}),\alpha(x_{\alpha}-y_{\alpha}),Y) 
			-\widetilde{F}(x_{\alpha},\widetilde{v}_{1}(t_{\alpha},y_{\alpha}),\alpha(x_{\alpha}-y_{\alpha}),X)\notag\\\notag
			\leq
			&\limsup_{s \rightarrow \infty} \frac{1}{s}\int_{0}^{s}\big[e^{\lambda t}\xi^{-1}(x_{\alpha})p_{1}b(r,x_{\alpha})-e^{\lambda t}\xi^{-1}(y_{\alpha})p_{2}b(r,y_{\alpha})\big]dr\\\notag
			&+\limsup_{s \rightarrow \infty} \frac{1}{s}\int_{0}^{s}\big[e^{\lambda t}\xi^{-1}(x_{\alpha})f(r,x_{\alpha},\beta_{1},p_{1}\sigma(r,x_{\alpha}))-e^{\lambda t}\xi^{-1}(y_{\alpha})f(r,y_{\alpha},\beta_{2},p_{2}\sigma(r,y_{\alpha}))\big]dr\\\notag
			&+\limsup_{s \rightarrow \infty} \frac{1}{s}\int_{0}^{s}\big[e^{\lambda t}\xi^{-1}(x_{\alpha})G(H(r,x_{\alpha},\beta_{1},p_{1},A_{1}))-e^{\lambda t}\xi^{-1}(y_{\alpha})G(H(r,y_{\alpha},\beta_{2},p_{2},A_{2}))\big]dr\\
			\coloneqq
			&\,U_{1}+U_{2}+U_{3}.
\end{align}
As for $U_{1}$, from assumption (H1) and the properties of the function $\eta$, we derive that
	\begin{align}  
			U_{1}
			=&\limsup_{s \rightarrow \infty} \frac{1}{s}\int_{0}^{s}\big[e^{\lambda t}\xi^{-1}(x_{\alpha})p_{1}b(r,x_{\alpha})-e^{\lambda t}\xi^{-1}(y_{\alpha})p_{2}b(r,y_{\alpha})\big]dr\notag\\
			\leq
			&\limsup_{s \rightarrow \infty} \frac{1}{s}\int_{0}^{s}
			\alpha\vert x_{\alpha}-y_{\alpha}\vert\vert b(r,x_{\alpha})-b(r,y_{\alpha})\vert dr\notag\\
			&+\limsup_{s \rightarrow \infty} \frac{1}{s}\int_{0}^{s}
			\big\vert \widetilde{v}_{1}(t_{\alpha},y_{\alpha})\xi^{-1}(x_{\alpha})D\xi(x_{\alpha})b(r,x_{\alpha})
			-\widetilde{v}_{1}(t_{\alpha},y_{\alpha})\xi^{-1}(y_{\alpha})D\xi(y_{\alpha})b(r,y_{\alpha})\big\vert dr\notag\\
			\leq
			&\,C(L,k)(1+\vert x_{\alpha}\vert)(\vert \widetilde{v}_{1}(t_{\alpha},y_{\alpha}) \vert \vert x_{\alpha}-y_{\alpha} \vert+\alpha\vert x_{\alpha}-y_{\alpha} \vert^{2}).\notag
	\end{align}
As for $U_{2}$, according to assumptions (H1)-(H2) and \eqref{(5.2)}, it is easy to check that
	\begin{align}  
			\hspace{-4em}U_{2}
			=
			&\limsup_{s \rightarrow \infty} \frac{1}{s}\int_{0}^{s}\big[e^{\lambda t}\xi^{-1}(x_{\alpha})f(r,x_{\alpha},\beta_{1},p_{1}\sigma(r,x_{\alpha}))-e^{\lambda t}\xi^{-1}(y_{\alpha})f(r,y_{\alpha},\beta_{2},p_{2}\sigma(r,y_{\alpha}))\big]dr\notag\\\notag
			\leq
			&\limsup_{s \rightarrow \infty} \frac{1}{s}\int_{0}^{s}
			\big\vert e^{\lambda t}\xi^{-1}(x_{\alpha}) -e^{\lambda t}\xi^{-1}(y_{\alpha}) \big\vert \big\vert f(r,x_{\alpha},\beta_{1},p_{1}\sigma(r,x_{\alpha}))\big\vert dr\\\notag        
			&+\limsup_{s \rightarrow \infty} \frac{1}{s}\int_{0}^{s}
			\big\vert e^{\lambda t}\xi^{-1}(y_{\alpha})\big\vert \big\vert f(r,x_{\alpha},\beta_{1},p_{1}\sigma(r,x_{\alpha}))-f(r,y_{\alpha},\beta_{2},p_{2}\sigma(r,y_{\alpha}))\big\vert dr\notag\\\notag
			\leq
			&\,C(L,T)\big(1+\vert x_{\alpha}\vert^{(k+2)\vee (m+1)}+\vert y_{\alpha}\vert^{(k+2)\vee (m+1)}\big)\big[(1+\vert \widetilde{v}_{1}(t_{\alpha},y_{\alpha}) \vert) \vert x_{\alpha}-y_{\alpha} \vert+\alpha\vert x_{\alpha}-y_{\alpha} \vert^{2}\big].
		\end{align}
As for $U_{3}$, by the sublinear property of $G$ and similar discussions as before, we get
	\begin{align}
			\!\!\!U_{3}
			=&
			\limsup_{s \rightarrow \infty} \frac{1}{s}\int_{0}^{s}\big[e^{\lambda t}\xi^{-1}(x_{\alpha})G(H(r,x_{\alpha},\beta_{1},p_{1},A_{1}))-e^{\lambda t}\xi^{-1}(y_{\alpha})G(H(r,y_{\alpha},\beta_{2},p_{2},A_{2}))\big]dr\notag\\\notag
			\leq  
			&\limsup_{s \rightarrow \infty} \frac{1}{s}\int_{0}^{s}
			\!G\big(\sigma(r,x_{\alpha}\mathbf{)}^\top \!X\sigma(r,x_{\alpha})-\sigma(r,y_{\alpha}\mathbf{)}^\top \!Y\sigma(r,y_{\alpha})\big)dr\notag \\ \notag 
			&\!+\limsup_{s \rightarrow \infty} \frac{1}{s}\!\int_{0}^{s}
			\!G\big(\sigma(r,x_{\alpha}\mathbf{)}^\top \!\widetilde{v}_{1}(t_{\alpha},y_{\alpha})\kappa(x_{\alpha})\sigma(r,x_{\alpha})-\sigma(r,y_{\alpha}\mathbf{)}^\top\widetilde{v}_{1}(t_{\alpha},y_{\alpha})\kappa(y_{\alpha})\sigma(r,y_{\alpha})\big)dr \notag\\  \notag             
			&\!+\limsup_{s \rightarrow \infty} \frac{1}{s}\!\int_{0}^{s}
			\!G\big(\sigma(r,x_{\alpha}\mathbf{)}^\top\!(\alpha(x_{\alpha}-y_{\alpha})  \otimes \eta(x_{\alpha}))\sigma(r,x_{\alpha})
			-\sigma(r,y_{\alpha}\mathbf{)}^\top\!(\alpha(x_{\alpha}-y_{\alpha})  \otimes  \eta(y_{\alpha}))\sigma(r,y_{\alpha})\big)dr \notag
		\end{align}
	\begin{align}
			&+\limsup_{s \rightarrow \infty} \frac{1}{s}\!\int_{0}^{s}
			\!G\big(\sigma(r,x_{\alpha}\mathbf{)}^\top\!(\eta(x_{\alpha}) \otimes \alpha(x_{\alpha}-y_{\alpha}))\sigma(r,x_{\alpha})
			-\sigma(r,y_{\alpha}\mathbf{)}^\top\!(\eta(y_{\alpha}) \otimes \alpha(x_{\alpha}-y_{\alpha}))\sigma(r,y_{\alpha})\big)dr\notag\\ \notag
			\coloneqq
			&\,V_{1}+V_{2}+V_{3}+V_{4}.
		\end{align}
As for $V_{1}$, by \eqref{(5.16)} and $G(X)\leq \frac{1}{2}\bar{\sigma}^{2}tr[X]$ for any $0 \leq X \in \mathbb{S}(d)$, we derive that
	\begin{equation}
		\begin{aligned}  
			V_{1}
			=
			&\limsup_{s \rightarrow \infty} \frac{1}{s}\int_{0}^{s}
			G\big(\sigma(r,x_{\alpha}\mathbf{)}^\top X\sigma(r,x_{\alpha})-\sigma(r,y_{\alpha}\mathbf{)}^\top Y\sigma(r,y_{\alpha})\big)dr \\	
			\leq
			&\limsup_{s \rightarrow \infty} \frac{1}{s}\int_{0}^{s}
			G\big(3\alpha (\sigma(r,x_{\alpha})-\sigma(r,y_{\alpha})\mathbf{)}^\top(\sigma(r,x_{\alpha})-\sigma(r,y_{\alpha}))\big)dr \\	\notag
			\leq
			&\,C(L,\bar{\sigma})\alpha\vert x_{\alpha}-y_{\alpha}\vert^{2}.
		\end{aligned}
	\end{equation}
Likewise, as for $V_{2}$ and $V_{3}$, from assumptions (H1)-(H2) and the properties of $\kappa$ and $G$, we get
	\begin{equation}
		\begin{aligned}  
			V_{2}
			=
			&\limsup_{s \rightarrow \infty} \frac{1}{s}\int_{0}^{s}
			G\big(\sigma(r,x_{\alpha}\mathbf{)}^\top\widetilde{v}_{1}(t_{\alpha},y_{\alpha})\kappa(x_{\alpha})\sigma(r,x_{\alpha})-\sigma(r,y_{\alpha}\mathbf{)}^\top\widetilde{v}_{1}(t_{\alpha},y_{\alpha})\kappa(y_{\alpha})\sigma(r,y_{\alpha})\big)dr \\  
			\leq
			&\limsup_{s \rightarrow \infty} \frac{1}{s}\int_{0}^{s}
			\frac{1}{2}\bar{\sigma}^{2}d\big\vert \sigma(r,x_{\alpha}\mathbf{)}^\top\widetilde{v}_{1}(t_{\alpha},y_{\alpha})\kappa(x_{\alpha})\sigma(r,x_{\alpha})-\sigma(r,y_{\alpha}\mathbf{)}^\top\widetilde{v}_{1}(t_{\alpha},y_{\alpha})\kappa(y_{\alpha})\sigma(r,y_{\alpha}) \big\vert dr \\	
			\leq
			&\,C(L,\bar{\sigma},d)\vert \widetilde{v}_{1}(t_{\alpha},y_{\alpha}) \vert   \vert x_{\alpha}-y_{\alpha}\vert(1+\vert x_{\alpha}\vert^{2}+\vert y_{\alpha}\vert^{2}),\notag
		\end{aligned}
	\end{equation}
and
	\begin{align}  
			\!\!\!\!\!V_{3}
			=
			&\limsup_{s \rightarrow \infty} \!\frac{1}{s}\!\int_{0}^{s}
			\!G\big(\sigma(r,x_{\alpha}\mathbf{)}^\top\!(\alpha(x_{\alpha}-y_{\alpha})  \otimes \eta(x_{\alpha}))\sigma(r,x_{\alpha})
			-\sigma(r,y_{\alpha}\mathbf{)}^\top\!(\alpha(x_{\alpha}-y_{\alpha})  \otimes  \eta(y_{\alpha}))\sigma(r,y_{\alpha})\big)dr\notag\\\notag 
			\leq
			&\limsup_{s \rightarrow \infty} \frac{1}{s}\!\int_{0}^{s}
			\!\frac{\bar{\sigma}^{2}d}{2} \big\vert \sigma(r,x_{\alpha}\mathbf{)}^\top\!(\alpha(x_{\alpha}-y_{\alpha})  \otimes \eta(x_{\alpha}))\sigma(r,x_{\alpha})
			-\sigma(r,y_{\alpha}\mathbf{)}^\top\!(\alpha(x_{\alpha}-y_{\alpha})  \otimes  \eta(y_{\alpha}))\sigma(r,y_{\alpha})\big\vert dr\\\notag
			\leq
			&\,C(L,\bar{\sigma},d)\alpha\vert x_{\alpha}-y_{\alpha}\vert^{2}(1+\vert x_{\alpha}\vert^{2}+\vert y_{\alpha}\vert^{2}).
		\end{align}
Similarly, it is easy to check that
\begin{align}  
			V_{4}		
			\leq
			&\,C(L,\bar{\sigma},d)\alpha\vert x_{\alpha}-y_{\alpha}\vert^{2}(1+\vert x_{\alpha}\vert^{2}+\vert y_{\alpha}\vert^{2}).\notag
\end{align}
Thus, we get
\begin{align}  
	U_{3}		
	\leq
	&\,C(L,\bar{\sigma},d)(1+\vert x_{\alpha}\vert^{2}+\vert y_{\alpha}\vert^{2})(\vert \widetilde{v}_{1}(t_{\alpha},y_{\alpha}) \vert\vert x_{\alpha}-y_{\alpha}\vert+\alpha\vert x_{\alpha}-y_{\alpha}\vert^{2}),\notag
\end{align}
and consequently we derive that
	\begin{equation}
		\begin{aligned}   
			&\,\widetilde{F}(y_{\alpha},\widetilde{v}_{1}(t_{\alpha},y_{\alpha}),\alpha(x_{\alpha}-y_{\alpha}),Y) 
			-\widetilde{F}(x_{\alpha},\widetilde{v}_{1}(t_{\alpha},y_{\alpha}),\alpha(x_{\alpha}-y_{\alpha}),X)\\\notag
			\leq
			&\,U_1+U_2+U_3\\
			\leq
			&\,C(L,\bar{\sigma},d)\big(1+\vert x_{\alpha}\vert^{(m+1)\vee (k+2)}+\vert y_{\alpha}\vert^{(m+1)\vee (k+2)}\big)\big[(\vert \widetilde{v}_{1}(t_{\alpha},y_{\alpha}) \vert+1)\vert x_{\alpha}-y_{\alpha}\vert+\alpha\vert x_{\alpha}-y_{\alpha}\vert^{2}\big],
		\end{aligned}
	\end{equation}
which together with \eqref{(5.18)} can indicate that
		\begin{align}  \label{(5.21)}
			\frac{\varepsilon}{t^2}
			\leq
			&\,C(L,\bar{\sigma},d)\big(1+\vert x_{\alpha}\vert^{(m+1)\vee (k+2)}+\vert y_{\alpha}\vert^{(m+1)\vee (k+2)}\big)\big[(\vert \widetilde{v}_{1}(t_{\alpha},y_{\alpha}) \vert+1)\vert x_{\alpha}-y_{\alpha}\vert+\alpha\vert x_{\alpha}-y_{\alpha}\vert^{2}\big].
		\end{align}
And also, since $t_{\alpha},\,t^{\ast} \,\textgreater\, 0$ and $\widetilde{v}$ are uniformly bounded, we can get that $\widetilde{v}_{1}(t_{\alpha},y_{\alpha}) \,\textless +\infty$ and $\widetilde{v}_{1}(t^{\ast},x^{\ast}) \,\textless +\!\infty$ hold. With the help of \eqref{(5.15)}, the right side of \eqref{(5.21)} tends to zero as $\alpha \rightarrow \infty$, which induces a contradiction. To be precise, we get $\widetilde{u}(t,x) \leq \widetilde{v}_{1}(t,x)$ for any $(t,x) \in [0,T] \times \mathbb{R}^n$, which implies that $u \leq v$ by using \eqref{(5.10)} and the continuity of $u,v$. Thus, the averaged PDE \eqref{(4.1)} has a unique viscosity solution.
\end{proof}\par

\end{document}